%% file: lopsp.tex
\documentclass[preprint,10pt]{elsarticle}
\usepackage[utf8]{inputenc}
\usepackage[english]{babel}
\usepackage[dvipsnames]{xcolor}
\usepackage{amsmath}
\usepackage{amsfonts}
\usepackage{amsthm}
\usepackage{amssymb}
\usepackage{enumerate}
\usepackage{tikz}
\usepackage{graphicx}
\usepackage{calrsfs}
\usepackage{geometry}
\usetikzlibrary{matrix, positioning}

\newtheorem{theorem}{Theorem}
\newtheorem{definition}{Definition}
\newtheorem{lemma}[theorem]{Lemma}
\newtheorem{corollary}[theorem]{Corollary}

\newcommand{\invpi}{\pi_P^{-1}}

\begin{document}
\begin{frontmatter}

\author{Gunnar Brinkmann}
\ead{Gunnar.Brinkmann@UGent.be}

\author{Heidi Van den Camp \corref{cor1}}
\ead{Heidi.VandenCamp@UGent.be}


\cortext[cor1]{Corresponding author \\ Heidi Van den Camp is a PhD fellow at Ghent University on the BOF (Special Research Fund) scholarship 01D00320 }

\address{Applied Mathematics, Computer Science, and Statistics, \\ Ghent University,\\ 
Krijgslaan 281 S9, 9000 Ghent, Belgium}

\title{On Local Operations that Preserve Symmetries \\ and on Preserving Polyhedrality of Embeddings}
\date{ }

\begin{abstract}

        We prove that local operations (as defined in \cite{gocox}) that preserve all symmetries, as e.g.\ {\em dual, truncation, ambo}, or {\em join,}, as well as local operations
        that preserve all symmetries except orientation reversing ones, as e.g.\ {\em gyro} or {\em snub}, preserve the polyhedrality of simple embedded graphs. This generalizes a
        result by Mohar proving this for the operation {\em dual} in \cite{moh}. We give the proof based on an abstract characterization of these operations, prove that the operations are well
        defined, and also demonstrate the close connection between these operations and Delaney-Dress symbols. We also discuss more general operations not coming from 3-connected
        simple tilings of the plane.

\end{abstract}


\end{frontmatter}
        
	\section{Introduction}	
		Symmetry-preserving operations on polyhedra have been studied since a very long time. They were probably first applied in ancient Greece. Some of the Archimedean
                solids can be obtained from Platonic solids by applying the operation which was later called `{\em truncation}' by Kepler. Over the centuries, polyhedra and specific
                operations on them have been studied extensively. However, a general definition of the concept `{\em local symmetry-preserving operation}' and a systematic way of describing such operations
                was only presented in 2017 \cite{gocox}. This description covers a large class of operations on embedded graphs, including all well-known
                symmetry-preserving operations such as truncation, the dual, or those operations known as achiral Goldberg-Coxeter operations \cite{CK62}. Goldberg-Coxeter operations
                were in fact introduced by Caspar and Klug and can be used to construct all fullerenes or certain viruses with icosahedral symmetry.
                The new approach in \cite{gocox} allows to tackle various problems from a more abstract perspective, and also allows to prove
                general theorems about the whole class of operations instead of considering each operation separately.
		
		In addition to these local symmetry-preserving operations (lsp-operations), which preserve all the symmetries of an embedded graph, there are also operations that
                are only guaranteed to preserve the orientation-preserving symmetries. Well-known examples of such operations are `{\em snub}' and `{\em gyro}' \cite{randic2000bridges}, or the
                chiral Goldberg-Coxeter operations. In \cite{gocox}, also a general description of such
                `{\em local orientation-preserving symmetry preserving operations}' (lopsp-operations) was presented.

                As the description in \cite{gocox} was aimed at a broader audience than just mathematicians, the approach was described by the more intuitive way of {\em cutting out
                  certain parts of simple periodic 3-connected tilings}.  In this article, we will give the more direct definition based on Delaney-Dress symbols that forms the base of this approach
                and show the connection to the original description.  We will also show that for every lsp-operation there is an equivalent lopsp-operation, i.e. a lopsp-operation
                that has the same result as the lsp-operation when applied to an embedded graph.
		
		In \cite{gocox}, it is proved that the result of applying an lsp-operation to a polyhedron (that is: a simple 3-connected graph embedded in the plane) is also a
                polyhedron. In \cite{lopsp2020} this result is also announced for all lopsp-operations. We will refer to some modified concepts in that paper, but due to some serious problems
                in that paper not use the results given there.
		
		Originally, lsp- as well as lopsp-operations were only defined for simple plane graphs because of their origin in the study of
                polyhedra. However, there is no mathematical reason why these definitions should not be applied to multigraphs and embeddings of higher genus. The question then arises in how far we
                can extend the theorem for 3-connected simple plane graphs to 3-connected embedded graphs of higher genus.
		
		In general, lopsp-operations do not necessarily preserve 3-connectedness for embeddings that are not plane. While this is obvious for multigraphs with faces of size
                1 or 2, it is even true for simple graphs and also requiring the result to be simple. 
                The most striking example of a local symmetry-preserving operation that can turn 3-connected graphs into (even simple)
                graphs with lower connectivity is probably the dual. In \cite{dualconnectivity} it is proven that for any $k\geq 1$, there exist embeddings of
                $k$-connected simple graphs so that the dual is simple and has a 1-cut.
		
		However, even the dual always preserves 3-connectedness in embedded simple graphs of face-width at least three, as proven in \cite{moh}. After defining the notion
                of $ck$-embedded in Definition~\ref{def:ckembed} and 
                of $ck$-operations in Definition \ref{def:lopsp_connectivity}, we will prove the more general Theorem \ref{thm:main} from which the following key result, of which
                the result for the dual is a special case, follows immediately. Here $O(G)$ is the result of applying the operation $O$ to the embedded graph $G$ and all known and well studied
                operations are in the class of $c3$-operations:

		\begin{corollary}\label{cor:k-conn_preserved}
			Let $k\in \{1,2,3\}$. If $G$ is a $ck$-embedded graph, and $O$ is a $ck$-lsp- or $ck$-lopsp-operation, then $O(G)$ is also
                        $ck$-embedded.
		\end{corollary}

		This theorem is most interesting and relevant for $k=3$. This has two reasons. Firstly, the set of $c3$-operations contains all well-known and often used operations
                for which numerous results have already been proven. Lsp-operations that are not $c3$-lsp-operations were not even included in the original definition of
                lsp-operations \cite{gocox}. Secondly, $c3$-embedded graphs, which are in fact simple embedded 3-connected graphs of face-width at least three, have some very
                interesting properties, and are also known as polyhedral embeddings \cite{moh} (Proposition 3.9). They can be defined equivalently as simple embedded graphs where
                every facial walk is a simple cycle and two faces are either disjoint or their intersection consists of only one vertex or one edge. As the name suggests,
                polyhedral embeddings are a generalisation of polyhedra to surfaces of higher genus. It turns out that the key concept in this research is not 3-connectedness but
                polyhedrality, a property that is equivalent to being simple and 3-connected in the plane, but only in the plane. The main result of this article follows immediately from
                Corollary \ref{cor:k-conn_preserved}.
		
		\begin{theorem}\label{thm:pol_lopsp}
		  If $G$ is a polyhedral embedding of a graph and $O$ is a $c3$-lsp- or $c3$-lopsp-operation, then $O(G)$ is also a polyhedral embedding.
		\end{theorem}

		\section{Basic definitions}
                
		Embedded graphs are often studied as topological embeddings into surfaces. However, in this text we will use a combinatorial description. We will only consider
                embeddings in orientable surfaces. Though we are mainly interested in simple graphs, graphs with multiple edges and loops will occur in a natural way -- e.g. as tools
                or as the result of an operation -- so that
                we will in general assume that a graph can have multiple edges and loops and explicitly restrict the class where necessary. For simplicity we will also refer to
                the {\em sets} of vertices and edges, though the edges can of course form a multiset.
	
                With every edge $e$ with different endpoints $v,w$ of a graph $G$, we associate two oriented edges, $e_v$ starting at $v$ and $e_v^{-1}=e_w$ starting at $w$.
                In case of a loop, that is $v=w$, we associate two oriented edges, $e_v, (e_v)'$ starting at $v$ and write $e_v^{-1}=(e_v)'$ and $(e_v)'^{-1}=e_v$ 
                If $G$ is simple or if there is no danger of misunderstandings, we denote these edges as $e_v=(v,w)$ and $e_w=(w,v)=(v,w)^{-1}$. For every vertex $v$ of $G$, a cyclic order
                is assigned to all oriented edges $e_v$. This way, all oriented edges $e_v$ (including the edges $(e_v)'$) have a `successor' $\sigma_G(e_v)$. An \emph{embedded graph} is a
                connected graph together with this successor
                function $\sigma_G()$. To keep notation simple we will use the same notation for the abstract graph $G$ and the embedded graph $G$. It will always be clear from the
                context whether a graph is embedded or not, and how it is embedded. When drawing combinatorial embeddings, the cyclic order around the vertices corresponds to the
                clockwise order of edges around that vertex in the drawing.
			
		Consider three oriented edges $e_1, e_2$, and $e_3$ incident with a vertex $v$. We say that $e_2$ is \emph{between} $e_1$ and $e_3$ if $e_1,e_2,e_3$ occur in this
                order as part of the cyclic order around $v$.

		We say that $e$ and $\sigma_G(e^{-1})$ form an \emph{angle}. A \emph{face} is a cyclic sequence of oriented edges such that every two
                consecutive edges form an angle. We will use the term \emph{facial walk} to refer to the closed walk in $G$ corresponding to this cyclic sequence of oriented
                edges. The \emph{boundary} of $f$ is the subgraph of $G$ consisting of all the vertices and edges in the facial walk of $f$.
			
		For this embedded graph $G$, with $V_G$, $E_G$ and $F_G$ denoting the sets of vertices, edges and faces of $G$ respectively, $\chi(G)= |V_G| - |E_G| +
                |F_G|$\index{$\chi(G)$} is the \emph{Euler characteristic}\index{Euler characteristic} of $G$.
			
		The \emph{genus} of $G$ is now defined as $gen(G)= \frac{2-\chi(G)}{2}$\index{$gen(G)$}. If an embedded graph has genus 0, it is called \emph{plane}.
	
		\begin{definition}\label{def:kembed}
                  
		Let $G$ be an embedded graph and $G'$ be a subgraph of $G$. The graph $G'$ with the embedding induced by that of $G$ is called an \emph{embedded subgraph}\index{subgraph!embedded}.
		
		We will now define what a \emph{bridge}\index{bridge} for $G'$ in $G$ is. There are two kinds of bridges:
		\begin{itemize}
			\item If $e \in E_G\setminus E_{G'}$ is an edge with endpoints $v,w \in V_{G'}$, then the graph with vertex set $\{v,w\}$ and edge set $\{e\}$ is a bridge.
			\item Let $C$ be a component of the subgraph of $G$ induced by the vertices of $G$ that are not in $G'$, and define $E_C'=\{e \in E_G \mid e \cap V_C \neq
                          \emptyset \}$ and $V_C' = \{v\in V'\mid \exists e\in E_C' : v\in e\}$. Then the graph with vertex set $V_C'$ and edge set $E_C'$ is a bridge.
		\end{itemize}

		If a bridge has an edge that is between two edges $e$ and $e'$ so that $e^{-1}$ and $e'$ form an angle in a face of $G'$, then the bridge \emph{is in
                  that face}. All the vertices and edges of the bridge are also said to be \emph{in the face}. A vertex or edge of $G$ is in the \emph{interior} of a face of $G'$
                if it is in that face and it is not in the boundary.
	
		If a bridge is in more than one face, we say that those faces are \emph{bridged}\index{face!bridged}. A face that is not bridged is called \emph{simple}\index{face!simple}.
		
		Let $f$ be a simple face of $G'$. We will define the \emph{internal component}\index{internal component} of $f$ as follows. Start with the graph $H$ that consists of
                the boundary of $f$ together with all bridges in $f$. Then replace every vertex $v$ of $H$ that appears $k>1$ times in the facial walk of $f$ by $k$ pairwise
                different vertices $v_1,...,v_k$.  If both oriented edges associated with an edge of $G'$ appear in the facial walk, this edge is also split into two different edges
                between different copies of its vertices.  Let $(x,v)$ and $(v,y)$ be the oriented edges that form the angle (in $G'$) at the $i$-th occurrence of $v$. Then we
                define the rotational order (and also the neighbours) of $v_i$ to be the same as the rotational order around $v$ (in $G$) but restricted to the edges between
                $(v,x)$ and $(v,y)$ and of course some vertices possibly replaced by their copies.
                The result of this is the internal component $IC(f)$\index{$IC(f)$} of $f$. An example of an internal component is illustrated in Figure
                \ref{fig:internal_component}. If $IC(f)$ is plane, we call $f$ \emph{internally plane}\index{face!internally plane}.
		
		\begin{figure}	
				\centering
				\input{NoBridges.tikz}
				\qquad
				\input{InternalComponent.tikz}

			        \caption{\label{fig:internal_component}A graph and the internal component of its face with bridges.}
		\end{figure}
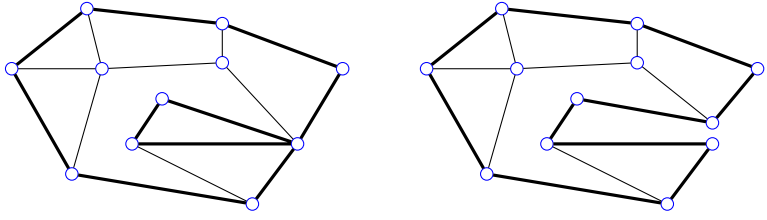
	\end{definition}

	A very important concept in the definition of lsp- and lopsp-operations is the barycentric subdivision of an embedded graph. It is obtained by subdividing every face into triangular faces, which we
        will call chambers. We will also use the barycentric subdivision to define contractible cycles and face-width in a combinatorial way.
	
	\begin{definition}		
		The \emph{barycentric subdivision}\index{barycentric subdivision} $B_G$ of an embedded graph $G$ is an embedded graph that has a unique vertex for every vertex, for
                every edge and for every face of $G$. We say that these vertices are of type 0, 1, and 2 respectively. There are edges between vertices of type 0 and 1 if the
                corresponding vertex and edge are incident. There are edges between vertices of type 0 or 1 and type 2 if the corresponding vertex or edge appears in the boundary of the
                corresponding face. There are no edges between vertices of the same type. For $i\in \{0,1,2\}$, an edge is \emph{of type $i$} if it is not incident with a vertex of
                type $i$. We will also refer to vertices and edges of type $i$ as $i$-vertices and $i$-edges. The rotational order around the vertices of type 2 is the inverse of
                the order of vertices and edges in the corresponding face. Around the vertices of type 0 the rotational order follows the alternating sequence
                of edges and angles (representing the faces they are contained in) around the vertex and analogously for vertices of type 1.
                This is illustrated in Figure \ref{fig:barycentric}.
                With this embedding, a short calculation of the Euler
                characteristic shows that $gen(B_G)=gen(G)$. To keep notation simple, we will often use the same names for the faces, edges and vertices of $G$ and their
                corresponding vertices in $B_G$.
		
		\begin{figure}
			\centering
			\input{Hexagon.tikz}
			\quad
			\input{BaryHexagon2.tikz}
			\caption{\label{fig:barycentric}A face in an embedded graph $G$ and the corresponding part of $B_G$. Types are represented by colours in the order rgb, that is: a red vertex or
                          edge is type 0, a green one is type 1 and a black one is type 2.}
		\end{figure}
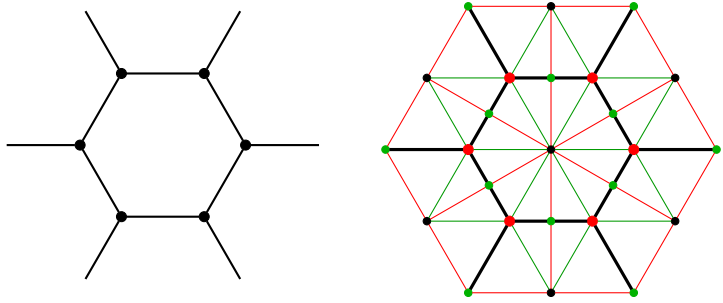
		
		Every face of $B_G$ is a triangle, with exactly one vertex and one edge of each type. We call such a triangle a \emph{chamber}\index{chamber}.
                Let $\Sigma=\left\langle\sigma_0, \sigma_1, \sigma_2 \mid \sigma_i^2=1\right\rangle$ denote the free Coxeter group with generators $\sigma_0$, $\sigma_1$ and $\sigma_2$.
                For $0\le i \le 2$ we define $C\sigma_i=C'$ if the chambers $C$ and $C'$ share their $i$-edges. This way $B_G$ defines a chamber system that we also denote as $B_G$.
		
		The \emph{double chamber system} $D_G$ of a graph $G$ is the subgraph of $B_G$ that only contains the edges of types 1 and 2.
		A \emph{double chamber} of a graph $G$ is a face in $D_G$. Every double chamber has length four: two (in case of no loops different) vertices of type 0, one of type 1, and one of type 2.
	\end{definition}

        \section{Delaney-Dress symbols and a combinatorial description of lsp- and lopsp-operations}\label{sec:delaney-dress}

        In a series of papers, Andreas Dress (in later papers together with coauthors) developed a finite symbol encoding the topology as well as the symmetry of periodic tilings. He attributed the idea to Matthew Delaney
        and called these symbols {\em Delaney symbols}. In later papers by other authors, these symbols are called {\em Delaney-Dress symbols}. In \cite{olaf_ddsymbols} and \cite{DressHuson87} Delaney-Dress
        symbols of periodic tilings of the Euclidean plane and the hyperbolic plane are characterized.

        As a topological definition of tilings falls outside the scope of this article, we will directly start with the combinatorial characterization in \cite{olaf_ddsymbols},\cite{DressHuson87}. We will sketch the connection
        to tilings, but for a detailed description we refer the reader to \cite{olaf_ddsymbols} or \cite{DressHuson87}.

\begin{definition}\label{def:dd_symbol}
	
	Let $\mathcal{D}$ be a set together with an action  (from the right) of $\Sigma$ on $\mathcal{D}$, and for $(i,j)\in \{(0,1),(0,2),(1,2)\}$ let $m_{ij} : \mathcal{D}\rightarrow \mathbb{N}$ be maps with
        $m_{02}(C)=2$ for all $C\in \mathcal{D}$. The tuple $(\mathcal{D}, m_{01}, m_{02}, m_{12})$ is a \emph{Delaney-Dress symbol} of the Euclidean plane if the following properties hold:
	
	\begin{enumerate}
		\item $\mathcal{D}$ has finitely many elements
		\item $\Sigma$ acts transitively on $\mathcal{D}$
		\item For $i,j\in \{0,1,2\}, i<j$, $m_{ij}()$ is constant on $\langle \sigma_i,\sigma_j\rangle$-orbits and $C(\sigma_i\sigma_j)^{m_{ij}(C)}=C$ for all $C\in \mathcal{D}$
		\item For the {\em curvature} $\mathcal{C}(\mathcal{D}, m_{01}, m_{02}, m_{12})$ we have\\
                  $\displaystyle	
		\mathcal{C}(\mathcal{D}, m_{01}, m_{02}, m_{12})=\sum\limits_{C\in\mathcal{D}}\left(\frac{1}{m_{01}(C)} + \frac{1}{m_{12}(C)} - \frac{1}{m_{02}(C)}\right)=0$
	\end{enumerate}

\end{definition}

Such Delaney-Dress symbols encode the combinatorial structure of periodic tilings of the Euclidean plane together with a symmetry group acting on the tiling.
If $\mathcal{Cv}(\mathcal{D}, m_{01}, m_{02}, m_{12})\not=0$,
the Delaney-Dress symbol codes a periodic tiling of the hyperbolic plane ($\mathcal{Cv}<0$) or -- in case additional divisibility rules are fulfilled -- the sphere ($\mathcal{Cv}>0$) \cite{olaf_ddsymbols}. 
The elements of $ \mathcal{D}$ are the orbits of
chambers of the tiling under the symmetry group. An element $C\in \mathcal{D}$ with $C\sigma_i=C$ represents an orbit of chambers with mirror symmetries stabilizing the edges of type $i$.
If there are no $C\in \mathcal{D}$ with $C\sigma_i=C$, the symmetry group contains no pure reflections, but maybe sliding reflections.
If there are no odd cycles, that is $C\sigma_{i_1}\dots \sigma_{i_k}\not= C$ for odd $k$, all symmetries are orientation preserving.
Also the maps $m_{01}$ and $m_{12}$ give information about the symmetry group of that tiling. Let $\{i,j,k\}=\{0,1,2\}$, $i<j$ and for $C \in \mathcal{D}$ let $r_{ij}(C)=\min\{r \mid C(\sigma_i\sigma_j)^r=C\}$.
Note that $r_{ij}()$ is constant on $\langle \sigma_i, \sigma_j\rangle$-orbits. If 
a $\langle \sigma_i,\sigma_j\rangle$-orbit $C^{\langle \sigma_i,\sigma_j\rangle}$  contains no $C'$ with
$C'\sigma_i=C'$ or $C'\sigma_j=C'$, then the vertices of type $k$ of the corresponding chambers in the tiling are centers of an $f_r$-fold rotation
with $f_r=m_{ij}(C)/r_{ij}(C)$.
If an orbit $C^{\langle \sigma_i,\sigma_j\rangle}$ contains a $C'$ with
$C'\sigma_i=C'$ or $C'\sigma_j=C'$, then with $f_m=2m_{ij}(C)/r_{ij}(C)$ for $f_m>1$ the vertices of type $k$ of the chambers in orbit $C$ are intersections of mirror axes with an angle
of $360/f_m$ degrees.

Later it will become clear (but not formally proven) that the Delaney-Dress symbol that is associated with lsp- and lopsp-operations in Theorem \ref{thm:lsp_is_dd} and Theorem
\ref{thm:lopsp_is_dd} is that of the tiling $T$ that can be obtained by applying the l(op)sp-operation to the (infinite) embedded graph associated with the regular hexagonal tiling
and choosing an appropriate symmetry group.  In this graph, every vertex has degree three and every face has six edges. One could replace the values 3 and 6 used for the $m_{ij}$
in the definitions of Delaney-Dress symbols corresponding to lopsp-operations by for example 4 and 4, and Theorem \ref{thm:lopsp_is_dd} would still be true. It would however be the
Delaney-Dress symbol of the tiling that can be obtained by applying the lopsp-operation to the regular square tiling, which is 4-regular and every face has 4 edges. By using other
numbers, other tilings -- even spherical or hyperbolic ones -- could be used as source tilings in \cite{gocox}.

We will now define lsp- and lopsp-operations and by showing the correspondence with Delaney-Dress symbols prove that they are equivalent to the operations defined in \cite{gocox}.

In \cite{gocox}, lsp- and lopsp-operations are defined as triangles `cut' out of the barycentric subdivision of a 3-connected tiling of the plane, such that in case of
lsp-operations the sides of the triangle are on symmetry axes of the tiling. In this article we give more direct definitions of lsp- and lopsp-operations, similar to
\cite{goetschalckx2020generation} and \cite{lopsp2020}, that are equivalent to those in \cite{gocox} when restricted to what we will later call $c3$-operations.  Later we
will prove that with this definition, lopsp-operations and lsp-operations can be associated with the Delaney-Dress symbol of a tiling of the Euclidean plane. 

\begin{definition}\label{def:lsp}

  Let $O$ be a 2-connected plane embedded graph with vertex set $V$, together with a labeling function $t: V \rightarrow \{0,1,2\}$ and an outer face which contains three special vertices marked as
  $v_0$, $v_1$ and $v_2$. We say that a vertex $v$ is of \emph{type} $i$ if $t(v)=i$. 
  This embedded graph $O$ is a \emph{local symmetry preserving operation}, lsp-operation for short, if the following properties hold:
	\begin{enumerate}
		\item Every inner face is a triangle
		\item There are no edges between vertices of the same type.
		\item For each vertex that is not in the outer face:
		\begin{align*}
			t(v)=1 &\Rightarrow deg(v)=4\\
			%t(v)\in \{0,2\} &\Rightarrow deg(v)\geq 2k
		\end{align*}
		For each vertex $v$ in the outer face, different from $v_0$, $v_1$ and $v_2$:
		\begin{align*}
			t(v)=1 &\Rightarrow deg(v)=3\\
			%t(v)\in \{0,2\} &\Rightarrow deg(v)\geq k+1
		\end{align*}
		and
		\[t(v_0),t(v_2)\neq 1\]
%		\[deg(v_0), deg(v_2)\geq 2\]
		\begin{align*}
			t(v_1)=1 &\Rightarrow deg(v_1)=2\\
		%	t(v_1)\in \{0,2\}  &\Rightarrow deg(v_1)\geq \lceil \frac{k}{2}\rceil +1
		\end{align*}
	\end{enumerate}
	
	\noindent
	Just like for barycentric subdivisions we say that an edge is of \emph{type $i$} if it is not incident to a vertex of type $i$. This is well-defined because of the second property.
	
	Every inner face has exactly one vertex and one edge of each type. We will refer to these triangular faces as \emph{chambers}.

\end{definition}

In the original paper \cite{gocox} only operations that preserve 3-connectivity of polyhedra were discussed, so the result of the operation also had to have only vertices of degree
at least 3. In \cite{goetschalckx2020generation} also operations were discussed that produced graphs with 1- or 2-cuts, but the restriction that vertices in the result
should have degree at least 3 was kept, so the definition in \cite{goetschalckx2020generation} differs slightly from the one given here.

{\bf How to apply an lsp-operation:} Let $O$ be an lsp-operation and let $G$ be an embedded graph. The operation is applied to $G$ by
first replacing the 0-edges of $B_G$ by copies of the boundary path in $O$ from $v_2$ to $v_1$ not containg $v_0$. The copy of $v_2$ is identified with the 2-vertex and the
copy of $v_1$ with the 1-vertex. Analogously the 1-edges of $B_G$ are replaced by copies of the boundary path from $v_2$ to $v_0$ not containg $v_1$, and
the 2-edges of $B_G$ are replaced by copies of the boundary path from $v_1$ to $v_0$. Then -- depending on the orientation --
        either a copy of $O$ or a copy of the mirror image of $O$ is glued into every chamber of $B_G$, identifying the boundary vertices with their copies.
        This results in a triangulation -- that is: an embedded graph where all faces are triangles --
        with labeled vertices. This triangulation is the barycentric subdivision of an embedded graph $O(G)$, the \emph{result} of applying $O$ to $G$.
        An example application -- restricted to a single face -- is given in Figure~\ref{fig:truncation_applied}.

\begin{figure}
	\centering
	\input{bary_hexagon_small.tikz} \qquad \input{deco_truncation_old.tikz}\qquad 
	\input{applied_truncation.tikz}
	\caption{\label{fig:truncation_applied}On the left, the barycentric subdivision of a hexagonal face is shown. In the middle, the lsp-operation {\em truncation} is given and on the right the result of applying the operation. The blue shaded area shows one chamber of the original hexagon.}
\end{figure}
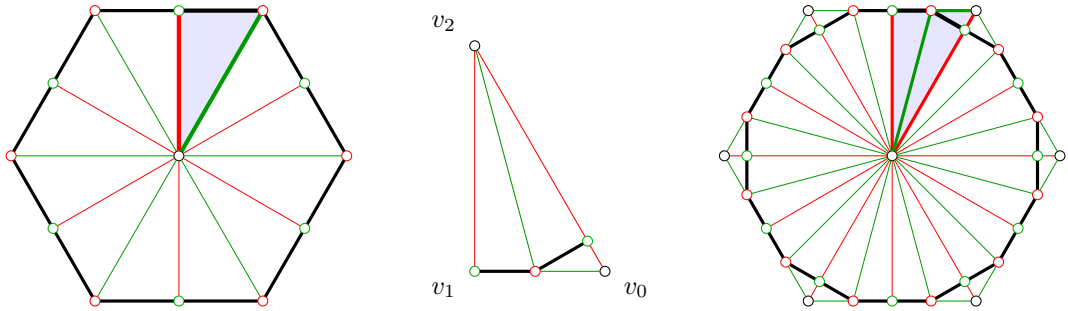

As any symmetry group acts on the chamber system, lsp-operations preserve all the symmetries of an embedded graph. Also new symmetries can occur and it is an open question whether
for polyhedra this can only occur for self-dual polyhedra. There are also interesting operations such as gyro and snub that only preserve the orientation
preserving symmetries. These cannot be described by lsp-operations. In the supplementary material of \cite{gocox} and in \cite{lopsp2020}, local orientation-preserving
symmetry preserving operations (lopsp-operations) are defined similarly to lsp-operations. The most important difference is that here the decoration is applied to double chambers
instead of chambers. As with lsp-operations, we will give a more direct definition of lopsp-operations that does not rely on tilings.

There are some problems that arise in the original definition of lopsp-operations that did not appear for lsp-operations. With the original definition, it is possible to cut
different patches out of a tiling that should describe the same operation and must be shown to have the same result. That is why we define a
lopsp-operation as a plane triangulation, similar to \cite{lopsp2020}, and not as a quadrangle that we can glue directly into double chambers. Nevertheless although this simplifies the definition of a lopsp-operation,
the same problem comes back when it is described how the operation is applied.

\begin{definition}\label{def:lopsp}

  Let $O$ be a 2-connected plane embedded graph with vertex set $V$, together with a labeling function $t: V \rightarrow \{0,1,2\}$ and three special vertices marked as $v_0$, $v_1$, and $v_2$. We say that a vertex is of \emph{type} $i$ if $t(v)=i$. 
  This embedded graph $O$ is a \emph{local orientation-preserving symmetry preserving operation}, lopsp-operation for short, if the following properties hold:
	\begin{enumerate}
		\item Every face is a triangle.
		\item There are no edges between vertices of the same type.
		\item For each vertex $v$ different from $v_0$, $v_1$, and $v_2$:
		\begin{align*}
		t(v)=1 &\Rightarrow deg(v)=4\\
		%t(v)\in \{0,2\} &\Rightarrow deg(v)\geq 2k
		\end{align*}
	and
	\[t(v_0),t(v_2)\neq 1\]
%	\[deg(v_0), deg(v_2)\geq 2\]
	\begin{align*}
		t(v_1)=1 &\Rightarrow deg(v_1)=2\\
		%t(v_1)\in \{0,2\}  &\Rightarrow deg(v_1)\geq k \quad (\mbox{which is only an extra condition if } k=3).
	\end{align*}
	\end{enumerate}

\noindent
Again we say that an edge is of \emph{type $i$} if it is not incident to a vertex of type $i$ and is this well-defined because of the second property. Note that the edges incident
with a vertex are of two different types, and as every face is a triangle, these types appear in alternating order in the cyclic order around the vertex. This implies that every
vertex has an even degree and therefore there are no odd cycles in the dual. The requirement that $O$ is 2-connected is mentioned in the beginning, but would in fact also follow
from the other conditions.

Again every face has exactly one vertex and one edge of each type and will be referred to as a \emph{chamber}.
\end{definition}

\begin{figure}
	\centering
	
	\input{gyro_lopsp_path.tikz} \qquad
	\input{gyro_patch.tikz}
	\caption{\label{fig:lopsp_gyro}On the left, the lopsp-operation gyro is shown. The dashed edges are the edges of the path $P$. On the right the corresponding double chamber patch $O_P$ is drawn. The two paths resulting from $P$ are once shown with dashed edges and once with dotted edges.}
\end{figure}
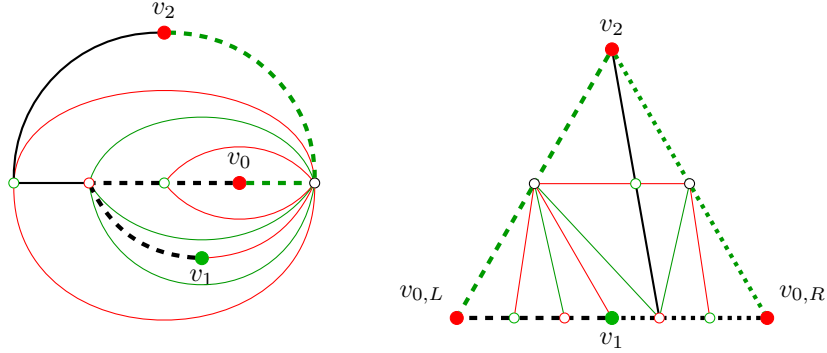

{\bf How to apply a lopsp-operation:}

For vertices $v,v'$ in a path $P$ we write $P_{v,v'}$ for the subpath of $P$ from $v$ to $v'$.

As lopsp-operations are 2-connected, due to Menger's theorem there are two paths -- one from $v_0$
to $v_1$ and one from $v_0$ to $v_2$ that have only $v_0$ in common.
These paths together form a longer path $P$ from $v_1$ to $v_2$ through $v_0$ and as a path $P$ has a single face in which only $v_1$ and $v_2$ occur once and all other vertices twice. We say that such a path $P$ is a \emph{cut-path} of $O$.
As $P$ is a subgraph of $O$, we can consider the internal component of its only face, which we call the {\em double chamber patch} $O_P$ for $O$ and $P$. It
 can be drawn in the plane, so that the two copies of the path form the boundary of the outer face.
Figure~\ref{fig:lopsp_gyro} shows
this for the operation gyro. The result of the cutting is a 4-gon with corner vertices $v_1$, $v_2$, and two copies of $v_0$, which we will denote as $v_{0,L}$ and $v_{0,R}$.

The lopsp-operation is now applied by first replacing the edges of a double chamber system $D_G$ of an embedded graph to form the graph $D_{G,P}$. An edge of type 2 is replaced by a copy of $P_{v_0,v_1}$ and an edge of type 1 is replaced by a copy of $P_{v_0,v_2}$ in a way that for $i\in \{0,1,2\}$ a copy of $v_i$ is identified with a vertex of type $i$.

Gluing copies of the double chamber patch $O_P$ into the faces of $D_{G,P}$ -- identifying corresponding vertices in $D_{G,P}$
and the copies of double chamber patches --  gives the graph $O_P(G)$ that is the result of the operation.
Note that the orientation inside a double chamber fixes how the different copies of $v_0$ have to be
identified. Figure~\ref{fig:gyro_applied} gives an example -- restricted to one face -- of this application. 

Of course the two paths are far from unique, so there are many ways how a lopsp-operation can be applied. One of the main results in this article is that although the ways in which
the operation is applied differ, the result of applying a lopsp-operation to an embedded graph is independent of the paths chosen. This will be proven in
Chapter~\ref{sec:path-invariance}.

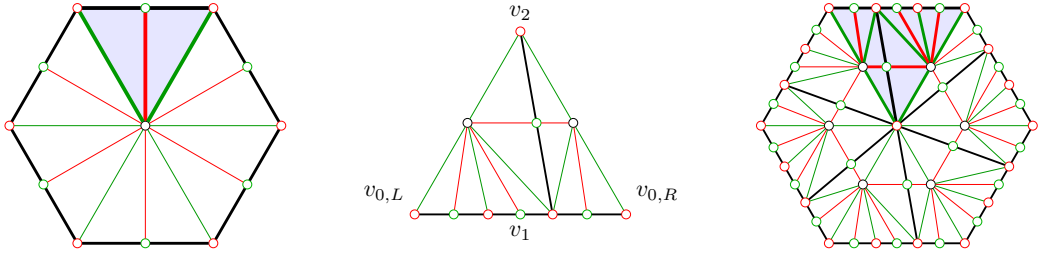
\begin{figure}
	\centering
	\scalebox{0.9}{\input{bary_hexagon_double.tikz}} \qquad
        \scalebox{0.9}{\input{gyro_patch_full.tikz}} \qquad
	\scalebox{0.9}{\input{gyro_applied.tikz}}
  
	\caption{\label{fig:gyro_applied}On the left the barycentric subdivision of a hexagonal face is shown. On the right the lopsp-operation gyro -- depicted in the middle -- is applied to it. The blue shaded area shows one double chamber.}
\end{figure}

\subsection*{The relation between Delaney-Dress symbols and operations}

\begin{figure}
	\centering
	\input{deco_truncation.tikz} \qquad \input{truncation_DD.tikz}
	\caption{\label{fig:truncation_DD}The operation {\em truncation} and the Delaney-Dress symbol encoding a tiling from which the operation can be obtained when the original definition is applied.}
\end{figure}
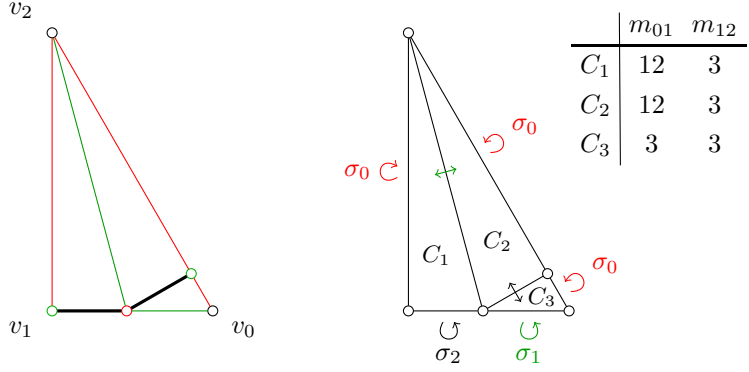

Now we will show the correspondence between Delaney-Dress symbols, lsp-, and lopsp-operations. Due to the relation between Delaney-Dress symbols and tilings as described in
\cite{olaf_ddsymbols} and \cite{DressHuson87}, this also shows the equivalence of the combinatorial definitions of lsp- and lopsp-operations defined here and the geometric ones
given in \cite{gocox}. For $(i,j)\in \{(0,1),(0,2),(1,2)\}$ let $v^{ij}(C)$ be the vertex of chamber $C$ that is not of type $i$ or $j$. 

	Let $O$ be an lsp-operation and $\mathcal{D}_O$ be the set of inner faces (chambers) of $O$. We define the action of $\Sigma$ on $\mathcal{D}_O$ by letting
        $C\sigma_i=C'$ if $C$ and $C'$ share their $i$-edge, and $C\sigma_i=C$ if $C$ shares its $i$-edge with the outer face of $O$. Let $(i,j)\in \{(0,1),(0,2),(1,2)\}$. We define $m_{ij}:\mathcal{D}_O \rightarrow \mathbb{N}$
        with $C^{\langle \sigma_i,\sigma_j\rangle}$ the $\sigma_i\sigma_j$-orbit of $C$ as follows:
	\begin{equation*}
	  m_{ij}(C)=\begin{cases}
          			\left|C^{\langle \sigma_i,\sigma_j\rangle}\right| \cdot 2 \quad &\text{if }v^{ij}(C)= v_1\\
			\left|C^{\langle \sigma_i,\sigma_j\rangle}\right| \cdot 3 &\text{if }v^{ij}(C)= v_0\\
			\left|C^{\langle \sigma_i,\sigma_j\rangle}\right| \cdot 6 &\text{if }v^{ij}(C)= v_2\\
			\dfrac{\left|C^{\langle \sigma_i,\sigma_j\rangle}\right|}{2} &\text{if $v^{ij}(C)$ is an inner vertex}\\
			{\left|C^{\langle \sigma_i,\sigma_j\rangle}\right|} \quad &\parbox{6cm}{if $v^{ij}(C)$ is an outer vertex different from $v_0,v_1$, and $v_2$}\\
		\end{cases}
	\end{equation*}

        Note that the requirements for the vertex degrees in an lsp-operation imply that $m_{02}(C)=2$ for all $C\in\mathcal{D}_O$.
        
        Then we define $\mathcal{D}(O)=(\mathcal{D}_O, m_{01}, m_{02}, m_{12})$ and call it the Delaney-Dress symbol corresponding to the lsp-operation
        $O$. This correspondence is illustrated for the operation
        ``truncation'' in Figure~\ref{fig:truncation_DD}. The following theorem states that it is in fact a Delaney-Dress symbol of a tiling of the Euclidean plane.

\begin{theorem}\label{thm:lsp_is_dd}
	If $O$ is an lsp-operation, then $\mathcal{D}(O)=(\mathcal{D}_O, m_{01}, m_{02}, m_{12})$ is the \emph{Delaney-Dress symbol} of a tiling of the Euclidean plane.
\end{theorem}

\begin{proof}
	We have to prove the properties in Definition \ref{def:dd_symbol}. The first two properties are obvious, so we will focus on the remaining ones.
	\begin{description}
		
	\item[3]: Let $(i,j)\in \{(0,1),(0,2),(1,2)\}$.  A $\langle \sigma_i,\sigma_j\rangle$-orbit consists of all the chambers sharing the same vertex $v^{ij}(C)$, so that by
          definition $m_{ij}$ is constant on $\langle \sigma_i,\sigma_j\rangle$-orbits.

          Elementary calculations show that if for a chamber $C(\sigma_i\sigma_j)^k=C$, then $C(\sigma_j\sigma_i)^k=C$ and that this holds for every chamber in $C^{\langle \sigma_i,\sigma_j\rangle}$.

	  Let $C$ be any chamber of $\mathcal{D}_O$. If $v^{ij}(C)$ is not in the outer face, we have $\left| C^{\langle\sigma_i,\sigma_j\rangle}\right| = deg(v^{ij}(C))$ and as $(\sigma_i\sigma_j)$ proceeds around
           $v^{ij}(C)$ in steps of two chambers, we have $C(\sigma_i\sigma_j)^{m_{ij}(C)}=C(\sigma_i\sigma_j)^{deg(v^{ij}(C))/2}=C$.

          If $v^{ij}(C)$ is in the outer face, it has two edges in the boundary. Let $C$ be a chamber in the boundary sharing $v^{ij}(C)$'s (w.l.o.g.) $i$-edge in the boundary. Then with $\sigma^{ij,m}$ a sequence of
          in total $m=\left|C^{\langle \sigma_i,\sigma_j\rangle}\right|$ alternating $\sigma_i$ and $\sigma_j$ we have that $C'=C\sigma^{ij,m}$ is the chamber of the orbit sharing the other edge with the boundary.
          Applying the same argument again to $C'$ (note that $C'$ can share an $i$-edge or a $j$-edge with the boundary), we get that $C(\sigma_i\sigma_j)^{\left|C^{\langle \sigma_i,\sigma_j\rangle}\right|}=C\sigma^{ij,2m}=C$
          and therefore $C(\sigma_i\sigma_j)^{k\cdot \left|C^{\langle \sigma_i,\sigma_j\rangle}\right|}=C$ which gives the result with $k\in \{1,2,3,6\}$ for the different cases for $v^{ij}(C)$.
	
	\item[4]: Let $\{i,j,k\}=\{0,1,2\}$.
          For a vertex $v$ of type $k$ and $i<j$ we define $\alpha(v)=	\sum\limits_{\substack{C\in \mathcal{D}_O\\ v\in C}}\left(\frac{1}{m_{ij}(C)}\right)$.

		Counting the number of chambers with a certain vertex and using the definition of $m_{ij}()$, we get that
		
		\begin{equation*}
			\alpha(v)=
			\begin{cases}
				2 &\text{if $v$ is an inner vertex}\\
				1 &\parbox{8cm}{if $v$ is an outer vertex different from $v_i$ for $i=0,1,2$}\\
				1/2 &\text{if }v= v_1\\
				1/3 &\text{if }v= v_0\\
				1/6 &\text{if }v= v_2
			\end{cases}.
		\end{equation*}
		
		Let $V_O, E_O, F_O$ be the sets of vertices, edges, and faces of $O$, so $|F_O|=|\mathcal{D}_O|+1$. Let $k$ be the number of vertices (and equivalently edges) in the outer face.

                As every vertex of $O$ is of exactly one type we get that:
		\begin{align*}
                  \mathcal{C}(\mathcal{D}_O, m_{01}, m_{02}, m_{12})&= 
		  \sum\limits_{D\in\mathcal{D}_O}\left(\frac{1}{m_{01}(C)} + \frac{1}{m_{12}(C)} -\frac{1}{m_{02}(C)}\right) \\
                  &=\sum\limits_{D\in\mathcal{D}_O}\left(\frac{1}{m_{01}(C)} + \frac{1}{m_{12}(C)} + \frac{1}{m_{02}(C)} \right) -\sum_{C \in\mathcal{D}_O}\left(1\right)\\
			&=\sum_{\substack{v\in V_O}}\alpha(v) - (|F_O|-1)\\
			&=(|V_O|-k)\cdot 2 +(k-3)\cdot 1 + \frac{1}{2} + \frac{1}{3} + \frac{1}{6} - (|F_O|-1)\\
			&=2|V_O| - |F_O| - k - 1
		\end{align*}
		
		By counting the number of directed edges associated with edges in the triangulation $O$ in two ways, we get that $2|E_O|=3(|F_O|-1) + k$ or equivalently $|F_O|=  2|E_O| -2|F_O| +3 - k$. We also know that $O$ is plane, so $|V_O| -|E_O|+|F_O|=2$. It follows that:
		\begin{align*}
			\mathcal{C}(\mathcal{D}_O, m_{01}, m_{02}, m_{12}) 
			&= 2|V_O| -2|E_O|+2|F_O| - 3 + k - k - 1=0
		\end{align*}
	\end{description}
	\vspace{-0.8cm}
\end{proof}

\begin{figure}
	\centering
	\scalebox{0.9}{\input{gyro_patch.tikz}} \qquad \scalebox{0.9}{\input{gyro_DD.tikz}}
	\caption{\label{fig:lopsp_gyro_DD}On the left the double chamber patch of the lopsp-operation gyro is shown and on the right the corresponding Delaney-Dress symbol.}
\end{figure}
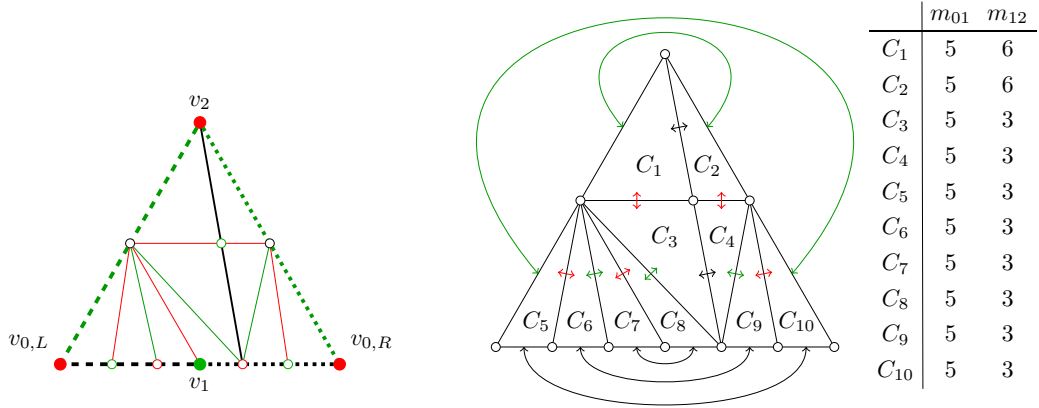

	Let $O$ be a lopsp-operation and let $\mathcal{D}_O$ be the set of chambers of $O$. We define the action of $\Sigma$ on $\mathcal{D}_O$ by letting $C\sigma_i=C'$ if $C$ and $C'$ share their $i$-edge. We define $m_{01}, m_{02}, m_{12}:\mathcal{D}_O \rightarrow \mathbb{N}$ as follows:
	\begin{equation*}
		m_{ij}(C)=\begin{cases}
			\dfrac{\left|C^{\langle \sigma_i,\sigma_j\rangle}\right|}{2} \cdot 2 \quad &\text{if }v^{ij}(C)= v_1\\
			\dfrac{\left|C^{\langle \sigma_i,\sigma_j\rangle}\right|}{2} \cdot 3 &\text{if }v^{ij}(C)= v_0\\
			\dfrac{\left|C^{\langle \sigma_i,\sigma_j\rangle}\right|}{2} \cdot 6 &\text{if }v^{ij}(C)= v_2\\
			\dfrac{\left|C^{\langle \sigma_i,\sigma_j\rangle}\right|}{2} &\text{else}
		\end{cases}
	\end{equation*}

        Note that $m_{02}(C)=2$ for all $C\in \mathcal{D}_O$.
	
	Then we define $\mathcal{D}(O)=(\mathcal{D}_O, m_{01}, m_{02}, m_{12})$ and call it the Delaney-Dress symbol corresponding to the lopsp-operation $O$.   This
correspondence is illustrated for the operation ``gyro'' in Figure~\ref{fig:lopsp_gyro_DD}.

\begin{theorem}\label{thm:lopsp_is_dd}
  
If $O$ is a lopsp-operation, then $\mathcal{D}(O)=(\mathcal{D}_O, m_{01}, m_{02}, m_{12})$ is the \emph{Delaney-Dress symbol} of a tiling of the Euclidean plane. We call this tiling the {\em associated tiling} $T_O$ of $O$.
        
\end{theorem}
\begin{proof}
	We prove the properties in Definition \ref{def:dd_symbol}. Again, the first two are obvious.
	\begin{description}
	\item[3]: With exactly the same reasoning as for internal vertices in case 3 of the proof of Theorem~\ref{thm:lsp_is_dd} we get that 
          for $(i,j)\in \{(0,1),(0,2),(1,2)\}$ we have that $C(\sigma_i\sigma_j)^{\left|C^{\langle \sigma_i,\sigma_j\rangle}\right| /2}=C(\sigma_i\sigma_j)^{deg(v^{ij}(C))/2}=C$ and as
          all $m_{ij}(C)$ are a multiple of $\left|C^{\langle \sigma_i,\sigma_j\rangle}\right|/2$, we have that $C(\sigma_i\sigma_j)^{m_{ij}(C)}=C$ for all $C$ and all $(i,j)\in \left\{(0,1),(0,2),(1,2)\right\}$.

	\item[4]:    Let $\{i,j,k\}=\{0,1,2\}$ and $V_O,E_O,F_O$ be the sets of vertices, edges, and faces of the operation.
                                 For a vertex $v\in V_O$ of type $k$ and $i<j$ we again define $\alpha(v)=	\sum\limits_{\substack{C\in \mathcal{D}_O\\ v\in C}}\left(\frac{1}{m_{ij}(C)}\right)$.

                                 		Counting the number of chambers with a certain vertex and using the definition of $m_{ij}()$, we get that
		
		\begin{equation*}
			\alpha(v)=
			\begin{cases}
				2 &\text{if } v \notin \{v_0,v_1,v_2\}\\
				1 &\text{if }v= v_1\\
				2/3 &\text{if }v= v_0\\
				1/3 &\text{if }v= v_2
			\end{cases}.
		\end{equation*}

                We can now compute the curvature $\mathcal{C}(\mathcal{D}_O, m_{01}, m_{02}, m_{12})$:

                        \begin{align*}         \mathcal{C}(\mathcal{D}_O, m_{01}, m_{02}, m_{12})&= 
			\sum\limits_{C\in\mathcal{D}_O} \left(\frac{1}{m_{01}(C)} + \frac{1}{m_{12}(C)} -\frac{1}{m_{02}}\right) \\
			&=\sum\limits_{C\in\mathcal{D}_O}\left(\frac{1}{m_{01}(C)} + \frac{1}{m_{12}(C)} +\frac{1}{m_{02}} - 1\right)
                        	\end{align*}

                  Reordering the sum so that the contributions of the $m_{ij}()$ are grouped around the common vertices we get

                        \begin{align*}    
                   \mathcal{C}(\mathcal{D}_O, m_{01}, m_{02}, m_{12})&=  \sum_{v\in V_O}\alpha(v)-|\mathcal{D}_O| \\
			&= (|V_O|-3)\cdot 2 + \frac{2}{3} + 1 + \frac{1}{3} - |F_O|\\
			&= 2|V_O| -|F_O| - 4	
		\end{align*}
		
		As $O$ is a triangulation, we get that $2|E_O|=3|F_O|$ and as $O$ is plane, we have $|V_O| -|E_O|+|F_O|=2$. It follows that:

                \[\mathcal{C}(\mathcal{D}_O, m_{01}, m_{02}, m_{12})
		= 2|V_O| -2|E_O|+2|F_O| - 4 = 0\]

	\end{description}
	\vspace{-0.8cm}
\end{proof}

As lsp-operations preserve all symmetries of an embedded graph, they also preserve the orientation-preserving symmetries, so one would expect that for every lsp-operation, there is
a lopsp-operation that has the same result when applied to any embedded graph. This observation allows to prove some properties of the result of applying lsp- or lopsp-operations
only for lopsp-operations, as the result for lsp-operations can then be deduced from the corresponding lopsp-operation.  Such an equivalent lopsp-operation can be obtained in the
following way:

Let $O$ be an lsp-operation, and let $c$ be the boundary of the outer face of $O$. Let $O_{lopsp}$ be the embedded graph obtained by glueing a mirrored copy of the inner face of
        $c$ into the outer face (identifying the vertices on $c$ with their copies). The special vertices $v_0,v_1,v_2$ of $O$ are also the special vertices of $O_{lopsp}$.

\begin{lemma}\label{lem:lsp_to_lopsp}
  If $O$ is an lsp-operation, then $O_{lopsp}$ is a lopsp-operation, and $O(G)=O_{lopsp}(G)$ for any embedded graph $G$.\\
   The Delaney-Dress symbols $\mathcal{D}(O)$ and $\mathcal{D}(O_{lopsp})$ are Delaney-Dress symbols of combinatorially isomorphic tilings.
\end{lemma}

\begin{proof}
$O_{lopsp}$ is obviously a triangulation of the disc and there are no edges between vertices of the same type. As $O_{lopsp}$ consists of two copies of $O$, glued along the boundary $c$, we can associate a
        unique vertex $o(x)$ of $O$ with every vertex $x$ of $O_{lopsp}$. The degree of $x$ in $O_{lopsp}$ is given by
	\begin{equation*}
		deg(x)=\begin{cases}
			deg(o(x)) &\text{if $o(x)$ is not in $c$}\\
			2deg(o(x))-2 &\text{if $o(x)$ is in $c$}
		\end{cases}.
	\end{equation*}
	
	From the degree restrictions for lsp-operations we can now deduce the conditions on the degrees in the definition of lopsp-operations for $O_{lopsp}$.
        It follows that $O_{lopsp}$ is a lopsp-operation.
	
	Choosing the cut-path in $O_{lopsp}$ that corresponds to the path from $v_1$ to $v_2$ through $v_0$ in the boundary of $O$ for the application
        of $O_{lopsp}$ shows immediately that the results of applying $O$ and $O_{lopsp}$ are isomorphic as the two neighbouring chambers are filled in the same way. The fact
        that the result is independent of the choice of the path will be proven in Chapter~\ref{sec:path-invariance}.

        Mapping each chamber $C_{lopsp}$ of $\mathcal{D}(O_{lopsp})$ onto the corresponding chamber $C$ of $\mathcal{D}(O)$, we have (in the notation of \cite{DressHuson87})
        a morphism between the symbols and in the notation of \cite{olaf_ddsymbols} a Delaney map $f()$, that is:

        For all $k\in \{0,1,2\}$, $(i,j)\in \{(0,1),(0,2),(1,2)\}$, and chambers $C$ of $\mathcal{D}(O_{lopsp})$ we have
        $f(C\sigma_k)=(f(C))\sigma_k$ and $m_{ij}(C)=m_{ij}(f(C))$.

        The existence of such a morphism guarantees (see \cite{olaf_ddsymbols},\cite{DressHuson87}) that $\mathcal{D}(O)$ and $\mathcal{D}(O_{lopsp})$ code combinatorially
        isomorphic tilings and that the tiling coded by $\mathcal{D}(O_{lopsp})$ can be obtained from the tiling coded by $\mathcal{D}(O)$ by {\em symmetry breaking}, that is:
        modifying the tiling, so that the combinatorial structure is preserved, but some metric symmetries of the tiling are destroyed.
\end{proof}

\section{The path invariance of lopsp-operations}\label{sec:path-invariance}

In this chapter we will prove that the result of applying a lopsp-operation is independent of the path chosen to form the double chamber patch. An essential tool are {\em chamber
  flips} which simulate homotopic deformations.

\begin{definition}\label{def:chamber_flip}
  Let $P$ be a directed walk in a triangulation. For any two different vertices of a triangle (or chamber) $C$, there are two different simple paths $P_0$, $P_1$ between these vertices
  in the boundary of $C$. If for $i\in \{0,1\}$ path $P_i$ occurs at a certain position in $P$, then a \emph{chamber flip} of $C$ (at this position) is the operation of replacing $P_i$ by
  $P_{1-i}$.
\end{definition}

As a first tool we will discuss transformations of one path into another:

\begin{lemma}\label{lem:move_path_over_face}

  Let $P,P'$ be two directed paths of the form $P=P_sR$, $P'=P_sR'$ from $x$ to $y$ in a triangulation $T$ of any genus, so that 
  $R'R^{-1}$ is a facial walk of an internally plane face $f$ in the subgraph of $T$ consisting of vertices and edges of $P$ and $P'$.

  Then there is a sequence of paths $P=P_0,P_1,\dots ,P_k=P'$ so that for $1 \le i \le k$ path $P_i$ is obtained from $P_{i-1}$ by a chamber flip and all paths contain only
  vertices of $f$ or in the interior of $f$. As chamber flips can be reversed, the same is true with the role of $P$ and $P'$ interchanged.

\end{lemma}

\begin{proof}

  We will prove this by induction on the number $c=|C_f|$ with $C_f$ the set of chambers of $T$ inside $f$.  If $c=1$, $R$ and $R'$ are the two paths along the boundary of a chamber,
  so one can be transformed into the other by a chamber flip. First we will prove that if $c\ge 2$ there are at least two chambers in $C_f$ with a connected intersection with the facial walk
  $R'R^{-1}$ containing at least one edge: If all chambers have this property, we are done. Otherwise there is a chamber $C$ containing an edge and with a disconnected intersection
  with $R'R^{-1}$. This chamber splits the face into two parts. Let $C_0$ be a chamber in one of the parts containing an edge of $R'R^{-1}$ that has the largest distance to $C$
  along a path in the inner dual. If this chamber had a disconnected intersection with $R'R^{-1}$, it would again split the set of chambers into two parts and each path from the
  part containing $C$ into the other must pass $C_0$, in contradiction with the maximal distance between $C$ and $C_0$. So in each part there is such a chamber.

  If there is a chamber with connected intersection with $R'R^{-1}$ consisting of a single edge or of two edges 
  not containing the start vertex $x'$ or the end vertex $y$
  of $R$ and $R'$  as the middle vertex, we can do a chamber flip to obtain either a path $P_1$ or $P_{k-1}$, 
  so that we can apply induction to $P_1,P'$ or $P,P_{k-1}$ and use that each chamber flip can be undone by a reverse chamber flip.

  Otherwise $x'$ is contained in a chamber of $C_f$ with one edge of $P$ and one edge of $P'$. Applying a chamber flip using the edge of $P$, we get a path $P_1$ to which we can apply induction.
  \end{proof}

\begin{lemma}\label{lem:chamber_flip_whole_paths}
Let $Q,Q'$ be two directed paths from $x$ to $y$ in a plane triangulation $T$, and $z$ a vertex not contained in any of the paths.

  Then there is a sequence of paths $Q=Q_0,Q_1,\dots ,Q_k=Q'$ from $x$ to $y$ so that for $1 \le i \le k$ path $Q_i$ is obtained from $Q_{i-1}$ by a chamber flip and none of the paths contain $z$.
  
\end{lemma}

\begin{proof}
 
  We will prove this by induction on the number $n_e$ of edges in the beginning of $Q$ that $Q$ and $Q'$ have in common. Remember that for vertices $v,v'$ in a path $Q$ we write $Q_{v,v'}$ for the subpath of $Q$ from $v$ to $v'$.

  If $n_e=|Q'|$, then $Q=Q'$, so assume that $n_e<|Q'|$ and
  that the assumption is true for $n_e'>n_e$. Then there is a first vertex $a$ in $Q$ that is incident with an edge that is in $Q'$ but not in $Q$. Let $b$ be the next vertex after $a$ on $Q'$ that $Q'$
  shares with $Q$. We will show that $Q$ can be transformed to $Q'_{x,a}Q'_{a,b}Q_{b,y}$ in the described way, so that we can apply induction to transform $Q'_{x,a}Q'_{a,b}Q_{b,y}$ to $Q'$.

  $Q_{a,b} \cup Q'_{a,b}$ is a cycle and we call the face containing $z$ the exterior. Note that neither $Q'_{x,a}=Q_{x,a}$ nor $Q_{b,y}$ can intersect $Q_{a,b} \cup Q'_{a,b}$.

  So there are four possible combinations (depicted in Figure~\ref{fig:face_options}) of where $Q_{x,a}$ and $Q_{b,y}$ are situated.
  If they are in the interior, we use them as part of the face boundary when applying Lemma~\ref{lem:move_path_over_face},
  otherwise not. As Lemma~\ref{lem:move_path_over_face} already allows to consider also paths with a beginning common part outside the face, we can choose
  $P,P'$ from  Lemma~\ref{lem:move_path_over_face} in the following way:
  \medskip

  \begin{tabular}{ll}
  $Q_{b,y}$ outside: & Choose $P=Q_{x,a}Q_{a,b}$, $P'=Q'_{x,a}Q'_{a,b}$.\\
  $Q_{b,y}$ inside:& Choose $P=Q_{x,a}Q_{a,b}Q_{b,y}$, $P'=Q'_{x,a}Q'_{a,b}Q_{b,y}$. \\
  \end{tabular}

  \medskip

  Note that in case $Q_{x,a}$ is outside  $Q_{a,b} \cup Q'_{a,b}$ it forms the $P_s$ from Lemma~\ref{lem:move_path_over_face}, otherwise $P_s$ consists of a single vertex.
  In each case Lemma~\ref{lem:move_path_over_face} can be applied to prove that $Q$ can be transformed to
  $Q'_{x,a}Q'_{a,b}Q_{b,y}$ in the described way, and as the beginning of $Q'_{x,a}Q'_{a,b}Q_{b,y}$ has more than $n_e$ edges in common with $Q'$, we can apply induction.
\end{proof}

\begin{lemma}\label{lem:chamber_flip_paths}
	Let $O$ be a lopsp-operation and let $P$ and $P'$ be two cut-paths in $O$. There exists a series of walks $P=P_0,\ldots, P_k=P'$ with
        endpoints $v_1$ and $v_2$, such that every $P_{i+1}$ can be obtained from $P_i$ by applying one chamber flip, and $v_0$, $v_1$, and $v_2$ each appear exactly once in every
        $P_i$. 
\end{lemma}

Note that $P_1,\ldots, P_{k-1}$ need not be (simple) paths and that only for $v_0$, $v_1$ and $v_2$ it is guaranteed that they occur exactly once in these paths.

\begin{proof}
  This is now a direct consequence of Lemma~\ref{lem:chamber_flip_whole_paths} when applying the lemma first with $Q=P_{v_1,v_0}$, $Q'=P'_{v_1,v_0}$, $z=v_2$ and afterwards with
  $Q=P_{v_0,v_2}$, $Q'=P'_{v_0,v_2}$, $z=v_1$. Intermediate walks need not be paths, as the interior of a path $(P_i)_{v_1,v_0}$ can intersect with the interior
  of $(P_i)_{v_0,v_2}$.
\end{proof}

	\begin{figure}
		\centering
		\input{flips_situation1.tikz}\quad		\input{flips_situation4.tikz}
		\input{flips_situation2.tikz}\quad
		\input{flips_situation3.tikz}
		\caption{\label{fig:face_options} The four different cases in the proof of Lemma \ref{lem:chamber_flip_whole_paths} are shown here. The shaded area represents the interior.}
	\end{figure}
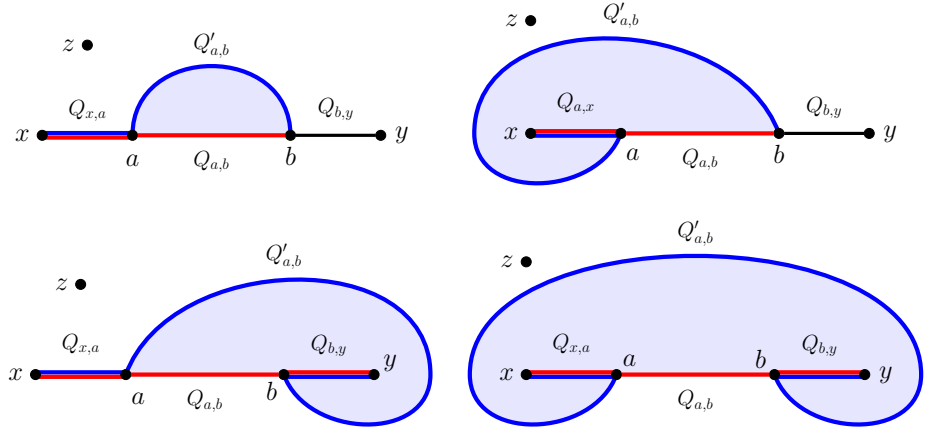

Let $G$ be an embedded graph, $O$ a lopsp-operation with cut-path $P$ and $O_P$ the corresponding double chamber patch.
As before, the result of applying the operation $O$ with path $P$ to $G$ is denoted by ${O_P}(G)$. 
	
As ${O_P}(G)$ is obtained by glueing copies of $O_P$ into $D_G$, every vertex $v$ in ${O_P}(G)$ is in at least one copy of $O_P$. If $v$ is in more than
        one copy, $v$ corresponds to the same vertex of $O$ in each of these copies. Similarly, every edge or face of ${O_P}(G)$ also corresponds to
        exactly one edge or face of $O$ respectively. This allows us to define a surjective map $\pi_P$, that maps every vertex, edge and face of ${O_P}(G)$ to its
        corresponding vertex, edge or face of $O$. 
	
	We define the map $\invpi$, that maps a set $X$ of vertices, edges or faces in $O$ to the set of all the vertices, edges or faces in ${O_P}(G)$ whose image under
        $\pi_P$ is in $X$. If we apply $\invpi$ to a single vertex, edge or face $x$ of $O$, we will often write $\invpi(x)$ instead of $\invpi(\{x\})$. Analogously we can also define $\invpi(G')$
        for subgraphs $G'$ of $O$ and if $\invpi(G')$ is a connected graph, we can interpret it as an embedded graph by the embedding induced by ${O_P}(G)$.

\begin{lemma}\label{lem:glue_lopsp_is_bary}
  Let $G$ be an embedded graph and let $O$ be a lopsp-operation with a cut-path $P$. 
  The labelled embedded graph ${O_P}(G)$ is the barycentric subdivision of a connected embedded graph.
\end{lemma}
\begin{proof}
	By construction there are no edges between vertices of the same
        type and every vertex of type 1 has degree 4. Furthermore every face of ${O_P}(G)$ is a triangle. With these properties, an embedded graph $G'$
        is defined by ${O_P}(G)$ such that ${O_P}(G)$ is its barycentric subdivision.

        The fact that $G'$ is connected follows from the fact that ${O_P}(G)$ is connected and that the edges of type 2 of chambers sharing the same vertex of type 2 form a connected subgraph.
        This way each path between two vertices of type 0 in ${O_P}(G)$ can be transformed into a path consisting only of edges of type 2.
\end{proof}

The definition of ${O_P}(G)$ depends on $P$. One of the main results of this paper is that the result is in fact independent of $P$, so that we can define $O(G)$ for a lopsp-operation $O$.

\begin{theorem}\label{thm:indep_of_paths}
	Let $O$ be a lopsp-operation and let $P$ and $Q$ be two cut-paths in $O$. Let $G$ be an embedded graph. Then ${O_P}(G) \cong {O_Q}(G)$.
\end{theorem}

\begin{proof}
  
  Let $e$ be an edge of $D_G$, and let $P^e$ be the result of replacing $e$ with $P_{v_j,v_0}$ in ${O_P}(G)$ with $j=1$ if $e$ is of type 2 and $j=2$ if $e$ is of type 1.  By
  Lemma~\ref{lem:chamber_flip_whole_paths} there is a series of paths $P_{v_j,v_0}=P_0,\ldots, P_k=Q_{v_j,v_0}$ from $v_j$ to $v_0$ in $O$, so that one path is obtained from the
  previous one by a chamber flip and none of $v_0,v_1,v_2$ occur as interior points of any of the paths. We define a sequence of paths $P^e=P^e_0,\ldots, P^e_k$ in ${O_P}(G)$ with
  $\pi_P(P^e_i)=P_i$ for $0\le i \le k$. The chamber flips in $O$ on the paths $P_i$ replace subpaths of one or two edges. In case of one edge it is clear that a corresponding 
  chamber flip can be performed on $P^e_i$ in ${O_P}(G)$. In case of two edges, we have to prove that the two corresponding edges in $P^e_i$ are also contained in the same chamber. A
  vertex $v\not\in \{v_0,v_1,v_2\}$ on $P$ not only has the same degree as a vertex in $\invpi(v)$, but seen as a vertex in $P$, respectively $\invpi(P)$, it also has equally many
  edges on the left, resp.\ right side of $P$ or $\invpi(P)$ (with the canonical choice of direction for parts of $\invpi(P)$). As none of the intermediate paths $P_i$ contain
  $v_0,v_1$, or $v_2$ at an intermediate point, this property is preserved by the chamber flips, so that if there is no edge in the angle between two consecutive edges in some
  $P_i$ and a chamber flip is performed, there are also no edges in the corresponding path $P^e_i$ in ${O_P}(G)$ and the chamber flip is also possible there. Thus a sequence of
  paths $P^e=P^e_0,\ldots, P^e_k$ in ${O_P}(G)$ with $\pi_P(P^e_i)=P_i$ exists for $0\le i \le k$ and $\pi_P(P^e_k)=Q_{v_j,v_0}$. We denote $P^e_k$ as $Q^e$. Note that $Q^e$ is
  isomorphic to $Q_{v_j,v_0}$, not to $Q$. The index $j$ is determined by the type of the edge $e$. Let $O_P(G)|_Q$ be the graph consisting of all the vertices and edges of
  $O_P(G)$ contained in $Q^e$ for some edge $e$.  With the rotational orders induced by ${O_P}(G)$ we have that $O_P(G)|_Q$ is an embedded subgraph of ${O_P}(G)$.
  
  {\bf Claim 1:} $O_P(G)|_Q$ is a subgraph of ${O_P}(G)$ that is isomorphic (as a non-embedded graph) to $D_{G,Q}$.

  {\em Proof of Claim 1:} Two paths $Q^e$ and $Q^{e'}$ can only intersect in their endpoints, as every internal vertex $v$ has $\pi_P(v) \not\in \{v_0,v_1,v_2\}$ which implies that $v$ has only
  two incident edges with an image in $Q$.  This implies that two intersecting paths are either disjoint except possibly for their endpoints or identical.
  As the first edges at $v_1$ or $v_2$ of two different
  paths are different -- they come from different edges of $D_{G,P}$ and were all transformed in the same way -- the paths must be disjoint.  It follows that $O_P(G)|_Q$ is isomorphic to $D_{G,Q}$ as an abstract graph. It is not immediately clear that they are also
  isomorphic as embedded graphs. \hfill $\square$

        Let $m$ denote the total number of chamber flips necessary to transform first $P_{v_1,v_0}$ to $Q_{v_1,v_0}$ and then $P_{v_2,v_0}$ to $Q_{v_2,v_0}$.
	With every face (that is: double chamber) $D$ of $D_G$ and $0\le i \le m$
        we can now associate a closed walk $W_i$ that consists of the four paths $P_i^{e_1},P_i^{e_2},P_i^{e_3},P_i^{e_4}$ in ${O_P}(G)$ with
        $e_1,\dots ,e_4$ the four edges of
        the double chamber, in the same order as they appear in the double chamber. We will prove that $W_m$ is a face of $O_P(G)|_Q$, so that
        $O_P(G)|_Q$ is  an embedded subgraph of ${O_P}(G)$ that is isomorphic (as an embedded graph) to $D_{G,Q}$. In fact:

         {\bf Claim 2:} $O_P(G)|_Q$ is a subgraph of ${O_P}(G)$ that is isomorphic as an embedded graph to $D_{G,Q}$ and for each face the internal component is isomorphic to $O_{Q}$.

         {\em Proof of Claim 2:} Let $D$ be a face of $D_G$, $\cal C$ be the set of all chambers in ${O_P}(G)$, and $n$ be the number of chambers in $O$.
         We will define functions $\alpha_i: {\cal C}\rightarrow \mathbb{Z}$ ($0\leq i\leq m$) with the following properties:

         \begin{description}
         \item [(i):] Let $C,C'$ in ${O_P}(G)$ be two different chambers sharing the directed edges $e$ and $e^{-1}$, so that $C$ is on the left of $e$.  For $e'\in \{e,e^{-1}\}$
           we define $n_i(e')$ as the number of times $e'$ occurs in the cyclic walk $W_i$: Then: \\
           $\alpha_i(C)-\alpha_i(C')= n_i(e)-n_i(e^{-1})$.
         \item [(ii):] For every chamber $C$ in $O$: $\sum_{C'\in {\invpi(C)}}\alpha_i(C')=1$
         %\item [(iii):] $\sum_{C\in {\cal C}}\alpha_i(C)=n$
         \end{description}

         As a consequence of (ii) we have $\sum_{C\in {\cal C}}\alpha_i(C)=n$.

         $W_0$ is an internally plane facial walk of $D_{G,P}$ with an internal component that is isomorphic to $O_{P}$. We define $\alpha_0(C)=1$ if $C$ is a chamber on the inside of $W_0$ and
         $\alpha_0(C)=0$ if $C$ is on the outside. As $W_0$ has exactly one copy of each chamber in $O$ inside  we get (ii) for $\alpha_0()$.
         As $\alpha_0()$ only differs for neighbouring chambers if they share an edge of $W_0$ (and then in the way described by (i)) we also get (i).

         \begin{figure}
		\centering
		\resizebox{.7\textwidth}{!}{\input{alpha_1.tikz}}
		\caption{\label{fig:alpha}The evolution of $\alpha$ after chamber flips.}
	\end{figure}
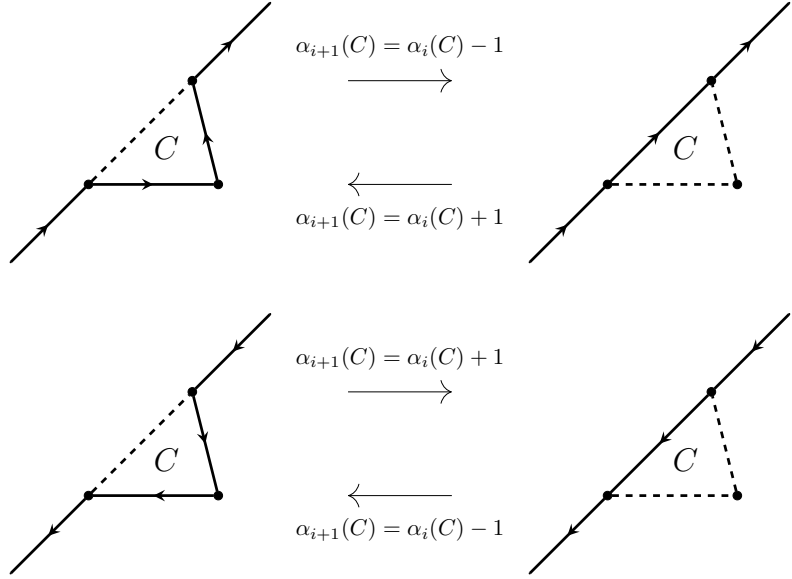	

         For $i>0$ we define $\alpha_i$ inductively. Let $C$ be the chamber of $O$ to which a chamber flip is applied when changing $W_{i-1}$ to $W_i$. These chamber flips occur in
         two places of $W_{i-1}$ as paths that are the result of $i-1$ chamber flips of copies of $P_{v_j,v_0}$ occur twice in the cyclic walk, and in fact in different directions. So two chambers $C^-,C^+$
         with $\pi(C^-)=\pi(C^+)=C$ are involved, $C^-$ on the left of the cyclic walk and $C^+$ on the right. We now define $\alpha_i(C^-)=\alpha_{i-1}(C^-)-1$
         and $\alpha_i(C^+)=\alpha_{i-1}(C^+)+1$. This is illustrated in Figure~\ref{fig:alpha}. As we once
         add one and once subtract one for two chambers with the same image under $\pi()$, (ii) is immediate. Property (i) can be easily checked by looking at
         $\alpha_{i}()$ for $C^-,C^+$ and the neighbouring chambers sharing an edge with them.
         
         As the two start edges of the paths starting at $v_2$ and $v_1$ are always moved in the same direction by the chamber flips, these vertices are in each step $i$ contained
         in a chamber $C$ with $\alpha_i(C)=0$ and a chamber $C'$ with $\alpha_i(C')=1$ (note that these chambers are not involved in any flips but the ones changing the start
         edge). As $W_m$ is simple, except for possibly $v_0$ occurring twice, we can follow $W_m$ from $v_1$ and once from $v_2$ to the copies of $v_0$ to conclude that for
         each edge of $W_m$, the two chambers $C,C'$ containing it have $\alpha_m(C)=0$ and $\alpha_m(C')=1$, so that these are the only values of $\alpha_m()$.  As $Q$ is a path, for any
         two chambers $C_O,C'_O$ in $O$ there is a path between them in the dual of $O$ not crossing edges of $Q$. So for any chamber $C$ with $\alpha_m(C)=1$, we can follow paths in
         the dual of $O_P(G)$ corresponding to the paths from $\pi(C)$ avoiding $Q$ in the dual of $O$ to find $n$ chambers with pairwise different images under $\pi()$ and $\alpha_m(C')=1$ for all these
         chambers.  As $0$ and $1$ are the only values of $\alpha_m()$, these are all the chambers $C$ with $\alpha_m(C)=1$, they form the interior of $W_m$, and they are mapped
         bijectively onto $O$ by $\pi()$. The fact that the interior is isomorphic to $O_{Q}$ can be seen by mapping a chamber $C$ to $\pi(C)$ (identifying the chambers of $O$ and
         $O_Q$) and observing that for all edges $e$ not in $W_m$ the image $\pi(e)$ is not in $Q$, so that we get -- except for the boundary -- a chamber system isomorphism
         between the interior of $W_m$ and $O_Q$.

         As the interior of $W_m$ does not contain any edges mapped to $Q$, we get that $W_m$ is a face of $O_P(G)|_Q$, so that $O_P(G)|_Q$ is isomorphic as an embedded graph to
         $D_{G,Q}$ as two paths $Q^e,Q^{e'}$ form an angle if and only if $P^e,P^{e'}$ form an angle -- which again is the case if and only if $e,e'$ form an angle, proving Claim~2
         and finally the theorem.
\end{proof}

\begin{definition}
  Let $O$ be a lopsp-operation and let $G$ be an embedded graph. Choose any cut-path $P$ in $O$.
  The result $O(G)$ of applying $O$ to $G$ is the embedded graph with barycentric subdivision ${O_P}(G)$. 
\end{definition}

This is well-defined
        by Lemma \ref{lem:glue_lopsp_is_bary}, and by Theorem \ref{thm:indep_of_paths}, the definition is independent of the chosen path. We can also define the map $\pi:=\pi_P$ as
        it is independent of the chosen path.

        \section{The effect of lsp- and lopsp-operations on polyhedrality}\label{sec:fw3}
        
Polyhedral embeddings are simple embedded graphs that are 3-connected and have a face-width of at least three. The face-width (or representativity) of an embedded graph is a measure of
`local planarity'. In some respect embeddings of high face-width have certain properties of plane graphs.  We will define face-width in a combinatorial way, using barycentric
subdivisions. It is not difficult to prove that the definition given here is equivalent to e.g.\ the definition in \cite{moh}.

\begin{definition}\label{def:ckembed}	
	A cycle in an embedded graph $G$ is \emph{contractible}\index{contractible!graph theory} if (as an embedded subgraph of $G$) it has a simple internally plane face.
	
	Let $G$ be an embedded graph. The \emph{face-width}\index{face-width} of $G$, denoted $fw(G)$\index{$fw(G)$}, is the minimal length of a non-contractible cycle in $B_G$,
        divided by two. If $G$ has no non-contractible cycles ($G$ is plane), then we define $fw(G)=\infty$.

        An embedded graph is said to be $ck$-embedded for $k\ge 1$, if it has no cut with fewer than $k$ vertices, and the face-width as well as the minimum face size and the
        minimum degree are at least $k$.

\end{definition}

The condition that neither cuts with fewer than $k$ vertices nor vertices with degree smaller than $k$ may be present instead of just requiring the graph to be $k$-connected is chosen
in order to deal with small boundary cases -- e.g.\ cycles that have as duals graphs with just two vertices.

\begin{lemma}
  A graph is $c3$-embedded if and only if the embedding is polyhedral.
\end{lemma}

The reason why nevertheless the term $c3$-embedded is used in this article is that many results are proven for the term {\em $ck$-embedded} for general $k\in \{1,2,3\}$.

\begin{proof}
  
  Let $G$ be $c3$-embedded. Then $G$ is simple as  facial loops and facial 2-cycles are excluded by the restrictions on face sizes and non-facial loops and non-facial 2-cycles would imply either smaller cuts or a smaller face-width.
  $G$ is also $3$-connected, as it has no cuts with fewer than $3$ vertices and it has at least 4 vertices, as it has minimum degree $3$ and is simple. The fact that $fw(G)\ge 3$ is also part of the definition of being $c3$-embedded.

  On the other hand let $G$ be a graph with a polyhedral embedding. Then $G$ has no faces of size smaller than $3$ as it is simple, and no vertices of degree smaller than $3$ as it is $3$-connected.
  The fact that $G$ is $3$-connected also implies that there are no cuts with fewer than $3$ vertices. Finally the fact that $fw(G)\ge 3$ is also part of the definition of a polyhedral embedding.

  \end{proof}

A cycle in an embedded graph $G$ is {\em non-contractible} if its two faces are bridged or if they are both simple and internally non-plane.  The result of applying a chamber flip
to a cycle is either another cycle, or a closed walk consisting of two cycles that share a vertex. If we apply a chamber flip to a non-contractible cycle, the result will be or
contain another non-contractible cycle:

\begin{lemma}\label{lem:non0h_onechamberflip}
	
	Let $c$ be a non-contractible cycle in the barycentric subdivision $B_G$ of an embedded graph $G$ and $C$ a chamber sharing one or two edges with $c$. Then the result of
	applying a chamber flip to $c$ at the chamber $C$ is either a non-contractible cycle or a closed walk containing a non-contractible cycle.
	
\end{lemma}

\begin{proof}
	
	If the result of applying a chamber flip is again a cycle $c'$, checking the outgoing edges of the three vertices one gets that the faces of the result are bridged if the
        faces of $c$ are bridged. If the faces of $c$ are simple, a short calculation gives that the two genera of the faces of $c'$ are the same as those of the faces of $c$.
	
	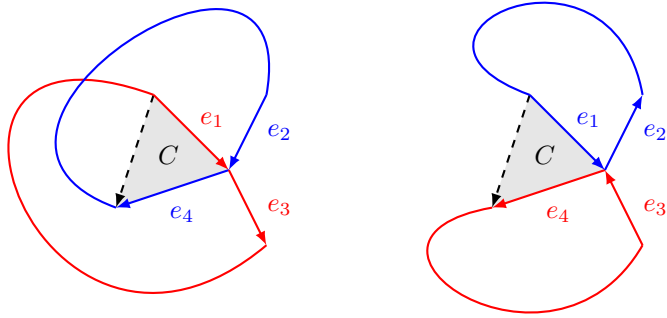
\begin{figure}
		\centering
		\input{flip_cycle.tikz}
		\caption{\label{fig:oneflip} The two possible ways a closed walk can cross a vertex of degree 4 in the walk after applying a chamber flip to a cycle.  }
	\end{figure}
	
	If the result is not a simple cycle, there are two possibilities how the closed walk passes the vertex of degree 4. These two possibilities are displayed in
	Figure~\ref{fig:oneflip}. In the left case, the two cycles cross, which means that in fact both cycles have bridged faces and both are non-contractible. In the second case the
	assumption that the cycle $c_{e_1,e_2}$ containing edges $e_1,e_2$ as well as the cycle $c_{e_3,e_4}$ containing edges $e_3,e_4$ are contractible implies that the faces $f_{e_1,e_2}$ and $f_{e_3,e_4}$ not
	containing the other cycle are internally plane, as otherwise the cycle $c$ would be part of the plane internal component of $c_{e_1,e_2}$ or $c_{e_3,e_4}$, so it would not be bridged and
	have an internal component that is also internally plane. As $f_{e_1,e_2}$, $f_{e_3,e_4}$, and the cycle bounding $C$ are all simple, the face of $c$ containing
	$C$ is also simple and the genus can be computed as $0$ using the genera of the internal components of $f_{e_1,e_2}$, $f_{e_3,e_4}$, and of $C$. So $c$ has a simple internally plane face and
	is contractible -- a contradiction.
\end{proof}

We can now prove the following Lemma, which will be important in the proof of our main result.

\begin{lemma}\label{lem:fw_G_leq_fw_O(G)}
	Let \(G\) be an embedded graph and let $O$ be a lopsp-operation. Then \[fw(G) \leq fw(O(G)).\]
\end{lemma}

While this lemma is more or less obvious for an lsp-operations $O_{lsp}$ as in that case $B_{O_{lsp}(G)}$ is a subdivision of $B_{G}$, the situation is more complicated for lopsp-operations.

\begin{proof}

	Let $c$ be a non-contractible cycle in $B_{O(G)}$ of length $2 fw(O(G))$ and $M$ a minimal set of double chambers of $B_G$ containing all edges of $c$ on the boundary or in the interior.
	
	We will construct a sequence of non-contractible cycles $c=c_0,c_1,\dots ,c_n$ also contained in chambers of $M$, so that $c_n$ will only contain edges in $D_{G,P}$ -- with $P$ the cut-path used for applying $O$.
	
	For a closed walk $w$ with all vertices of degree larger than two in $D_{G,P}$ we define a {\em segment} to be a part of $w$ that is a path starting and ending in a vertex of $D_{G,P}$ and with all internal vertices not in $D_{G,P}$. So
	segments are paths from one vertex in the boundary of a double chamber to another splitting the double chamber into two parts. For each segment $s$ we choose the chambers in the
	smallest of these parts as the set of {\em flip chambers} ${\cal F}_s$. If both parts contain the same number of chambers, an arbitrary but fixed part can be chosen. If ${\cal F}_{w}$ denotes the union
	of all ${\cal F}_s$ with $s$ a segment of $w$, then we will have ${\cal F}_{c_0} \supsetneq {\cal F}_{c_1} \supsetneq \dots    \supsetneq {\cal F}_{c_n} =\emptyset$.
	
	Let now some $c_i$, $0\le i$ be given. If ${\cal F}_{c_i}=\emptyset$, we are done. Otherwise let $s$ be a segment of $c_i$ such that $|{\cal F}_s|$ is minimal and let
        $C\in {\cal F}_s$ such that $C$ has a connected intersection with $s$ containing at least one edge. The fact that such a chamber exists can be proven by a distance argument like in the proof
        of Lemma~\ref{lem:move_path_over_face}.
        
         Let $w$ be the closed walk that is a result of a chamber flip at $C$.  Note that in $w$ the modified segment $s$ can
        become a segment $s'$ or two segments $s'_1$,$s'_2$, but in any case ${\cal F}_{s'}={\cal F}_{s}\setminus \{C\}$ or ${\cal F}_{s'_1}\cup {\cal F}_{s'_2}={\cal
          F}_{s}\setminus \{C\}$ and the smaller part of the one or two new segments are uniquely determined.  If the result of applying a chamber flip at $C$ is a simple cycle, we
        take that cycle as $c_{i+1}$ and ${\cal F}_{c_{i+1}}={\cal F}_{c_{i}}\setminus \{C\}$.  If the result is a closed walk, due to our choice of $C$ the vertex of degree 4 is
        on $D_{G,P}$ and the cycle $c_{i+1}$ is chosen as a non-contractible partial cycle that exists due to Lemma~\ref{lem:non0h_onechamberflip}.  Also in this case all segments
        of $c_{i+1}$ are segments of $w$ and therefore ${\cal F}_{c_{i+1}} \subseteq {\cal F}_{c_{i}}\setminus \{C\}$.
	
	The non-contractible cycle $c_n$ is a non-contractible cycle in $D_{G,P}$, which corresponds in a natural way to a cycle $c_G$ in $B_G$ by choosing all the edges of $B_G$ that
	occur as paths in $c_n$. If $c_G$ was contractible -- that is: had a simple internally plane face -- in $B_G$, the corresponding  face of $c_n$ in $O(G)$ would also be contractible. Therefore $c_G$ is not contractible.
	
	The cycle $c_G$ is a cycle containing only edges in double chambers of $M$, so we can again use chamber flips -- this time in $B_G$ and only using chambers contained in
	a double chamber from the set $M$ -- and apply Lemma~\ref{lem:non0h_onechamberflip} to obtain a non-contractible cycle $c'_G$ containing only type 1 edges in chambers of
        $M$. As there is a non-contractible cycle containing only type 1 edges in chambers of $M$, there is also a shortest one -- w.l.o.g.\ assume that $c'_G$ is shortest
	possible.

	\begin{figure}
		\centering
		\input{2_doublechambers.tikz}
		\caption{\label{fig:notshorter} Configurations that can not occur in shortest non-contractible cycles.  }
	\end{figure}
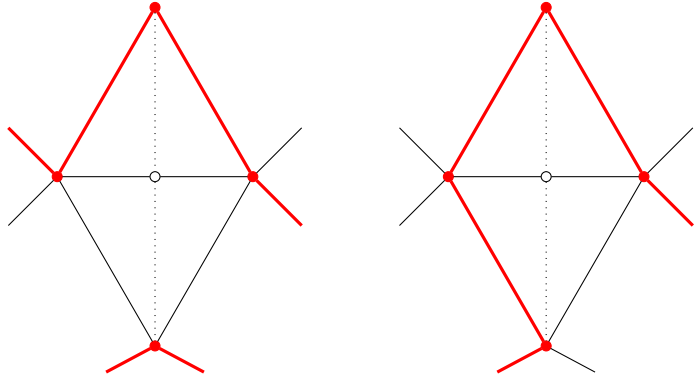

	Let $g_D: E(c'_G)\to M$ be a function mapping each edge of $c'_G$ onto a double chamber containing it, so that for double chambers $D$ containing two edges, both are
	mapped onto $D$.  Note that no edge of $c'_G$ can be contained in two double chambers together with two different edges, so this is well defined. For the length $l(c'_G)$ we have
	$l(c'_G)=\sum_{D\in M}|g_D^{-1}(D)|$.
	
	We call $D\in M$ a $k$-double chamber with $k=|g_D^{-1}(D)|$, so $k\in \{0,1,2\}$. The intersection of a 2-double chamber $D$ with the other double chambers in
        $M$ either consists of just the two 0-vertices, in which case we call it {\em isolated}, or it contains also the 1- or the 2-vertex. If $D'$ shares the 2-vertex $v$ with $D$,
	it can neither contain another edge of $c'_G$ nor share the 2-vertex with another 2-double chamber as $v$ has degree 2 in $c'_G$. If $D'$ shares the 1-vertex $w$ with $D$, of
	course it only shares it with $D$, as the 1-vertex is only contained in these two double chambers, but it also doesn't share the 2-vertex with another 2-double chamber, as
	otherwise due to Lemma~\ref{lem:non0h_onechamberflip} a shorter cycle could be obtained by chamber flips -- see Figure~\ref{fig:notshorter}. So for every 2-chamber $D$ that
	is not isolated there is a 0-chamber $z(D)$ that can uniquely be assigned to $D$.
	
	We will now prove that the length of $c'_G$ is at most $2 fw(O(G))$. Let $g: E(c)\to M$ be a function mapping each edge of $c$ onto a double chamber containing it in the boundary or in the interior. So 
	$2 fw(O(G))=\sum_{D\in M}|g^{-1}(D)|$ and due to the minimal choice of  $M$ we have $|g^{-1}(D)|>0$ for all $D\in M$. As an edge in $B_{O(G)}$ between the two 0-vertices of a double chamber
	would correspond to  a loop in $O$, such edges do not exist and for each isolated 2-chamber $D$ we have $|g^{-1}(D)|\ge 2$. With ${\cal D}_{c,1}$ the set of $1$-double chambers,   ${\cal D}_{c,2}$ the set of isolated
	2-double chambers and ${\cal D}'_{c,2}$ the set of not isolated 2-double chambers we have:
	
	\begin{align*}
		2 fw(G) & \le l(c'_G) = \sum_{D\in M}|g_D^{-1}(D)|   \\
		&  = \sum_{D\in {\cal D}_{c,1}}|g_D^{-1}(D)| + \sum_{D\in {\cal D}_{c,2}}|g_D^{-1}(D)| + \sum_{D\in {\cal D}'_{c,2}}(|g_D^{-1}(D)|+|g_D^{-1}(z(D))|)\\
	\end{align*}
	
	As the sets ${\cal D}_{c,1}$, ${\cal D}_{c,2}$, ${\cal D}'_{c,2}$, and $\{z(D) | D \in {\cal D}'_{c,2}\}$ are disjoint, and as we also have \\ $|g^{-1}(D)| \ge 1$ for all $D$, and $|g^{-1}(D)|\ge 2$ for all $D\in {\cal D}_{c,2}$, we get

	\begin{align*}
		2 fw(G) & \le \sum_{D\in {\cal D}_{c,1}}|g^{-1}(D)| + \sum_{D\in {\cal D}_{c,2}}|g^{-1}(D)| + \sum_{D\in {\cal D}'_{c,2}}(|g^{-1}(D)|+|g^{-1}(z(D))|)  \\
		& \le \sum_{D\in M}|g^{-1}(D)| = 2 fw(O(G))\\
	\end{align*}
	
	which completes the proof.

\end{proof}

We have seen in Section~\ref{sec:delaney-dress} that lopsp-operations are closely related to tilings. We use the connectivity of the associated tiling of a lopsp-operation to
define a property that combines connectivity and the non-existence of small faces that would induce small cuts in a trivial way if e.g.\ the operation {\em dual} is applied.
Theorem~\ref{thm:main} will provide an equivalent characterisation which does not use the associated tiling.

\begin{definition}\label{def:lopsp_connectivity}
  For $k\in \{1,2,3\}$ an lsp-operation resp.\ lopsp-operation $O$ is a \emph{$ck$-lsp-operation}, resp.\  \emph{$ck$-lopsp-operation} if the associated tiling $T_O$ is $k$-connected and all faces have size at least $k$.
 \end{definition}

The following lemma characterises $c2$- and
$c3$-embedded graphs by a condition based on the chamber system. A 4-cycle in a barycentric subdivision is called {\em trivial} if it has a face without a vertex or only a single type-1
vertex in it.

\begin{lemma}\label{lem:conn_characterisation}
	Let \(G\) be an embedded graph. 
	\begin{enumerate}[(i)]
	\item $G$ is $c2$-embedded if and only if $B_G$ has no cycles of length 2.
          \item $G$ is $c3$-embedded if and only if $G$ is $c2$-embedded, and $B_G$ has no nontrivial cycles of length 4
	\end{enumerate}
\end{lemma}

\begin{proof}
	
  {\bf (i)} Let $G$ be an embedded graph and assume that $G$ is not $c2$-embedded.
  There are four possible reasons for not being $c2$-embedded: the existence of a cutvertex, the existence
  of a facial loop (a face of size 1), the existence of a vertex of degree 1, or the existence of a non-contractible 2-cycle in $B_G$.
  The last three immediately imply the existence of a 2-cycle, so assume that $G$ has a
  cutvertex $v$.  If there is a loop in $G$, we also have a 2-cycle, so assume that there are no loops.
  Then $v$ has neighbours in different components, so it has also neighbours $x,y$ in different components, so that $y$ follows $x$ in
  the rotational order around $v$.  The facial walk $(x,v),(v,y),(y,w_1),\dots ,(w_k,x)$ of the face $f$ containing this
  angle must contain $v$ also as one of the $w_i$ as otherwise there would be a path from $x$ to $y$ in $G\setminus\{v\}$. This implies that in the barycentric subdivision there
  are 2 edges between $v$ and the central vertex of $f$ -- a 2-cycle.

  For the other direction assume that there is a 2-cycle $c$. If $c$ is non-contractible, then $fw(G)=1$, so $G$ is not $c2$-embedded and we are done. If $c$ is contractible, then
  assume w.l.o.g.\ that $c$ is {\em innermost}, that is: that it contains no 2-cycle in its interior. If it has a face without a 0-vertex, then it is either a 2-cycle between a 0-
  and a 1-vertex with just a 2-vertex in the interior (which is a facial loop, so $G$ is not $c2$-embedded), or a 2-cycle between a 0- and a 2-vertex with a 1-vertex in the
  interior. As this 1-vertex must have two edges between the 0-vertex in $c$, $c$ was not innermost. If $c$ has no face without a 0-vertex, then every path in $G$ between the
  0-vertices in the two faces using just 2-edges must pass through a vertex of $c$, which contains at most one 0-vertex and one 1-vertex, and no path can use just the 1-vertex on
  $c$ in case there is also a $0$-vertex. So we have just one path crossing $c$, which implies the existence of a cutvertex or a vertex with degree 1, so that again $G$ is not
  $c2$-embedded.

  {\bf (ii)} Let $G$ be an embedded graph and assume that $G$ is not $c3$-embedded. If it also not $c2$-embedded we are done, so assume that $G$ is $c2$-embedded.
  There are four possible reasons for not being $c3$-embedded: the existence of a cutset $\{x,y\}$ of size two, the
  existence of a face of size two, the existence of a vertex with degree 2, or the existence of a non-contractible 4-cycle in $B_G$. Again the last three, as well as double edges forming a non-facial 2-cycle in $G$,
  immediately give a nontrivial 4-cycle in $B_G$, so assume that there are no double edges or loops, but a 2-cut $\{x,y\}$.

  Then both, $x$ and $y$, have neighbours in different components. Let $u\not= y$ be a neighbour of $x$, so that the previous vertex in the rotational order around $x$ is not in
  the same component of $G\setminus \{x,y\}$ as $u$ and let $v$ be the last vertex in the rotational order around $x$ such that all intermediate vertices and $v$ are in the same
  component as $u$. Note that $u=v$ is possible. Then the edges $(x,u)$ and $(v,x)$ belong to different faces $f_1,f_2$ (otherwise we had a cycle of length 2 in $B_G$) and both
  faces also contain $y$ as otherwise also the next, resp. previous neighbour of $x$ would belong to the same component as $u$ and $v$. So $x,f_1,y,f_2$ is a nontrivial cycle of
  length 4.

        For the other direction assume that $G$ is not $c2$-embedded or that $B_G$ has nontrivial cycles of length 4. If it is not $c2$-embedded, we are done, so assume that $c$
        is a nontrivial 4-cycle. If $c$ is non-contractible, we have $fw(G)\le 2$, so $G$ is not $c3$-embedded and we are done. So assume that $c$ is a contractible innermost
        nontrivial 4-cycle. If it has a face without a 0-vertex, it must contain a 2-vertex (as otherwise it would have to contain at least 2 adjacent 1-vertices), but this vertex would
        have to have degree at most 4 (as it can be adjacent to at most two 0-vertices on $c$), which would be a 2-face, so $G$ is not $c2$-embedded. The last case would be that there are 0-vertices
        in both faces. If in one face there is just one 0-vertex, it has degree at most 4 in $B_G$, so degree 2 in $G$, so that $G$ is not $c3$-embedded. Otherwise there are at least 4 vertices and at most two
        paths from vertices on the inside to vertices on the outside crossing $c$, so there is a 1- or 2-cut in $G$, so that again $G$ is not $c3$-embedded.

\end{proof}

Lemma \ref{lem:conn_characterisation} is very useful to determine whether an embedded graph is $c2$- or $c3$-embedded. It will often be used in the following
lemmas and theorems. The main theorem of this last section is Theorem \ref{thm:main}, which shows the equivalence of different definitions of $ck$-lopsp-operations and states that
when applying $ck$-lopsp-operations with $k\in \{1,2,3\}$ to certain embedded graphs, the result is $ck$-embedded. The most difficult
part of its proof is captured in Theorem \ref{thm:cycle_butterfly_2} for $c2$-embeddings and Theorem \ref{thm:cycle_butterfly_3} for $c3$-embeddings.

\begin{lemma}\label{lem:minpath_2}
	Let $O$ be a lopsp-operation with a cut-path $P$ of
	minimal length.
        \begin{description}
          \item[(i)]
        If the vertices of an edge in $O$ are in the same $P_{v_0,v_i}$ for an $i\in\{1,2\}$, then the edge or a parallel edge is also in $P_{v_0,v_i}$.

      \item[(ii)]
        If the vertices of an edge in $O_P$ are in different copies of $P_{v_0,v_1}$, then there is a nontrivial 4-cycle in two copies of $O_P$ sharing their copies of $P_{v_0,v_1}$.
	  
          \end{description}
\end{lemma}
\begin{proof}
	(i): Assume that the vertices $v$ and $w$ of the edge $e$ in $O$ are in $P_{v_0,v_i}$ for an $i\in\{1,2\}$ and that $e$ is not in $P_{v_0,v_i}$. Replacing the
        subpath of $P_{v_0,v_i}$ between $v$ and $w$ by $e$ we get a different cut-path $P'$ in $O$. As the path that is replaced has at least one edge, $P'$ is at most as long as
        $P$. However, $P$ has minimal length, so $P'$ must have the same length as $P$. It follows that there is also an edge $e_P$ between $v$ and $w$ in $P$.
        
        (ii) If $P_{v_0,v_1}$ is $v_0=t_1,\dots ,t_k=v_1$ and we denote one copy with $t_1,\dots ,t_k$ and the other with $t'_1,\dots ,t'_k$, then -- again due to minimality and as
        $O$ has no loops -- such an edge connects w.l.o.g.\ $t_i$ with $t'_{i+1}$ (or the other way around) for some $1\le i <(k-1)$. Considering two copies of $O_P$ sharing the
        copies of $P_{v_0,v_1}$, this gives a 4-cycle $c=t_i,t'_{i+1},t'_i,t_{i+1}$ with $v_1$ in the interior. If $c$ was trivial, then $v_1$ would be a 1-vertex adjacent to all 4
        vertices on $c$ -- also $t_i$ and $t'_i$ -- which contradicts the minimality of $P$.
 \end{proof}

\begin{theorem}\label{thm:cycle_butterfly_2}
  Let $G$ be a $c2$-embedded graph and $O$ a lopsp-operation. Then $O(G)$ is $c2$-embedded if and only if for each cut path $P$ in $O$ we have that there is no 2-cycle in $O_P$.
\end{theorem}

\begin{proof}
  It follows from Lemma \ref{lem:fw_G_leq_fw_O(G)} that $fw(O(G))\geq fw(G) \geq 2$, so we can use Lemma~\ref{lem:conn_characterisation} on both,
  $G$ and $O(G)$.

  The fact that $G$ is $c2$-embedded implies that $B_G$ is a simple graph. Together with the fact that $v_{0,L}$ and $v_{0,R}$ cannot be adjacent, as they have the same type,
  this also
  implies that if two different faces of $D_{G,P}$ share at least two adjacent vertices,
  they share exactly one {\em side} of the double chamber containing these two vertices.
  A {\em side} of a double chamber is a path in a face of $D_{G,P}$ between $v_2$ and $v_{0,L}$, between $v_2$ and $v_{0,R}$, or between $v_{0,L}$ and $v_{0,R}$.
	
	If there is a 2-cycle in $O_P$ for a cut-path $P$ then each copy of $O_P$ inserted into $D_{G,P}$ contains a copy of this 2-cycle which is in
        $B_{O(G)}$. Lemma~\ref{lem:conn_characterisation} now implies that $O(G)$ is not $c2$-embedded.
	
	Conversely, assume that $O(G)$ is not $c2$-embedded. Then there is a 2-cycle $c$ in $B_{O(G)}$. Let $x$ and $y$ be the vertices of $c$ and let $e_1$ and $e_2$ be
        its edges. Let $P$ be a cut-path in $O$ of minimal length. If there exists a face of $D_{G,P}$ that contains both edges of $c$ in the interior or on the boundary, then $c$ is a
        cycle in a copy of $O_P$ and we are done. Now assume that $e_1$ and $e_2$ are in different faces $D_1$ and $D_2$.

        As both faces contain $x$ and $y$, they share a whole side. If they are in the same copy of $P_{v_0,v_i}$ for an $i\in\{1,2\}$, then by Lemma~\ref{lem:minpath_2}~(i)
        a parallel edge $e_0$ is also in that copy of $P_{v_0,v_i}$
        and therefore in both faces, so that the copy of $O_P$ containing an edge in $\{e_1,e_2\}\setminus \{e_0\}$ contains a 2-cycle.

        	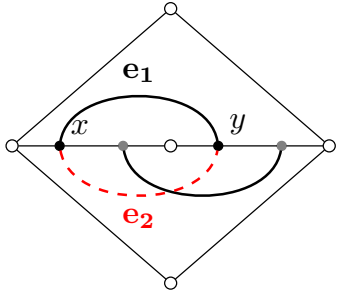
\begin{figure}
		\centering
		\resizebox{.3\textwidth}{!}{\input{jordan_contradiction2.tikz}}
		
		\caption{\label{fig:jcontra} An edge between different copies of $P_{v_0,v_1}$.  }
	\end{figure}

        The last possibility is that $x$ and $y$ are in different copies of $P_{v_0,v_1}$ and as $O$ does not contain loops we have  $\pi(x)\not= \pi(y)$.
        Using the cycle in $D_1$ consisting of $e_1$ and the part of $P_{v_{0,L},v_{0,R}}$  between $x$ and $y$ as a Jordan curve, the copy of this cycle in $D_2$ implies 
        that there cannot be an edge between $x$ and $y$ in $D_2$ -- a contradiction (see Figure~\ref{fig:jcontra}).
   \end{proof}

\begin{lemma}\label{lem:around_v2}
	
	Let $G$ be a $c3$-embedded graph and $O$ a lopsp-operation with a cut-path $P$ of minimal length. Let $v$ be a 2-vertex of $D_G$, and consider the subgraph $S_v$ of $B_{O(G)}$
        consisting of all the edges and vertices in the double chambers with 2-vertex $v$. If there is a nontrivial 4-cycle $c$ in $S_v$, then there is a 2-cycle in one of these chambers or
        $c$ is contained in either only one of these double chambers,
        or in two of the double chambers that share one edge.
	
\end{lemma}

\begin{proof}
	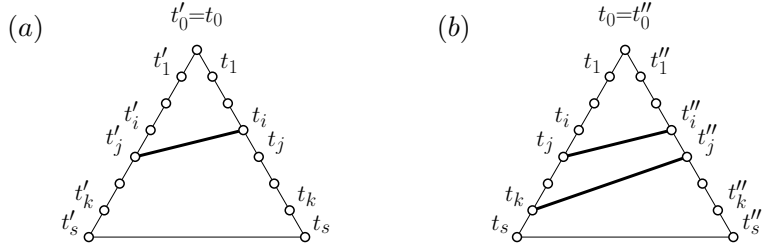
\begin{figure}
		\centering
		\resizebox{.7\textwidth}{!}{\input{jordan_curve_double_chamber.tikz}}
		
		\caption{\label{fig:connectingedge} Double chambers with one or two edges that are copies of edges in a separating cycle.  }
	\end{figure}

	As $G$ is $c3$-embedded, the embedded subgraph of all vertices and edges of $O(G)$ belonging to one of the double chambers containing $v$ is plane
	and the boundary -- that is: the graph formed by the
        vertices and edges lying on the 2-edges of $D_G$ -- is a simple cycle, that we consider to be the boundary of the outer face of the graph $S_v$.
        If $c$ contains only edges on edges of $D_G$, then
        $c$ is the boundary of one double chamber, so assume that $c$ contains at least one edge in the interior of a double chamber.
	
	Assume that there is a double chamber $D$ with 2-vertex $v$ but no subpath of $c$ with edges inside $D$ connecting interior vertices on the paths replacing the two 1-edges.
        Either using the topological notion of a path or subdividing the faces of $D$ again, there is a path $P$ in $D$ connecting $v$ with a (possibly new) vertex on the 2-edge of $D$
        not containing vertices of $c$, except maybe for $v_2$.
        Let $P_e$ be the path that is one of the subdivided 1-edges containing $v$. Connecting the two endpoints of $P\cup P_e$ by a new edge through the outer face of $S_v$, we have a Jordan
        curve that must be
        crossed an even number of times by $c$ using vertices on $e$.  So each such edge $e$ is crossed $0$, $2$, or $4$ times.
        Four crossings would imply that $c$ is contained in the two neighbouring chambers, which would fulfill the requirements, but is due to Lemma~\ref{lem:minpath_2} (i) in fact impossible.
        If no such $e$ is crossed, $c$ is in just one double chamber. If an edge $e$ is crossed twice, due to Lemma~\ref{lem:minpath_2} (i) in case of only one edge in a double chamber
        there is a 2-cycle with one edge on the boundary (so only in one double chamber). If there are two edges connecting the vertices
        on each side, $c$ must in fact be in the two double chambers sharing $e$.
	
	Now we can assume that every double chamber with 2-vertex $v$ has a subpath of $c$ connecting interior vertices of the two subdivided 1-edges. As $G$ is $c3$-embedded, it has no
        faces of size two and $v$ is contained in at least three double chambers, so that there must be three or four double chambers as every double chamber must contain at least
        one edge of the 4-cycle $c$. In each case there are at least two double chambers with only one edge of $c$.  If the vertices on $P_{v_2,v_0}$ in $O$ are -- in this order --
        $t_0,t_1,\dots,t_s$, the edge must be $\{t_i,t_j\}$ with $0< i < j\le s$, as there are no loops in $O$.  In the double chamber patch we will denote the vertices with
        $t_0,t_1,\dots,t_s$, resp. $t'_0,t'_1,\dots,t'_s$ as in Figure~\ref{fig:connectingedge}(a).  Assume w.l.o.g.\ that the situation is as in Figure~\ref{fig:connectingedge}(a)
        and that the vertex with index $j$ is $t'_j$ and not $t_j$.  Due to the Jordan curve theorem applied to the cycle $t_0=t'_0,t'_1,\dots ,t'_j,t_i,t_{i-1},\dots ,t_0$ there
        is no edge $\{t'_k,t_m\}$ with $k<j$ and $m>i$. If the edge $\{t'_j,t''_n\}$ of $c$ sharing $t'_j$ with $\{t_i,t'_j\}$ is also the only edge in the double chamber, this
        argumentation gives $n>j$ and iterating this over all neighbouring double chambers containing only one edge, one gets for the (possibly closed) path $t_i,\dots ,t''_o$ or
        $t_i,\dots ,t''''_o=t_i$ from the first vertex of the first chamber to the last vertex of the last chamber that $o>i$ -- so in case of four double chambers we already have
        our contradiction.
	
	In case of three double chambers, we have a path $\{t_i,t'_j,t''_o\}$ with $i<j<o$ through two double chambers and in the third double chamber -- containing two edges of $c$ --
	we have another path of length two from $t''_o$ to $t_i$. As visualized in Figure~\ref{fig:connectingedge}(b), such a path has to cross two nested Jordan curves and therefore
	would have to have length at least three -- a contradiction.
\end{proof}

\begin{theorem}\label{thm:cycle_butterfly_3}
	
	Let $G$ be a $c3$-embedded graph and $O$ a lopsp-operation. Then $O(G)$ is $c3$-embedded if and only if it is $c2$-embedded and for each cut-path $P$ in $O$ we have that
        there is no nontrivial 4-cycle in a patch of at most two adjacent copies of $O_P$ -- that is:
        copies in double chambers that share a side.
        
\end{theorem}

\begin{proof}
	
	It follows from Lemma \ref{lem:fw_G_leq_fw_O(G)} that $fw(O(G))\geq fw(G) \geq 3$, so we can use Lemma~\ref{lem:conn_characterisation} on both $G$ and $O(G)$. 
	
	The implication that $O(G)$ is not $c3$-embedded if there is a 2- or nontrivial 4-cycle for some cut-path $P$ is obvious, as corresponding pairs of two adjacent copies of $O_P$
        in $O(G)$ would contain such cycles.

    For the other direction we have to prove that if $O(G)$ is not $c3$-embedded, but it is $c2$-embedded, then there is a nontrivial 4-cycle in a patch of two adjacent copies of $O_P$ for some cut-path $P$. 
	
    Let $P$ be a cut-path in $O$ of minimal length. We will refer to the copies of $v_2$ or $v_0$ in the double chamber patches as the corners of the double chamber patches and use
    that two double chambers sharing two corners also share the side between them, as neither $B_G$ nor $G$ have double edges.  Let $c$ be a nontrivial 4-cycle in $O(G)$ and $M$ be
    a set of double chambers in $D_G$ of minimal size, so that the union of all double chamber patches for double chambers in $M$ contains $c$.  This implies that
    $1\le |M|\le 4$.
    By abuse of language we will also use $M$ for the set of corresponding double chamber patches.  First we want to prove that there is a 4-cycle in $D_G$ containing at least
    one edge of every double chamber in $M$, which we will call a saturating 4-cycle.  Suppose the contrary -- that is, that there is no saturating 4-cycle.  As for $|M|\le 2$ the
    boundary of any double chamber in $M$ is a saturating 4-cycle (note that for two double chambers they would have to share a side), we have $|M|\in \{3,4\}$.
	
	\medskip
	
	{\bf Claim 1:} There is no double chamber patch in $M$ sharing exactly one corner with the other double chambers in $M$ and there is no double chamber patch $D\in M$ that
        contains only one edge $e$ of $c$ and shares only one side with the other double chamber patches in $M$.

        \bigskip
	
	{\em Proof of Claim 1:} While the first part is obvious, as $c$ has to pass through the double chamber, the second part is a consequence of Lemma~\ref{lem:minpath_2}: the
        edge $e$ is either between points on the same edge of the double chamber patch or between different copies of $P_{v_0,v_1}$. In the first case Lemma~\ref{lem:minpath_2} implies
        that $e$ is on the boundary of $D$, so that it is also in the double chamber sharing the side with $D$.
        In the second case  Lemma~\ref{lem:minpath_2}  implies a nontrivial 4-cycle in two neighbouring double chambers sharing the side with the 1-vertex, so $|M|\le2$. In each case we have
        a contradiction to the minimality of $M$. \hfill $\square$

	\medskip 
	
	{\bf Claim 2:} There is no double chamber patch $D\in M$ sharing at
	least two sides with the other double chambers in $M$.
      \bigskip
	
	{\em Proof of Claim 2:} 	
	Assume the contrary. Then the two sides have a common vertex, which is a
	0-vertex or a 2-vertex,
	so we have one of the two situations in Figure~\ref{fig:3chambers}. In
	part (a) we can have that $v_0^1=v_0^2$, if the central 2-vertex
	corresponds to a 3-gon, but in
	case (b) all vertices and edges must be distinct as $G$ is $c3$-embedded.
	
	In both cases there must be a fourth double chamber, as otherwise the
	boundary of the central double chamber would be a saturating 4-cycle.
	This implies that
	all double chambers contain exactly one edge of $c$. In case (a) Claim~1
	implies that the fourth chamber must contain $v_0^1$ and $v_0^2$ and as
	a path between the two 0-vertices would require two edges of $c$ (there are no loops in $O$), the chamber must also contain the edges $\{v_0^1,v_2\}$ and $\{v_0^2, v_2\}$. It follows that the
	only possible situation for case (a) is when
	$v_2$ corresponds to a 4-gon and $M$ consists of the double chambers
	containing $v_2$. Then Lemma~\ref{lem:around_v2}
	implies that $M$ was not minimal.
	
	In case (b) Claim 1 implies that there is a fourth chamber containing
	$v_0^1$ and $v_2^2$, but as in this chamber there is an edge connecting
	these vertices, $v_0^1,v_2^1,v_0^2,v_2^2$
	would be a saturating 4-cycle. \hfill $\square$
	
	\medskip
	
	\begin{figure}
          \begin{center}
	    	\resizebox{.75\textwidth}{!}{\input{twoedges2.tikz}\qquad\qquad
		\input{twoedges0.tikz}}
		\caption{\label{fig:3chambers}The two configurations where a double chamber shares two
		  edges. }
          \end{center}
	\end{figure}
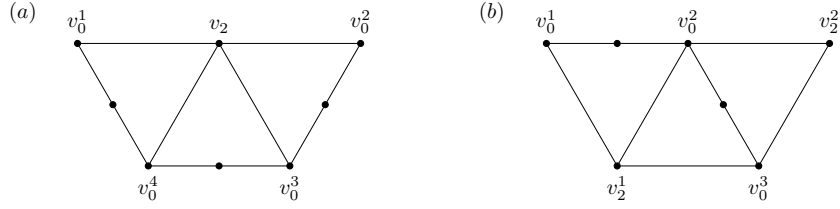

	Let the edges of $c$ be $e_0,\dots ,e_3$ in this cyclic order and
	$d_0,\dots ,d_3$ be the cyclic sequence of double chambers of $M$, so
	that $e_i\in d_i$ for $0\le i <4$. Let $b_0,\dots ,b_k$
	be the sequence where $d_i$ is removed if $d_i=d_{(i+1)mod 4}$. Our
	assumption that no saturating 4-cycle exists implies $k\in \{2,3\}$ as
	these are all elements of $M$. Let now
	$g_0,\dots ,g_r$ be the sequence where two double chambers following
	each other in $b_0,\dots ,b_k$ are joined to one $g_i$ if they share an
	edge. By Claim 2, the $g_i$ are a single double chamber
	or a pair of two double chambers sharing an edge
	and each $g_i$ shares a vertex with $g_{(i+1)mod (r+1)}$ and a vertex
	with $g_{(i-1)mod (r+1)}$. Using the fact that there is no edge between
	the two 0-vertices in a double chamber patch,
	it follows that for each $g_i$ there is a path along the
	edges of its double chambers
	between the two connecting vertices
	that has an edge of every double chamber in $g_i$ and has at most as many edges
	of $D_G$ as $g_i$ contains edges of $c$. As edges in different
	$g_i$ are different and as there are no double edges in $B_G$, we get a
	saturating cycle by connecting these parts. Note that though our
	argumentation gives that we have a cycle of length
	at most 4, it must have length exactly 4 as no shorter cycles exist.

	\medskip
	
	So in each case our assumption that no saturating cycle exists, leads to
	a contradiction and we can assume the existence of a saturating cycle $c_s$.
	As $G$ is $c3$-embedded, due to Lemma~\ref{lem:conn_characterisation}, $c_s$ must be the boundary of a double chamber,
        or two double chambers sharing the 1-vertex. 
	
	If $e$ is the 1-vertex in these one or two double chambers, then $c$ is contained in the set $N_e$ consisting of all six double chambers that share an edge with a double
        chamber containing $e$. We say that the two double chambers with 1-vertex $e$ are the central double chambers, and the four other double chambers in $N_e$ are the extremal
        double chambers. The three possible configurations of $N_e$ are shown in Figure~\ref{fig:butterflies}.
	
	\begin{figure}
		\begin{center}
			\resizebox{.7\textwidth}{!}{\input{butterfly.tikz}
			\quad
			\input{butterfly_1closed.tikz}
			\quad
			\input{butterfly_2closed.tikz}}
		\end{center}
		\caption{\label{fig:butterflies} The three possible configurations of the six double chambers in the set $N_e$ are shown here. Different vertices in the drawing
                  represent different vertices of $D_G$.}
	\end{figure}
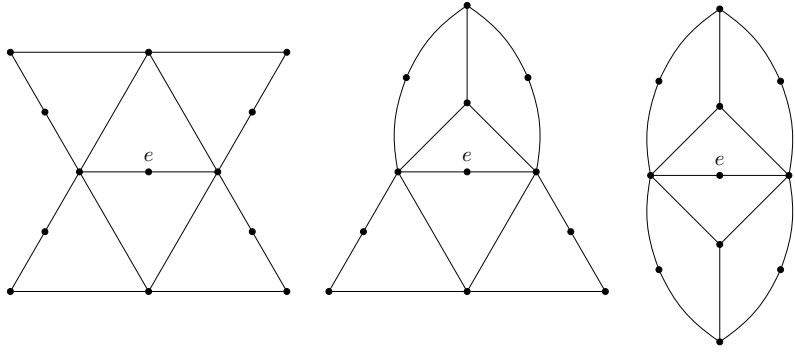

        If $c$ lies entirely in the chambers containing the same 2-vertex of $N_e$, we can apply Lemma~\ref{lem:around_v2} to conclude that $c$ is contained in at most 2 adjacent
        chambers.  Otherwise Lemma~\ref{lem:minpath_2}
        implies that there are two edges in each set of double chambers sharing a 2-vertex. As
        each external chamber shares just one vertex with the edge between the two central double chambers, the only possibility (again using Lemma~\ref{lem:minpath_2})
        that at least one of these two edges is not
        contained in the central chamber is that a 2-vertex $f$ in $N_e$ corresponds to a face of size three in $G$, as in the second and third configuration of
        Figure~\ref{fig:butterflies} and that both edges go from a 0-vertex to an internal vertex of the edge shared by the two external chambers.
        Let $e_1$ and $e_2$ be the edges of $c$ in the two external chambers.
        But then $e_1$ would correspond to an edge in $O_P$ from w.l.o.g. $v_{0,R}$ to an internal vertex of a copy of $P_{(v_{0,L}),v_2}$ and $e_2$ would
        correspond to an edge from $v_{0,L}$ to a copy of $P_{(v_{0,R})
          ,v_2}$, which would imply crossing edges in $O_P$.
        So if $c$ is not contained in two double chambers sharing a 2-vertex, it is contained in the two central double chambers.
\end{proof}

For a $ck$-lopsp-operation $O$ and an embedded graph $G$ with minimum face size at least $k$ and minimum degree at least $k$ (but not necessarily $ck$-embedded),
the graph $O(G)$ has also minimum face
  size at least $k$ and minimum degree at least $k$. For the vertices of $B_{O(G)}$ of type 0 or 2 that are not on 0-vertices or 2-vertices of $D_G$, the fact that they have degree
  at least $2k$ follows from $O$ being a $ck$-lopsp-operation, as these degrees also occur in the tiling $T_O$. For the others it follows from the degrees of 0-vertices and
  2-vertices in $D_G$. In case $G$ is also $ck$-embedded, we have a stronger result:

\begin{theorem}\label{thm:main}
  The following statements are equivalent for $k\in\{1,2,3\}$ and a lopsp-operation $O$:

  \begin{description}
    \item[(a)] $O$ is a $ck$-lopsp-operation 
    \item[(b)] For all $ck$-embedded graphs $G$, $O(G)$ is $ck$-embedded.
    \item[(c)] There exists a $ck$-embedded graph $G$ such that $O(G)$ is $ck$-embedded.
      \end{description}
\end{theorem}
\begin{proof}

  \fbox{$(a)\Rightarrow (b)$}
  
	 For $k=1$ this is trivial. Assume that there is a $c2$-embedded graph $G$ such that $O(G)$ is not $c2$-embedded.  By Theorem~\ref{thm:cycle_butterfly_2} there is a 2-cycle
         in $O_P$ for some cut-path $P$ in $O$. This cycle induces a cycle in $B_{T_O}$, which implies a 1-cut, or a face of size 1 in the tiling $T_O$.  It
         follows that $O$ is not a $c2$-operation -- a contradiction. Similarly, if there is a $c3$-embedded graph $G$ such that $O(G)$ is not $c3$-embedded, we find a 2-cut in
         $T_O$ using the cycle from Theorem~\ref{thm:cycle_butterfly_3}.
	
	\fbox{$(b)\Rightarrow (a)$} For $k=1$ this is trivial.

        Note that due to the definition of $m_{01},m_{02}$ and $m_{12}$ for the vertices $v_0,v_1$, and $v_2$, the lopsp-operation $O_{id}$ representing the identity corresponds to
        the hexagonal tiling of the plane, so that $T_O$ is just the tiling obtained by inserting copies of $O_P$ (for some $P$)
        into the double chamber system of the hexagonal tiling $T_{O_{id}}$. Let us call
        the graph formed by the subdivided 2-edges of this double chamber system the hexagonal skeleton.

        Assume now that $O$ is not a $c2$-lopsp operation. Then there is a face $f$ of size $1$ or a 1-cut
        $\{x\}$ in $T_O$. In case of a 1-cut, at least one of the components of $T_O\setminus \{x\}$, say $C_0$, must be finite. Let $C$ denote a
        finite subgraph of the hexagonal skeleton of $T_O$ containing $f$, resp.\  $C_0$ together with the cut vertex $x$, in bounded faces.
        Using Goldberg-Coxeter operations (see \cite{gocox} or \cite{CK62})
        with sufficiently large parameters to construct large icosahedral fullerenes, we get a fullerene $F$, that is $c3$-embedded and contains an isomorphic copy of $C$ with the
        paths between the 0-vertices replaced by edges. Applying $O$ to that fullerene, we get a subgraph $S$ of $O(F)$ that has a face $f'$ of size 1 or that is isomorphic to $C_0$
        and where all vertices corresponding to
        vertices of $C_0$ have -- except for the vertex $x'$ corresponding to $x$ -- only neighbours in $S$. So $f'$ or the vertex $x'$, which is a cut-vertex of $O(F)$, are contradictions to
        the assumption.
	
	The case $k=3$ is completely analogous, with also a 2-face and a 2-cut in the argument.
	
	\fbox{$(b)\Rightarrow (c)$} Trivial.
	
	\fbox{$(c)\Rightarrow (b)$} Note that the conditions in Theorems~\ref{thm:cycle_butterfly_2} and \ref{thm:cycle_butterfly_3} are -- except for being $ck$-embedded --
        independent of $G$, as $O_P$ and the union of two copies of $O_P$ sharing a side are the same for all these $G$. This implies that if $O(G)$ is $ck$-embedded for some
        $ck$-embedded graph $G$, then $O(H)$ is $ck$-embedded for any $ck$-embedded graph $H$.
\end{proof}

One of the main results of this paper, Theorem~\ref{thm:pol_lopsp}, now follows from Theorem~\ref{thm:main} and Lemma~\ref{lem:fw_G_leq_fw_O(G)}.
\setcounter{theorem}{1}

\begin{theorem}
	If $G$ is a polyhedral embedding of a graph and $O$ is a  $c3$-lsp- or $c3$-lopsp-operation, then $O(G)$ is also a polyhedral embedding.
\end{theorem}

\section{Future work}
In the last section of \cite{gocox} many open problems are described. They are sometimes just formulated for lsp-operations, but are often as relevant and interesting for
lopsp-operations, so we refer the reader to \cite{gocox}.
While lsp-operations applied to polyhedra seem (this is still an open problem) to increase the symmetry at most by a factor of two, lopsp-operations can increase the
symmetry much more, as can be seen when e.g.\ applying gyro to the tetrahedron. Classifying lopsp-operations that can introduce new symmetries would be
an interesting problem, but maybe even more difficult than solving the problem for lsp-operations.

We knew that there is at least one lsp-operation (the dual) that does not always preserve 3-connectedness for embedded graphs, if the face-width is at most two \cite{dualconnectivity}, so
an obvious question was which other operations do not always preserve 3-connectedness. This was answered for lsp-operations in \cite{heidi_thesis}, where the class of such operations,
called {\em edge-breaking operations}, was characterized. Heidi Van den Camp is currently working on extending these results to lopsp-operations.

Another problem mentioned in \cite{gocox} -- the generation of lsp-operations for a given inflation factor -- has been solved \cite{goetschalckx2020generation}. Such an algorithm
not only allows the generation of lsp-operations, but also the generation of polyhedra and other embedded graphs with some specific symmetry groups of the embedding. For
lopsp-operations however such an algorithm does not yet exist and is therefore the topic of future research.

So some of the problems in \cite{gocox} have been solved in the meantime, but most of them are still open.

		\bibliographystyle{plain}
		\bibliography{../literatur}
\end{document}

%% file: NoBridges.tikz
\begin{tikzpicture}[scale=0.4]
\tikzset{every node/.style={shape=circle, draw=blue, scale=0.5}}
\tikzset{face/.style={->, thick, draw=red}}
\tikzset{red/.style={thick, draw=red}}
\tikzset{rednode/.style={fill=red, draw=red}}

\node (v5) at (1,2) {};
\node (v6) at (5,0.5) {};
\node (v1) at (3.5,-2) {};
\node (v2) at (2,-4) {};
\node (v3) at (-4,-3) {};
\node (v4) at (-6,0.5) {};
\node (v8) at (1,0.7) {};
\node(v7) at (-3,0.5) {};
\node (v9) at (-3.5,2.5) {};
\node (v10) at (-1,-0.5) {};
\node (v11) at (-2,-2) {};

\draw[very thick] (v1) -- (v2) -- (v3) -- (v4) -- (v9) -- (v5) -- (v6) -- (v1);
\draw[very thick] (v1) -- (v10) -- (v11) -- (v1);

\draw (v11) edge (v2);

\begin{scope}
\draw (v4) -- (v7) -- (v8) -- (v1);
\draw  (v8) edge (v5);
\draw  (v7) edge (v9);
\draw  (v7) edge (v3);
\end{scope}

\end{tikzpicture}

%% file: InternalComponent.tikz
\begin{tikzpicture}[scale=0.4]
\tikzset{every node/.style={shape=circle, draw=blue, scale=0.5}}
\tikzset{face/.style={->, thick, draw=red}}
\tikzset{red/.style={thick, draw=red}}
\tikzset{rednode/.style={fill=red, draw=red}}

\node (v5) at (1,2) {};
\node (v6) at (5,0.5) {};
\node (v1) at (3.5,-2) {};
\node (v1') at (3.5,-1.3) {};
\node (v2) at (2,-4) {};
\node (v3) at (-4,-3) {};
\node (v4) at (-6,0.5) {};
\node (v8) at (1,0.7) {};
\node(v7) at (-3,0.5) {};
\node (v9) at (-3.5,2.5) {};
\node (v10) at (-1,-0.5) {};
\node (v11) at (-2,-2) {};

\draw[very thick] (v1) -- (v2) -- (v3) -- (v4) -- (v9) -- (v5) -- (v6) -- (v1')-- (v10) -- (v11) -- (v1);

\draw (v11) edge (v2);

\begin{scope}
\draw (v4) -- (v7) -- (v8) -- (v1');
\draw  (v8) edge (v5);
\draw  (v7) edge (v9);
\draw  (v7) edge (v3);
\end{scope}

\end{tikzpicture}

%% file: Hexagon.tikz
\begin{tikzpicture}[scale=1.1]
\tikzset{normal/.style={shape=circle, draw=black, scale=0.4, fill=black}}
\tikzset{type2/.style={shape=circle, draw=red, scale=0.5}}
\tikzset{noNode/.style={draw=none}}
\tikzset{0edge/.style={draw=red}}
\node[normal] (v1) at (-1,0) {};
\node[normal] (v4) at (1,0) {};
\node[normal] (v6) at ({cos(120)}, {-sin(120)}) {};
\node[normal] (v3) at ({cos(60)}, {sin(60)}) {};
\node[normal] (v5) at ({cos(60)}, {-sin(60)}) {};
\node[normal] (v2) at ({cos(120)}, {sin(120)}) {};

\node [noNode] (v7) at (-2, 0) {};
\node [noNode] (v8) at ({2*cos(120)}, {2*sin(120)}) {};
\node [noNode] (v9) at ({2*cos(60)}, {2*sin(60)}) {};
\node [noNode] (v10) at (2, 0) {};
\node [noNode] (v11) at ({2*cos(60)}, {-2*sin(60)}) {};
\node [noNode] (v12) at ({2*cos(120)}, {-2*sin(120)}) {};
\begin{scope}[thick]
\draw (v1) edge (v7);
\draw (v2) edge (v8);
\draw (v3) edge (v9);
\draw (v4) edge (v10);
\draw (v5) edge (v11);
\draw (v6) edge (v12);
\draw (v1) -- (v2) -- (v3) -- (v4) -- (v5) -- (v6) -- (v1);
\end{scope}
\end{tikzpicture}

%% file: BaryHexagon2.tikz
\begin{tikzpicture}[scale=1.1]
\tikzset{normal/.style={shape=circle, draw=red, scale=0.4,fill=red}}
\tikzset{type1/.style={shape=circle, draw=black!30!green, scale=0.3,fill=black!30!green}}
\tikzset{type2/.style={shape=circle, draw=black, scale=0.3,fill=black}}
\tikzset{0edge/.style={draw=red, thin}}
\tikzset{1edge/.style={draw=black!40!green}}
\tikzset{2edge/.style={very thick}}

\node[normal] (v1) at (-1,0) {};
\node[normal] (v4) at (1,0) {};
\node[normal] (v6) at ({cos(120)}, {-sin(120)}) {};
\node[normal] (v3) at ({cos(60)}, {sin(60)}) {};
\node[normal] (v5) at ({cos(60)}, {-sin(60)}) {};
\node[normal] (v2) at ({cos(120)}, {sin(120)}) {};

\node[type1] (e1) at ({(cos(120)-1)/2}, {sin(60)/2}) {};
\node[type2] (f1) at ({(cos(120)-1)}, {sin(60)}) {};
\node[type1] (e2) at (0, {sin(60)}) {};
\node[type2] (f2) at (0, {2*sin(60)}) {};
\node[type1] (e3) at({(1 - cos(120))/2}, {sin(60)/2}) {};
\node[type2] (f3) at ({(-cos(120)+1)}, {sin(60)}) {};
\node[type1] (e4) at({(1 - cos(120))/2}, {-sin(60)/2}) {};
\node[type2] (f4) at ({(-cos(120)+1)}, {-sin(60)}) {};
\node[type1] (e5) at({(0}, {-sin(60)}) {};
\node[type2] (f5) at (0, {-2*sin(60)}) {};
\node[type1] (e6) at ({(cos(120)-1)/2}, {-sin(60)/2}) {};
\node[type2] (f6) at ({(cos(120)-1)}, {-sin(60)}) {};

\node[type2] (f) at (0, 0) {};

\node [type1] (v7) at (-2, 0) {};
\node [type1] (v8) at ({2*cos(120)}, {2*sin(120)}) {};
\node [type1] (v9) at ({2*cos(60)}, {2*sin(60)}) {};
\node [type1] (v10) at (2, 0) {};
\node [type1] (v11) at ({2*cos(60)}, {-2*sin(60)}) {};
\node [type1] (v12) at ({2*cos(120)}, {-2*sin(120)}) {};

\begin{scope}[0edge]
\draw (f) -- (e1)--(f1);
\draw (f) -- (e2)--(f2);
\draw (f) -- (e3)--(f3);
\draw (f) -- (e4)--(f4);
\draw (f) -- (e5)--(f5);
\draw (f) -- (e6)--(f6);
\draw (f1)--(v8)--(f2)--(v9)--(f3)--(v10)--(f4)--(v11)--(f5)--(v12)--(f6)--(v7)--(f1);
\end{scope}

\begin{scope}[1edge]
\draw (f) edge (v1);
\draw (f) edge (v2);
\draw (f) edge (v3);
\draw (f) edge (v4);
\draw (f) edge (v5);
\draw (f) edge (v6);
\draw (v1)--(f1)--(v2)--(f2)--(v3)--(f3)--(v4)--(f4)--(v5)--(f5)--(v6)--(f6)--(v1);
\end{scope}

\begin{scope}[2edge]
\draw (v1) -- (e1) -- (v2) -- (e2) -- (v3) -- (e3) -- (v4) -- (e4) -- (v5) -- (e5) -- (v6) -- (e6) -- (v1);
\draw  (v1) edge (v7);
\draw  (v2) edge (v8);
\draw (v3) edge (v9);
\draw (v4) edge (v10);
\draw (v5) edge (v11);
\draw (v6) edge (v12);
\end{scope}

\end{tikzpicture}

%% file: bary_hexagon_small.tikz
\begin{tikzpicture}[scale=2]
\tikzset{type0/.style={shape=circle, draw=red, scale=0.4, fill=white}}
\tikzset{type1/.style={shape=circle, draw=black!30!green, scale=0.4, fill=white}}
\tikzset{type2/.style={shape=circle, draw=black, scale=0.4, fill=white}}
\tikzset{0edge/.style={draw=red, very thin}}
\tikzset{1edge/.style={draw=black!40!green, very thin}}
\tikzset{2edge/.style={very thick}}

\fill[blue!10!white] (0,0) -- (0, {cos(15)}) -- ({cos(15)*tan(30)}, {cos(15)}) -- cycle;

\node[type2] (v2) at (0, 0) {};

\node[type1] (v1) at (0, {cos(15)}) {};
\node[type1] (v11) at ({cos(15) *cos(30)}, {cos(15)*sin(30)}) {};
\node[type1] (v12) at ({cos(15) *cos(30)}, {-cos(15)*sin(30)}) {};
\node[type1] (v13) at (0, {-cos(15)}) {};
\node[type1] (v14) at ({-cos(15) *cos(30)}, {-cos(15)*sin(30)}) {};
\node[type1] (v15) at ({-cos(15) *cos(30)}, {cos(15)*sin(30)}) {};

\node[type0] (v0) at ({cos(15)*tan(30)}, {cos(15)}) {};
\node[type0] (v01) at ({cos(15)/sin(60)}, 0) {};
\node[type0] (v02) at ({cos(15)*tan(30)}, {-cos(15)}) {};
\node[type0] (v03) at ({-cos(15)*tan(30)}, {-cos(15)}) {};
\node[type0] (v04) at ({-cos(15)/sin(60)}, 0) {};
\node[type0] (v05) at ({-cos(15)*tan(30)}, {cos(15)}) {};

\begin{scope}[0edge]
\draw[ultra thick] (v1) -- (v2) ;
\draw (v2) -- (v11);
\draw (v12) -- (v2) -- (v13);
\draw (v14) -- (v2) -- (v15);
\end{scope}

\begin{scope}[2edge]
\draw[ultra thick] (v1) -- (v0);
\draw (v0) --(v11)--(v01) -- (v12) -- (v02) -- (v13) 
-- (v03) -- (v14) -- (v04) -- (v15) -- (v05) -- (v1);
\end{scope}

\begin{scope}[1edge]
\draw[ultra thick] (v0) -- (v2) ;
\draw (v2) -- (v01);
\draw (v02) -- (v2) -- (v03);
\draw (v04) -- (v2) -- (v05);
\end{scope}

\end{tikzpicture}

%% file: deco_truncation_old.tikz
\begin{tikzpicture}[scale = 3]
\tikzset{type0/.style={shape=circle, draw=red, scale=0.4, fill=white}}
\tikzset{type1/.style={shape=circle, draw=black!30!green, scale=0.4, fill=white}}
\tikzset{type2/.style={shape=circle, draw=black, scale=0.4, fill=white}}
\tikzset{noNode/.style={draw=none}}
\tikzset{0edge/.style={draw=red, thin}}
\tikzset{1edge/.style={draw=black!40!green, thin}}
\tikzset{2edge/.style={very thick}}

\node[type2] (v2) at (0, 1) {};
\node[type1] (v1) at (0, 0) {};
\node[type2] (v0) at ({tan(30)}, 0) {};
\node[type1] (e) at ({sin(30)},{1 - cos(30)}) {};
\node[type0] (v) at ({tan(15)}, {0})  {};

\begin{scope}[0edge]
\draw (v1) edge (v2);
\draw (e) edge (v2);
\draw (v0) edge (e);
\end{scope}

\begin{scope}[1edge]
\draw (v) edge (v0);
\draw (v) edge (v2);
\end{scope}

\begin{scope}[2edge]
\draw (v1) edge (v);
\draw (e) edge (v);
\end{scope}

\node[label=150:{$v_2$}] at (0, 1) {};
\node[label=190:$v_1$] at (0, 0) {};
\node[label={-10:$v_0$}] at ({tan(30)}, 0) {};
\end{tikzpicture}

%% file: applied_truncation.tikz
\begin{tikzpicture}[scale=2]
\tikzset{type0/.style={shape=circle, draw=red, scale=0.4, fill=white}}
\tikzset{type1/.style={shape=circle, draw=black!30!green, scale=0.4, fill=white}}
\tikzset{type2/.style={shape=circle, draw=black, scale=0.4, fill=white}}
\tikzset{0edge/.style={draw=red, very thin}}
\tikzset{1edge/.style={draw=black!40!green, very thin}}
\tikzset{2edge/.style={very thick}}

\fill[blue!10!white] (0,0) -- (0, {cos(15)}) -- ({cos(15)*tan(30)}, {cos(15)}) -- cycle;

\node[type2] (v2) at (0, 0) {};

\node[type1] (v1) at (0, {cos(15)}) {};
\node[type1] (v11) at ({cos(15) *cos(30)}, {cos(15)*sin(30)}) {};
\node[type1] (v12) at ({cos(15) *cos(30)}, {-cos(15)*sin(30)}) {};
\node[type1] (v13) at (0, {-cos(15)}) {};
\node[type1] (v14) at ({-cos(15) *cos(30)}, {-cos(15)*sin(30)}) {};
\node[type1] (v15) at ({-cos(15) *cos(30)}, {cos(15)*sin(30)}) {};

\node[type2] (v0) at ({cos(15)*tan(30)}, {cos(15)}) {};
\node[type2] (v01) at ({cos(15)/sin(60)}, 0) {};
\node[type2] (v02) at ({cos(15)*tan(30)}, {-cos(15)}) {};
\node[type2] (v03) at ({-cos(15)*tan(30)}, {-cos(15)}) {};
\node[type2] (v04) at ({-cos(15)/sin(60)}, 0) {};
\node[type2] (v05) at ({-cos(15)*tan(30)}, {cos(15)}) {};

\node[type1] (e) at ({cos(15)*sin(30)},{cos(15)*cos(30)}) {};
\node[type1] (e1) at ({cos(15)},0) {};
\node[type1] (e2) at ({cos(15)*sin(30)},{-cos(15)*cos(30)}) {};
\node[type1] (e3) at ({-cos(15)*sin(30)},{-cos(15)*cos(30)}) {};
\node[type1] (e4) at ({-cos(15)},0) {};
\node[type1] (e5) at ({-cos(15)*sin(30)},{cos(15)*cos(30)}) {};

\node[type0] (v) at ({sin(15)}, {cos(15)})  {};
\node[type0] (v') at ({cos(45)}, {sin(45)})  {};
\node[type0] (v_1) at ({cos(15)}, {sin(15)})  {};
\node[type0] (v_1') at ({cos(15)}, {-sin(15)})  {};
\node[type0] (v_2) at ({cos(45)}, {-sin(45)})  {};
\node[type0] (v_2') at ({sin(15)}, {-cos(15)})  {};
\node[type0] (v_3) at ({-sin(15)}, {-cos(15)})  {};
\node[type0] (v_3') at ({-cos(45)}, {-sin(45)})  {};
\node[type0] (v_4') at ({-cos(15)}, {sin(15)})  {};
\node[type0] (v_4) at ({-cos(15)}, {-sin(15)})  {};
\node[type0] (v_5') at ({-sin(15)}, {cos(15)})  {};
\node[type0] (v_5) at ({-cos(45)}, {sin(45)})  {};

\begin{scope}[0edge]
\draw[very thick] (v0) -- (e) -- (v2) -- (v1);
\draw (v01) -- (e1) -- (v2) -- (v11);
\draw (v02) -- (e2) -- (v2) -- (v12);
\draw (v03) -- (e3) -- (v2) -- (v13);
\draw (v04) -- (e4) -- (v2) -- (v14);
\draw (v05) -- (e5) -- (v2) -- (v15);
\end{scope}

\begin{scope}[2edge]
\draw[ultra thick] (v1) --(v) -- (e);
\draw (v11) --(v') -- (e);
\draw (v11) --(v_1) -- (e1);
\draw (v12) --(v_1') -- (e1);
\draw (v12) --(v_2) -- (e2);
\draw (v13) --(v_2') -- (e2);
\draw (v13) --(v_3) -- (e3);
\draw (v14) --(v_3') -- (e3);
\draw (v14) --(v_4) -- (e4);
\draw (v15) --(v_4') -- (e4);
\draw (v15) --(v_5) -- (e5);
\draw (v1) --(v_5') -- (e5);
\end{scope}

\begin{scope}[1edge]
\draw[very thick] (v2) -- (v) -- (v0);
\draw (v2) -- (v') -- (v0);
\draw (v2) -- (v_1) -- (v01);
\draw (v2) -- (v_1') -- (v01);
\draw (v2) -- (v_2) -- (v02);
\draw (v2) -- (v_2') -- (v02);
\draw (v2) -- (v_3) -- (v03);
\draw (v2) -- (v_3') -- (v03);
\draw (v2) -- (v_4) -- (v04);
\draw (v2) -- (v_4') -- (v04);
\draw (v2) -- (v_5) -- (v05);
\draw (v2) -- (v_5') -- (v05);
\end{scope}

\end{tikzpicture}

%% file: gyro_lopsp_path.tikz
\begin{tikzpicture}[scale=2]
\tikzset{type0/.style={shape=circle, draw=red, scale=0.4, fill=white}}
\tikzset{type1/.style={shape=circle, draw=black!30!green, scale=0.4, fill=white}}
\tikzset{type2/.style={shape=circle, draw=black, scale=0.4, fill=white}}
\tikzset{t0/.style={draw=red, thin}}
\tikzset{t1/.style={draw=black!40!green, thin}}
\tikzset{t2/.style={thick}}

\node[type0, fill=red, scale=1.3,  label={[label distance=-0.02cm]90:{${v_0}$}}] (v0) at (0.5, 0) {};
\node[type1, fill=black!30!green, scale=1.3, label={[label distance=-0.06cm]-90:{$v_1$}}] (v1) at (0.25, -0.5) {};
\node[type0,  fill=red, scale=1.3, label={[label distance=-0.02cm]90:{$v_2$}}] (v2) at (0, 1) {};
\node[type1] (e) at (0, 0) {};
\node[type0] (v) at (-0.5, 0) {};
\node[type2] (f) at (1, 0) {};
\node[type1] (E) at (-1, 0) {};

\begin{scope}[t0]
\draw (f)  edge[bend left=90, looseness=1.5] (E) ;
\draw (f)  edge[bend right=85] (E) ;
\draw (f)  edge[bend left=50] (e) ;
\draw (e)  edge[bend left=50] (f) ;
\draw (f)  edge[bend left=30] (v1) ;
\end{scope}

\begin{scope}[t1]
\draw[ultra thick, dashed] (v0) -- (f);
\draw[ultra thick, dashed] (f)  edge[bend right=45] (v2) ;
\draw (v)  edge[bend left=70] (f) ;
\draw (v)  edge[bend right=80, looseness=1.5] (f) ;
\draw (v)  edge[bend right=55, looseness=1] (f) ;
\end{scope}

\begin{scope}[t2]
\draw[ultra thick, dashed] (v0) -- (e) -- (v) ;
\draw (v) -- (E);
\draw (v2)  edge[bend right=45] (E) ;
\draw[ultra thick, dashed] (v1)  edge[bend left=30] (v) ;
\end{scope}

\end{tikzpicture}

%% file: gyro_patch.tikz
\begin{tikzpicture}[scale=3.7]
\tikzset{type0/.style={shape=circle, draw=red, scale=0.4, fill=white}}
\tikzset{type1/.style={shape=circle, draw=black!30!green, scale=0.4, fill=white}}
\tikzset{type2/.style={shape=circle, draw=black, scale=0.4, fill=white}}
\tikzset{t0/.style={draw=red, thin}}
\tikzset{t1/.style={draw=black!40!green, thin}}
\tikzset{t2/.style={thick}}

\node[type0, fill=red, scale=1.3,  label={[label distance=-0.02cm]90:{$v_2$}}] (v) at (0, 0) {};

%vertices

\node[type1] (e14) at ({cos(50+4*60)*cos(15)/cos(20)}, {sin(50+4*60)*cos(15)/cos(20)}) {};
\node[type0] (v14) at ({cos(40+4*60)*cos(15)/cos(10)}, {sin(40+4*60)*cos(15)/cos(10)}) {};%middenlinks
\node[type1, fill=green!70!black, scale=1.3, label={[label distance=-0.02cm]-90:{$v_1$}}] (e24) at ({cos(30+4*60)*cos(15)}, {sin(30+4*60)*cos(15)}) {};
\node[type0] (v24) at ({cos(20+4*60)*cos(15)/cos(10)}, {sin(20+4*60)*cos(15)/cos(10)}) {};%middenrechts
\node[type1] (e34) at ({cos(10+4*60)*cos(15)/cos(20)}, {sin(10+4*60)*cos(15)/cos(20)}) {};
\node[type0,  fill=red, scale=1.3,  label={[label distance=-0.02cm]120:{$v_{0,L}$}}] (v34) at ({cos(4*60)*(cos(15)/cos(30))}, {sin(4*60)*(cos(15)/cos(30))}) {};%rechts
\node[type1] (e44) at ({cos(40+4*60)*cos(15)/cos(10)/2}, {sin(40+4*60)*cos(15)/cos(10)/2}) {};
\node[type2] (f4) at ({cos(4*60)*(cos(15)/cos(30))/2}, {sin(4*60)*(cos(15)/cos(30))/2}) {};

\node[type2] (f5) at ({cos(5*60)*(cos(15)/cos(30))/2}, {sin(5*60)*(cos(15)/cos(30))/2}) {};
\node[type0, fill=red, scale=1.3,  label={[label distance=-0.02cm]60:{$v_{0,R}$}}] (v35) at ({cos(5*60)*(cos(15)/cos(30))}, {sin(300)*(cos(15)/cos(30))}) {};%rechts
%edges

\draw[t2,ultra thick, dotted] (v35) -- (e14) -- (v14) --(e24) ;
\draw[t2,ultra thick, dashed] (e24) -- (v24) -- (e34) -- (v34) ;
\draw[t2] (v14) -- (e44)--(v);
\draw[t1] (v14) -- (f4) -- (v24);
\draw[t1,ultra thick, dashed] (v) -- (f4) -- (v34);
\draw[t0] (e24) -- (f4) -- (e34);
\draw[t0] (f4) -- (e44);
\draw[t0] (e14) -- (f5) -- (e44);
\draw[t1] (v14) -- (f5);

\draw[t1,ultra thick, dotted] (v) -- (f5) -- (v35);

\node[draw=white] (k) at (0.5,-1.1) {};

\end{tikzpicture}

%% file: bary_hexagon_double.tikz
\begin{tikzpicture}[scale=1.8]
\tikzset{type0/.style={shape=circle, draw=red, scale=0.4, fill=white}}
\tikzset{type1/.style={shape=circle, draw=black!30!green, scale=0.4, fill=white}}
\tikzset{type2/.style={shape=circle, draw=black, scale=0.4, fill=white}}
\tikzset{0edge/.style={draw=red, thin}}
\tikzset{1edge/.style={draw=black!40!green, thin}}
\tikzset{2edge/.style={very thick}}

\fill[blue!10!white] (0,0) -- ({-cos(15)*tan(30)}, {cos(15)})-- (0, {cos(15)}) -- ({cos(15)*tan(30)}, {cos(15)}) -- cycle;

\node[type2] (v2) at (0, 0) {};

\node[type1] (v1) at (0, {cos(15)}) {};
\node[type1] (v11) at ({cos(15) *cos(30)}, {cos(15)*sin(30)}) {};
\node[type1] (v12) at ({cos(15) *cos(30)}, {-cos(15)*sin(30)}) {};
\node[type1] (v13) at (0, {-cos(15)}) {};
\node[type1] (v14) at ({-cos(15) *cos(30)}, {-cos(15)*sin(30)}) {};
\node[type1] (v15) at ({-cos(15) *cos(30)}, {cos(15)*sin(30)}) {};

\node[type0] (v0) at ({cos(15)*tan(30)}, {cos(15)}) {};
\node[type0] (v01) at ({cos(15)/sin(60)}, 0) {};
\node[type0] (v02) at ({cos(15)*tan(30)}, {-cos(15)}) {};
\node[type0] (v03) at ({-cos(15)*tan(30)}, {-cos(15)}) {};
\node[type0] (v04) at ({-cos(15)/sin(60)}, 0) {};
\node[type0] (v05) at ({-cos(15)*tan(30)}, {cos(15)}) {};

\begin{scope}[0edge]
\draw[ultra thick] (v1) -- (v2) ;
\draw (v2) -- (v11);
\draw (v12) -- (v2) -- (v13);
\draw (v14) -- (v2) -- (v15);
\end{scope}

\begin{scope}[2edge]
\draw[ultra thick] (v05)--(v1) -- (v0);
\draw (v0) --(v11)--(v01) -- (v12) -- (v02) -- (v13) 
-- (v03) -- (v14) -- (v04) -- (v15) -- (v05);
\end{scope}

\begin{scope}[1edge]
\draw[ultra thick] (v0) -- (v2) --(v05);
\draw (v2) -- (v01);
\draw (v02) -- (v2) -- (v03);
\draw (v04) -- (v2);
\end{scope}

\end{tikzpicture}

%% file: gyro_patch_full.tikz
\begin{tikzpicture}[scale=2.8]
\tikzset{type0/.style={shape=circle, draw=red, scale=0.4, fill=white}}
\tikzset{type1/.style={shape=circle, draw=black!30!green, scale=0.4, fill=white}}
\tikzset{type2/.style={shape=circle, draw=black, scale=0.4, fill=white}}
\tikzset{t0/.style={draw=red, thin}}
\tikzset{t1/.style={draw=black!40!green, thin}}
\tikzset{t2/.style={thick}}

\node[type0,   label={[label distance=-0.02cm]90:{$v_2$}}] (v) at (0, 0) {};

%vertices

\node[type1] (e14) at ({cos(50+4*60)*cos(15)/cos(20)}, {sin(50+4*60)*cos(15)/cos(20)}) {};
\node[type0] (v14) at ({cos(40+4*60)*cos(15)/cos(10)}, {sin(40+4*60)*cos(15)/cos(10)}) {};%middenlinks
\node[type1, label={[label distance=-0.02cm]-90:{$v_1$}}] (e24) at ({cos(30+4*60)*cos(15)}, {sin(30+4*60)*cos(15)}) {};
\node[type0] (v24) at ({cos(20+4*60)*cos(15)/cos(10)}, {sin(20+4*60)*cos(15)/cos(10)}) {};%middenrechts
\node[type1] (e34) at ({cos(10+4*60)*cos(15)/cos(20)}, {sin(10+4*60)*cos(15)/cos(20)}) {};
\node[type0,  label={[label distance=-0.02cm]120:{$v_{0,L}$}}] (v34) at ({cos(4*60)*(cos(15)/cos(30))}, {sin(4*60)*(cos(15)/cos(30))}) {};%rechts
\node[type1] (e44) at ({cos(40+4*60)*cos(15)/cos(10)/2}, {sin(40+4*60)*cos(15)/cos(10)/2}) {};
\node[type2] (f4) at ({cos(4*60)*(cos(15)/cos(30))/2}, {sin(4*60)*(cos(15)/cos(30))/2}) {};

\node[type2] (f5) at ({cos(5*60)*(cos(15)/cos(30))/2}, {sin(5*60)*(cos(15)/cos(30))/2}) {};
\node[type0,  label={[label distance=-0.02cm]60:{$v_{0,R}$}}] (v35) at ({cos(5*60)*(cos(15)/cos(30))}, {sin(300)*(cos(15)/cos(30))}) {};%rechts
%edges

\draw[t2] (v35) -- (e14) -- (v14) --(e24) ;
\draw[t2] (e24) -- (v24) -- (e34) -- (v34) ;
\draw[t2] (v14) -- (e44)--(v);
\draw[t1] (v14) -- (f4) -- (v24);
\draw[t1] (v) -- (f4) -- (v34);
\draw[t0] (e24) -- (f4) -- (e34);
\draw[t0] (f4) -- (e44);
\draw[t0] (e14) -- (f5) -- (e44);
\draw[t1] (v14) -- (f5);

\draw[t1] (v) -- (f5) -- (v35);

\node[draw=white] (k) at (0.5,-1.1) {};

\end{tikzpicture}

%% file: gyro_applied.tikz
\begin{tikzpicture}[scale=1.8]
\tikzset{type0/.style={shape=circle, draw=red, scale=0.4, fill=white}}
\tikzset{type1/.style={shape=circle, draw=black!30!green, scale=0.4, fill=white}}
\tikzset{type2/.style={shape=circle, draw=black, scale=0.4, fill=white}}
\tikzset{t0/.style={draw=red, thin}}
\tikzset{t1/.style={draw=black!40!green, thin}}
\tikzset{t2/.style={thick}}

\fill[blue!10!white] (0,0) -- ({-cos(15)*tan(30)}, {cos(15)}) -- (0, {cos(15)}) -- ({cos(15)*tan(30)}, {cos(15)}) -- cycle;

\node[type0] (v) at (0, 0) {};

%vertices
\foreach \i in {0,...,5}
{
\node[type1] (e1\i) at ({cos(50+\i*60)*cos(15)/cos(20)}, {sin(50+\i*60)*cos(15)/cos(20)}) {};
\node[type0] (v1\i) at ({cos(40+\i*60)*cos(15)/cos(10)}, {sin(40+\i*60)*cos(15)/cos(10)}) {};
\node[type1] (e2\i) at ({cos(30+\i*60)*cos(15)}, {sin(30+\i*60)*cos(15)}) {};
\node[type0] (v2\i) at ({cos(20+\i*60)*cos(15)/cos(10)}, {sin(20+\i*60)*cos(15)/cos(10)}) {};
\node[type1] (e3\i) at ({cos(10+\i*60)*cos(15)/cos(20)}, {sin(10+\i*60)*cos(15)/cos(20)}) {};
\node[type0] (v3\i) at ({cos(\i*60)*(cos(15)/cos(30))}, {sin(\i*60)*(cos(15)/cos(30))}) {};
\node[type1] (e4\i) at ({cos(40+\i*60)*cos(15)/cos(10)/2}, {sin(40+\i*60)*cos(15)/cos(10)/2}) {};
\node[type2] (f\i) at ({cos(\i*60)*(cos(15)/cos(30))/2}, {sin(\i*60)*(cos(15)/cos(30))/2}) {};
}

%edges
\foreach \j / \i in {1/0,4/3,5/4,0/5}
{
\draw[t2] (v3\j) -- (e1\i) -- (v1\i) --(e2\i) -- (v2\i) -- (e3\i) -- (v3\i) ;
\draw[t2] (v1\i) -- (e4\i)--(v);
\draw[t1] (v1\i) -- (f\i) -- (v2\i);
\draw[t1] (v) -- (f\i) -- (v3\i);
\draw[t0] (e2\i) -- (f\i) -- (e3\i);
\draw[t0] (f\i) -- (e4\i);
\draw[t0] (e1\i) -- (f\j) -- (e4\i);
\draw[t1] (v1\i) -- (f\j);
}

%bluechambers
\foreach \j / \i in {2/1}
{
\begin{scope}[very thick]
\draw[t2,very thick] (v3\j) -- (e1\i) -- (v1\i) --(e2\i) -- (v2\i) -- (e3\i) -- (v3\i) ;
\draw[t2,very thick] (v1\i) -- (e4\i)--(v);
\draw[t1,very thick] (v1\i) -- (f\i) -- (v2\i);
\draw[t1,very thick] (v) -- (f\i) -- (v3\i);
\draw[t0,very thick] (e2\i) -- (f\i) -- (e3\i);
\draw[t0,very thick] (f\i) -- (e4\i);
\draw[t0,very thick] (e1\i) -- (f\j) -- (e4\i);
\draw[t1,very thick] (v1\i) -- (f\j);
\end{scope}
}

\foreach \j / \i in {3/2}
{
\draw[t2] (v3\j) -- (e1\i) -- (v1\i) --(e2\i) -- (v2\i) -- (e3\i) -- (v3\i) ;
\draw[t2] (v1\i) -- (e4\i)--(v);
\draw[t1] (v1\i) -- (f\i) -- (v2\i);
\draw[t1, very thick] (v) -- (f2) -- (v32);
\draw[t0] (e2\i) -- (f\i) -- (e3\i);
\draw[t0] (f\i) -- (e4\i);
\draw[t0] (e1\i) -- (f\j) -- (e4\i);
\draw[t1] (v1\i) -- (f\j);
}

\end{tikzpicture}

%% file: deco_truncation.tikz
\begin{tikzpicture}[scale = 3.7]
\tikzset{type0/.style={shape=circle, draw=red, scale=0.4, fill=white}}
\tikzset{type1/.style={shape=circle, draw=black!30!green, scale=0.4, fill=white}}
\tikzset{type2/.style={shape=circle, draw=black, scale=0.4, fill=white}}
\tikzset{noNode/.style={draw=none}}
\tikzset{0edge/.style={draw=red, thin}}
\tikzset{1edge/.style={draw=black!40!green, thin}}
\tikzset{2edge/.style={very thick}}

\node[type2] (v2) at (0, 1) {};
\node[type1] (v1) at (0, 0) {};
\node[type2] (v0) at ({tan(30)}, 0) {};
\node[type1] (e) at ({sin(30)},{1 - cos(30)}) {};
\node[type0] (v) at ({tan(15)}, {0})  {};

\begin{scope}[0edge]
\draw (v1) edge (v2);
\draw (e) edge (v2);
\draw (v0) edge (e);
\end{scope}

\begin{scope}[1edge]
\draw (v) edge (v0);
\draw (v) edge (v2);
\end{scope}

\begin{scope}[2edge]
\draw (v1) edge (v);
\draw (e) edge (v);
\end{scope}

\def\dist{0.04}
\def\halfwidth{0.02}
\def\angle{30}
\def\loose{6}
\newcommand{\drawroundarrow}[6]{%color, type, x, y, angle, side
	\draw[#1] ({#3 - \halfwidth*cos(#5) + cos(#5 + #6*90)*\dist},{#4 - \halfwidth*sin(#5) + sin(#5 + #6*90)*\dist}) 
	to[in={#5 + #6* 90 - #6*\angle} , out=#5 +#6*90 + #6*\angle, looseness=\loose] node[anchor=#5 - #6*90] {$\sigma_#2$} 
	++({2*\halfwidth*cos(#5)},{2*\halfwidth*sin(#5)});
}
\begin{scope}[->]
\drawroundarrow{white}{1}{3*tan(30)/4}{0}{0}{-1}
\end{scope}

\node[label=150:{$v_2$}] at (0, 1) {};
\node[label=190:$v_1$] at (0, 0) {};
\node[label={-10:$v_0$}] at ({tan(30)}, 0) {};
\end{tikzpicture}

%% file: truncation_DD.tikz
%\usetikzlibrary{matrix}
\begin{tikzpicture}[scale = 3.7]

%\tikzset{type0/.style={shape=circle, draw=red, scale=0.4, fill=white}}
%\tikzset{type1/.style={shape=circle, draw=black!30!green, scale=0.4, fill=white}}
\tikzset{type/.style={shape=circle, draw=black, scale=0.4, fill=white}}
\tikzset{noNode/.style={draw=none}}
\tikzset{0edge/.style={draw=red, thin}}
\tikzset{1edge/.style={draw=black!40!green, thin}}
\tikzset{2edge/.style={very thick}}

\node[type] (v2) at (0, 1) {};
\node[type] (v1) at (0, 0) {};
\node[type] (v0) at ({tan(30)}, 0) {};
\node[type] (e) at ({sin(30)},{1 - cos(30)}) {};
\node[type] (v) at ({tan(15)}, {0})  {};

\begin{scope}
\draw (v1) edge (v2);
\draw (e) edge (v2);
\draw (v0) edge (e);
\end{scope}

\begin{scope}
\draw (v) edge (v0);
\draw (v) edge (v2);
\end{scope}

\begin{scope}
\draw (v1) edge (v);
\draw (e) edge (v);
\end{scope}

\node at (0.1, 0.2) {\small$C_1$};
\node at (0.32, 0.25) {\small$C_2$};
\node at (0.48, 0.05) {\small$C_3$};

\node[label=150:{\color{white}{$v_2$}}] at (0, 1) {};

\def\dist{0.04}
\def\halfwidth{0.02}
\def\angle{30}
\def\loose{6}
\newcommand{\drawroundarrow}[6]{%color, type, x, y, angle, side
	\draw[#1] ({#3 - \halfwidth*cos(#5) + cos(#5 + #6*90)*\dist},{#4 - \halfwidth*sin(#5) + sin(#5 + #6*90)*\dist}) 
	to[in={#5 + #6* 90 - #6*\angle} , out=#5 +#6*90 + #6*\angle, looseness=\loose] node[anchor=#5 - #6*90] {$\sigma_#2$} 
	++({2*\halfwidth*cos(#5)},{2*\halfwidth*sin(#5)});
}
\begin{scope}[->]
\drawroundarrow{red}{0}{0}{1/2}{90}{1}
\drawroundarrow{red}{0}{sin(30)/2}{1-cos(30)/2}{120}{-1}
\drawroundarrow{red}{0}{(tan(30) +sin(30))/2}{(1-cos(30))/2}{120}{-1}
\drawroundarrow{black}{2}{tan(30)/4}{0}{0}{-1}
\drawroundarrow{black!40!green}{1}{3*tan(30)/4}{0}{0}{-1}
\end{scope}

\def\halflength{0.04}
\newcommand{\drawarrow}[4]{%color, x, y, angle
    \draw[#1] ({#2 - \halflength*cos(#4)},{ #3- \halflength*sin(#4)}) -- ++({2*\halflength*cos(#4)},{2*\halflength*sin(#4)});
}
\begin{scope}[<->]
\drawarrow{black!40!green}{tan(15)/2}{1/2}{15}
\drawarrow{black}{(sin(30)+tan(15))/2}{(1-cos(30))/2}{-60}
\end{scope}

\matrix (mat) [matrix of math nodes] at (0.9,0.8)
  {
       & m_{01}   & m_{12} \\ \hline
    C_1  & 12   & 3 \\
    C_2 & 12  & 3 \\
    C_3 & 3 & 3\\
  } ;
  
 \draw (mat-1-2.north west) -- (mat-4-1.south east);

\end{tikzpicture}

%% file: gyro_DD.tikz
%\usetikzlibrary{matrix}

\begin{tikzpicture}[scale=2.5]
\tikzset{type0/.style={shape=circle, draw=red, scale=0.4, fill=white}}
\tikzset{type1/.style={shape=circle, draw=black!30!green, scale=0.4, fill=white}}
\tikzset{type2/.style={shape=circle, draw=black, scale=0.4, fill=white}}
\tikzset{type/.style={shape=circle, draw=black, scale=0.4, fill=white}}
\tikzset{t0/.style={draw=red, thin}}
\tikzset{t1/.style={draw=black!40!green, thin}}
\tikzset{t2/.style={thick}}
\tikzset{t/.style={black}}

\clip (-1.15,2.1) -- (2.2,2.1) -- (2.2,-0.35) -- (-1.15,-0.35);

\node[type] (v1) at (0,0) {};
\node[type] (v2) at (0,{cot(30)}) {};
\node[type] (v0) at (1,0) {};
\node[type] (v) at (1/3,0) {};
\node[type] (e) at (2/3,0) {};
\node[type] (v0') at (-1,0) {};
\node[type] (v') at (-1/3,0) {};
\node[type] (e') at (-2/3,0) {};

\node[type] (f) at (1/2,{cot(30)/2}) {};
\node[type] (f') at (-1/2,{cot(30)/2}) {};

\node[type] (E) at (1/6, {cot(30)/2}) {};

\begin{scope}[t]
\draw (f) -- (E) -- (f') -- (v1)
	     (f) -- (e)
	     (f') -- (e');
%\end{scope}

%\begin{scope}[t]
\draw (v0) -- (f) -- (v2) -- (f') -- (v0')
	     (f) -- (v)
	     (f') -- (v')
	     (f') -- (v);
%\end{scope}

%\begin{scope}[t]
\draw (v2) -- (E) -- (v) -- (e) -- (v0)
	     (v0') -- (e') -- (v') -- (v1) -- (v);
\end{scope}

\def\halflength{0.05}
\newcommand{\drawarrow}[4]{%color, x, y, angle
    \draw[#1,<->] ({#2 - \halflength*cos(#4)},{ #3- \halflength*sin(#4)}) -- ++({2*\halflength*cos(#4)},{2*\halflength*sin(#4)});
}

\drawarrow{black}{1/12}{3*cot(30)/4}{11}
\drawarrow{black}{3/12}{cot(30)/4}{11}
\drawarrow{red}{7/12}{cot(30)/4}{11}
\drawarrow{black!40!green}{-5/12}{cot(30)/4}{11}
\drawarrow{black!40!green}{5/12}{cot(30)/4}{169}
\drawarrow{red}{-7/12}{cot(30)/4}{169}
\drawarrow{black!40!green}{-1/12}{cot(30)/4}{44}
\drawarrow{red}{-3/12}{cot(30)/4}{30}
\drawarrow{red}{-1/6}{cot(30)/2}{90}
\drawarrow{red}{1/3}{cot(30)/2}{90}

\newcommand{\drawouterarrow}[5]{%color, x, y, angle, looseness
    \draw[#1,<->] ({#2},{#3}) to[in=#4, out=180-#4, looseness=#5] ({-#2},{#3});
}

\drawouterarrow{black!40!green}{1/4}{3*cot(30)/4}{130}{5}
\drawouterarrow{black!40!green}{3/4}{cot(30)/4}{130}{4.5}
\drawouterarrow{black}{1/6}{0}{-90}{1}
\drawouterarrow{black}{3/6}{0}{-90}{0.7}
\drawouterarrow{black}{5/6}{0}{-90}{0.7}

\node at (-0.09, 1.07) {$C_1$};
\node at (0.25, 1.07) {$C_2$};
\node at (-0, 0.65) {$C_3$};
\node at (0.34, 0.65) {$C_4$};
\node at (-0.76, 0.15) {$C_5$};
\node at (-0.5, 0.15) {$C_6$};
\node at (-0.22, 0.15) {$C_7$};
\node at (0.05, 0.15) {$C_8$};
\node at (0.5, 0.15) {$C_9$};
\node at (0.78, 0.15) {$C_{10}$};

\matrix (mat) [matrix of math nodes] at (1.7,0.9)
  {
       & m_{01}   & m_{12} \\ \hline
    C_1& 5 & 6  \\
    C_2& 5 & 6  \\
    C_3& 5 & 3  \\
    C_4& 5 & 3  \\
    C_5& 5 & 3  \\
    C_6& 5 & 3  \\
    C_7& 5 & 3  \\
    C_8& 5 & 3  \\
    C_9& 5 & 3  \\
    C_{10}& 5 & 3  \\
  } ;
  
 \draw (mat-1-2.north west) -- (mat-11-1.south east);

\end{tikzpicture}

%% file: flips_situation1.tikz
\begin{tikzpicture}[scale=1.5]
\tikzset{every node/.style={shape=circle, draw=black, scale=0.4, fill=black}}
\tikzset{P/.style={draw=black, very thick}}
\tikzset{F/.style={draw=red, ultra thick}}
\tikzset{F'/.style={draw=blue, ultra thick}}

\coordinate (xcoo) at (0,0);
\coordinate (ycoo) at (3,0);
\coordinate (zcoo) at (0.4,0.8);
\coordinate (acoo) at (0.8,0);
\coordinate (bcoo) at (2.2,0);

%for position
\draw[draw=none]  (3.5,0) to[in=-70, out=-90, looseness=1.2] (bcoo);

%fill
\draw[F', fill=blue!10!white] (acoo) to[in=90, out=90, looseness=1.5] (bcoo) ;
\node[draw=none, fill=none] at (1.5, 0.8) {\huge$Q'_{a,b}$};

\def\length{0.02}
\begin{scope}[F]
\draw (0,-\length) -- (0.8,-\length)  node[midway, above, draw=none, fill=none] {\huge$Q_{x,a}$};
\end{scope}

\begin{scope}[F']
\draw(0,\length) -- (0.8,\length);
\end{scope}

%nodes
\node[label=180:\Huge{$x$}] (x) at (xcoo) {};
\node[label={[label distance=0.2cm]-90:\Huge{$a$}}] (a) at (acoo) {};
\node[label=-90:\Huge{$b$}] (b) at (bcoo) {};
\node[label=0:\Huge{$y$}] (y) at (ycoo) {};
\node[label=left:\Huge{$z$}] (v2) at (zcoo) {};

%edges
\begin{scope}[P]
\draw (b) -- (y) node[midway, above, draw=none, fill=none] {\huge$Q_{b,y}$};
\end{scope}

\begin{scope}[F]
\draw (a)  -- (b)  node[midway, below, draw=none, fill=none] {\huge$Q_{a,b}$};
\end{scope}

\end{tikzpicture}

%% file: flips_situation4.tikz
\begin{tikzpicture}[scale=1.5]
\tikzset{every node/.style={shape=circle, draw=black, scale=0.4, fill=black}}
\tikzset{noNode/.style={draw=none}}
\tikzset{P/.style={draw=black, very thick}}
\tikzset{F/.style={draw=red, ultra thick}}
\tikzset{F'/.style={draw=blue, ultra thick}}

\coordinate (xcoo) at (0,0);
\coordinate (ycoo) at (3,0);
\coordinate (zcoo) at (0.0,1.0);
\coordinate (acoo) at (0.8,0);
\coordinate (bcoo) at (2.2,0);

%for position on page, invisible
\draw[draw=none]  (3.5,0) to[in=-70, out=-90, looseness=1.2] (bcoo);

%fill
\draw[F', fill=blue!10!white] (bcoo) to[in=90, out=110, looseness=1.1] (-0.5,0) to[in=-110, out=-90, looseness=1.2] (acoo);
\node[draw=none, fill=none] at (0.8, 1.05) {\huge$Q'_{a,b}$};

\def\length{0.02}
\begin{scope}[F']
%\draw (v1)  edge[ultra thick] (b);
\draw (0.8,-\length) -- (0,-\length);
\end{scope}

\begin{scope}[F]
%\draw (v1)  edge[ultra thick] (b);
\draw(0.8,\length) -- (0,\length)  node[midway, above, draw=none, fill=none] {\huge$Q_{a,x}$};
\end{scope}

%nodes
\node[label=left:\Huge{$x$}] (x) at (xcoo) {};
\node[label={[label distance=0.2cm]-80:\Huge{$a$}}] (a) at (acoo) {};
\node[label=-90:\Huge{$b$}] (b) at (bcoo) {};
\node[label=0:\Huge{$y$}] (y) at (ycoo) {};
\node[label=left:\Huge{$z$}] (z) at (zcoo) {};

\begin{scope}[P]
\draw (b) -- (y) node[midway, above, draw=none, fill=none] {\huge$Q_{b,y}$}; 
\end{scope}

\begin{scope}[F]
\draw (a)  -- (b) node[midway, below, draw=none, fill=none] {\huge$Q_{a,b}$};
\end{scope}

%\node[draw=none, fill=none] (_) at (0.05,1) {\Huge$P^2$};

\end{tikzpicture}

%% file: flips_situation2.tikz
\begin{tikzpicture}[scale=1.5]
\tikzset{every node/.style={shape=circle, draw=black, scale=0.4, fill=black}}
\tikzset{P/.style={draw=black, very thick}}
\tikzset{F/.style={draw=red, ultra thick}}
\tikzset{F'/.style={draw=blue, ultra thick}}

\coordinate (xcoo) at (0,0);
\coordinate (ycoo) at (3,0);
\coordinate (zcoo) at (0.4,0.8);
\coordinate (acoo) at (0.8,0);
\coordinate (bcoo) at (2.2,0);

%fill
\draw[F', fill=blue!10!white] (acoo) to[in=90, out=70, looseness=1.1] (3.5,0) to[in=-70, out=-90, looseness=1.2] (bcoo);
\node[draw=none, fill=none] at (2.2, 1.05) {\huge$Q'_{a,b}$};

\def\length{0.02}
\begin{scope}[F']
\draw (0,\length) -- (0.8,\length)  node[midway, above, draw=none, fill=none] {\huge$Q_{x,a}$};
\draw (2.2,-\length) -- (3,-\length);
\end{scope}

\begin{scope}[F]
\draw(2.2,\length) -- (3,\length) node[midway, above, draw=none, fill=none] {\huge$Q_{b,y}$};
\draw (0,-\length) -- (0.8,-\length);
\end{scope}

%nodes
\node[label=left:\Huge{$x$}] (x) at (xcoo) {};
\node[label={[label distance=0.3cm]-80:\Huge{$a$}}] (a) at (acoo) {};
\node[label=-100:\Huge{$b$}] (b) at (bcoo) {};
\node[label=10:\Huge{$y$}] (y) at (ycoo) {};
\node[label=left:\Huge{$z$}] (z) at (zcoo) {};

\begin{scope}[P]
%\draw (x) -- (a); 
\end{scope}

\begin{scope}[F]
\draw (a)  -- (b) node[midway, below, draw=none, fill=none] {\huge$Q_{a,b}$};
\end{scope}

\end{tikzpicture}

%% file: flips_situation3.tikz
\begin{tikzpicture}[scale=1.5]
\tikzset{every node/.style={shape=circle, draw=black, scale=0.4, fill=black}}
\tikzset{noNode/.style={draw=none, scale=0.01}}
\tikzset{P/.style={draw=black, very thick}}
\tikzset{F/.style={draw=red, ultra thick}}
\tikzset{F'/.style={draw=blue, ultra thick}}

\coordinate (xcoo) at (0,0);
\coordinate (ycoo) at (3,0);
\coordinate (zcoo) at (0.0,1.0);
\coordinate (acoo) at (0.8,0);
\coordinate (bcoo) at (2.2,0);

\draw [F', fill=blue!10!white] (2.2,0) to[out=-70,in=-90, looseness=1.2](3.5,0) to[out=90,in=90, looseness=0.9] (-0.5,0) to[out=-90,in=-110, looseness=1.2] (0.8,0);
\node[draw=none, fill=none] at (1.5, 1.25) {\huge$Q'_{a,b}$};

\def\length{0.02}
\begin{scope}[F']
\draw (0,-\length) -- (0.8,-\length);
\draw (2.2,-\length) -- (3,-\length);
\end{scope}

\begin{scope}[F]
\draw(0,\length) -- (0.8,\length) node[midway, above, draw=none, fill=none] {\huge$Q_{x,a}$};
\draw(2.2,\length) -- (3,\length) node[midway, above, draw=none, fill=none] {\huge$Q_{b,y}$};
\end{scope}

\node[label=left:\Huge{$x$}] (x) at (xcoo) {};
\node[label={50:\Huge{$a$}}] (a) at (acoo) {};
\node[label=120:\Huge{$b$}] (b) at (bcoo) {};
\node[label=0:\Huge{$y$}] (y) at (ycoo) {};
\node[label=left:\Huge{$z$}] (z) at (zcoo) {};

\begin{scope}[P]

\end{scope}

\begin{scope}[F]
\draw (a)  -- (b) node[midway, below, draw=none, fill=none] {\huge$Q_{a,b}$};
\end{scope}

\end{tikzpicture}

%% file: alpha_1.tikz
\usetikzlibrary{decorations.markings}
\usetikzlibrary{arrows.meta}
\begin{tikzpicture}
[scale=0.8, decoration={markings, 
    mark= at position 0.5 with {\arrow{stealth}}}
] 

\tikzset{normal/.style={shape=circle, draw=black, scale=0.4, fill=black}}
\tikzset{noNode/.style={shape=circle, draw=black, scale=0.05, fill=black}}

%-------------------------------------------------------
\node[noNode] (a) at (-2.5,-2.5) {};
\node[normal] (b) at (-1, -1) {};
\node[normal] (c) at (1,1) {};
\node[noNode] (d) at (2.5,2.5) {};
\node[normal] (e) at (1.5,-1) {};

\begin{scope}[very thick] 
    \draw[postaction={decorate}] (a)--(b);
    \draw[postaction={decorate}] (b)--(e);
    \draw[postaction={decorate}] (e)--(c);
    \draw[postaction={decorate}] (c)--(d);
\end{scope}

\begin{scope}[very thick, dashed]
\draw (b) -- (c);
\end{scope}

\node at (0.5,-0.3) {\Large$C$};

%-------------------------------------------------------
\begin{scope}[shift={(10,0)}]

\node[noNode] (a) at (-2.5,-2.5) {};
\node[normal] (b) at (-1, -1) {};
\node[normal] (c) at (1,1) {};
\node[noNode] (d) at (2.5,2.5) {};
\node[normal] (e) at (1.5,-1) {};

\draw[very thick, postaction={decorate}] (a)--(b);
\draw[very thick, postaction={decorate}] (b)--(c);
\draw[very thick, postaction={decorate}] (c)--(d);

\draw[dashed, very thick] (b) --(e) -- (c);
\node at (0.5,-0.3) {\Large$C$};

\end{scope}
%---------------------------------------------------------
\begin{scope}[shift={(0,-6)}]
\node[noNode] (a) at (-2.5,-2.5) {};
\node[normal] (b) at (-1, -1) {};
\node[normal] (c) at (1,1) {};
\node[noNode] (d) at (2.5,2.5) {};
\node[normal] (e) at (1.5,-1) {};

\draw[very thick, postaction={decorate}] (b)--(a);
\draw[very thick, postaction={decorate}] (e)--(b);
\draw[very thick, postaction={decorate}] (c)--(e);
\draw[very thick, postaction={decorate}] (d)--(c);

\draw[very thick, dashed] (b) --(c);

\node at (0.5,-0.3) {\Large$C$};
\end{scope}
%-----------------------------------------------------------
\begin{scope}[shift={(10,-6)}]

\node[noNode] (a) at (-2.5,-2.5) {};
\node[normal] (b) at (-1, -1) {};
\node[normal] (c) at (1,1) {};
\node[noNode] (d) at (2.5,2.5) {};
\node[normal] (e) at (1.5,-1) {};

\draw[very thick, postaction={decorate}] (b)--(a);
\draw[very thick, postaction={decorate}] (c)--(b);
\draw[very thick, postaction={decorate}] (d)--(c);

\draw[dashed, very thick] (b) --(e) -- (c);
\node at (0.5,-0.3) {\Large$C$};

\end{scope}
%-------------------------------------------------------------

\draw[-{>[scale=2]}] (4,1) -- (6,1) node[pos=0.5, label={[label distance=1mm]:$\alpha_{i+1}(C)=\alpha_i(C) - 1$}] {};
\draw[{<[scale=2]}-] (4,-1) -- (6,-1) node[pos=0.5, label={[label distance=1mm]below:$\alpha_{i+1}(C)=\alpha_i(C) + 1$}] {};

\begin{scope}[shift={(0,-6)}]
\draw[-{>[scale=2]}] (4,1) -- (6,1) node[pos=0.5, label={[label distance=1mm]:$\alpha_{i+1}(C)=\alpha_i(C) + 1$}] {};
\draw[{<[scale=2]}-] (4,-1) -- (6,-1) node[pos=0.5, label={[label distance=1mm]below:$\alpha_{i+1}(C)=\alpha_i(C) - 1$}] {};
\end{scope}

\end{tikzpicture}

%% file: flip_cycle.tikz
\begin{tikzpicture}
\tikzset{normal/.style={shape=circle, draw=black, scale=0.4, fill=black}}
\tikzset{type2/.style={shape=circle, draw=red, scale=0.5}}
\tikzset{noNode/.style={draw=none}}
\tikzset{0edge/.style={draw=red}}

\clip(-2.5,-2) rectangle (6.5,2.5);

\coordinate (A) at (-1,-0.5);
\coordinate (B) at (-0.5,1);
\coordinate (C) at (0.5,0);
\coordinate (D) at (1,-1);
\coordinate (E) at (1,1);

\fill[black!10!white] (A) -- (B) -- (C) -- cycle;

\node (_) at (-0.3,0.2) {$C$};

\begin{scope}[-{latex}]
\draw[thick, dashed] (B) -- (A);
\draw[blue,thick] (C) -- (A) node[midway, below, draw=none, fill=none,  label = {[label distance =-0.3cm]-10:$e_4$}] {}; 
\draw[red, thick] (B) -- (C) node[midway, , draw=none, fill=none, label = {[label distance =-0.2cm]80:$e_1$}] {};
\draw[red, thick] (C) -- (D) node[midway, , draw=none, fill=none, label = {[label distance =-0cm]0:$e_3$}] {};
\draw[blue, thick] (E) -- (C)  node[midway, , draw=none, fill=none, label = {[label distance =-0cm]-0:$e_2$}] {};
\end{scope}

\draw[red, thick] (B) to[out=160, in=220, looseness=4] (D);
\draw[blue, thick] (A) to[out=160, in=80, looseness=3] (E);

\begin{scope}[shift={(5,0)}]

\coordinate (A) at (-1,-0.5);
\coordinate (B) at (-0.5,1);
\coordinate (C) at (0.5,0);
\coordinate (D) at (1,-1);
\coordinate (E) at (1,1);

\fill[black!10!white] (A) -- (B) -- (C) -- cycle;

\node (_) at (-0.3,0.2) {$C$};

\draw[-{latex},thick, dashed] (B) -- (A);
\draw[-{latex},red,thick] (C) -- (A) node[midway, below, draw=none, fill=none, label = {[label distance =-0.3cm]-10:$e_4$}] {}; 
\draw[-{latex},blue, thick] (B) -- (C) node[midway, , draw=none, fill=none, label = {[label distance =-0.2cm]80:$e_1$}] {};
\draw[-{latex},red, thick] (D) -- (C) node[midway, , draw=none, fill=none, label = {[label distance =-0cm]0:$e_3$}] {};
\draw[-{latex},blue, thick] (C) -- (E)  node[midway, , draw=none, fill=none, label = {[label distance =-0cm]-0:$e_2$}] {};

\draw[red, thick] (A) to[out=190, in=-120, looseness=3] (D);
\draw[blue, thick] (B) to[out=160, in=100, looseness=4] (E);

\draw[draw=none] (B) to[out=160, in=220, looseness=5] (D);

\end{scope}

\end{tikzpicture}

%% file: 2_doublechambers.tikz
\begin{tikzpicture}[scale=1.3]
\tikzset{normal/.style={shape=circle, draw=black, scale=0.4}}
\tikzset{on_c/.style={shape=circle, draw=red, scale=0.4, fill=red}}
\tikzset{noNode/.style={draw=none}}
\tikzset{0edge/.style={draw=red}}

%\clip(-2.5,-2) rectangle (6.5,2.5);

\coordinate (A) at (-1,0);
\coordinate (A1) at (-1.5,0.5);
\coordinate (A2) at (-1.5,-0.5);
\coordinate (B) at (0,{tan(60)});
\coordinate (C) at (1,0);
\coordinate (D) at (0,-{tan(60)});
\coordinate (O) at (0,0);

\node[on_c] (a) at (A) {};
\node[on_c] (b) at (B) {};
\node[on_c] (c) at (C) {};
\node[on_c] (d) at (D) {};
\node[normal] (o) at (O) {};

\begin{scope}[red, very thick]
\draw (A1) --(a) -- (b) -- (c) -- (1.5,-0.5);
\draw (-0.5,-2) -- (d) -- (0.5,-2);
\end{scope}

\begin{scope}[ ]
\draw (A2) --(a) --(d)--(c) -- (1.5,0.5);
\draw (a) -- (o) -- (c);
\end{scope}

\begin{scope}[ dotted]
\draw (b) -- (o) -- (d);
\end{scope}

\begin{scope}[shift={(4,0)}]

\coordinate (A) at (-1,0);
\coordinate (A1) at (-1.5,0.5);
\coordinate (A2) at (-1.5,-0.5);
\coordinate (B) at (0,{tan(60)});
\coordinate (C) at (1,0);
\coordinate (D) at (0,-{tan(60)});
\coordinate (O) at (0,0);

\node[on_c] (a) at (A) {};
\node[on_c] (b) at (B) {};
\node[on_c] (c) at (C) {};
\node[on_c] (d) at (D) {};
\node[normal] (o) at (O) {};

\begin{scope}[red, very thick]
\draw (d) --(a) -- (b) -- (c) -- (1.5,-0.5);
\draw (-0.5,-2) -- (d) ;
\end{scope}

\begin{scope}[ ]
\draw (A2) --(a) -- (-1.5,0.5); 
\draw (d)--(c) -- (1.5,0.5);
\draw (a) -- (o) -- (c);
\draw (d) -- (0.5,-2);
\end{scope}

\begin{scope}[ dotted]
\draw (b) -- (o) -- (d);
\end{scope}

\end{scope}

\end{tikzpicture}

%% file: jordan_contradiction2.tikz
\begin{tikzpicture}[scale=3.5]
\def\Pcolor{black}
\def\Qcolor{red}
\def\vsize{0.3}

\tikzset{normal/.style={shape=circle, draw=black, scale=0.4, fill=white}}
\tikzset{P/.style={shape=circle, draw=\Pcolor, scale=\vsize, fill=\Pcolor}}
\tikzset{P'/.style={shape=circle, draw=\Pcolor!50!white, scale=\vsize, fill=\Pcolor!50!white}}

\tikzset{edgeP/.style={draw=\Pcolor}}
\tikzset{edgeP'/.style={draw=\Pcolor!50!white}}
\tikzset{edgeQ/.style={draw=\Qcolor}}
\tikzset{edgeQ'/.style={draw=\Qcolor!50!white, dashed}}

\node[normal] (v1) at (0, 0) {};
\node[normal] (v2) at (0, {sqrt(3)/4}) {};
\node[normal] (v2') at (0, {-sqrt(3)/4}) {};
\node[normal] (v0) at (-0.5, 0) {};
\node[normal] (v0') at (0.5, 0) {};

\node[P, label={[label distance=-0.05cm]80:\textcolor{\Pcolor}{$x$}}] (x) at (-0.35,0) {};
\node[P'] (x') at (0.35,0) {};%',label={[label distance=-0.07cm]90:\textcolor{\Pcolor!50!white}{\quad$x'$}}
\node[P,label={[label distance=-0.05cm]80:\textcolor{\Pcolor}{$y$}}] (y) at (0.15,0) {};
\node[P'] (y') at (-0.15,0) {};%,label={[label distance=-0.0cm]90:\textcolor{\Pcolor!50!white}{$y'$}}

\begin{scope}
\draw (v0) -- (v2) -- (v0') -- (v2') -- (v0) --(x) -- (y') -- (v1) -- (y) -- (x') -- (v0');
\end{scope}

\begin{scope}[thick]
\draw[edgeP] (x) edge[bend left=80, looseness=1] (y);
\draw[edgeP'] (x') edge[bend left=80, looseness=1] (y');
\draw[red, dashed] (x) edge[bend right=80, looseness=1] (y);

\end{scope}

\node[] at (-0.1,0.23) {\textcolor{\Pcolor}{$\mathbf{e_1}$}};
%\node[] at (0.1,-0.3) {\textcolor{\Pcolor!50!white}{$\mathbf{e_1'}$}};
\node[] at (-0.1,-0.23) {\textcolor{red}{$\mathbf{e_2}$}};

\end{tikzpicture}

%% file: jordan_curve_double_chamber.tikz
\begin{tikzpicture}[scale=2]
\tikzset{every label/.append style={font=\Huge}}
\tikzset{every node/.style={shape=circle, draw=black, scale=0.25}}
\tikzset{face/.style={->, thick, draw=red}}
\tikzset{red/.style={thick, draw=red}}

\node[] (v0) at (0,1) {};
\node[draw=none] (v0') at (0,1) {};

\node[draw=none, scale=3] at (-0.8,1.1) {$(a)$};

\foreach \j/\i in {0/1,1/2,2/3,3/4,4/5,5/6,6/7}{
\node (v\i) at ({tan(30)*((\i*0.125))}, 1-\i*0.125) {};
\node (v\i') at ({-tan(30)*((\i*0.125))}, 1-\i*0.125) {};
\draw[very thin] (v\i) -- (v\j);
\draw[very thin] (v\i') -- (v\j');
}

\draw[very thin] (v7) -- (v7');

\node[scale=0,draw=none, label=90:$t_0'{=}t_0$] () at (0,1) {};
\node[scale=1,draw=none, label=10:$t_1$] () at ({tan(30)*((1*0.125))}, 1-1*0.125) {};
\node[scale=1,draw=none, label=170:$t_1'$] () at ({-tan(30)*((1*0.125))}, 1-1*0.125) {};
\node[scale=1,draw=none, label=10:$t_i$] () at ({tan(30)*((3*0.125))}, 1-3*0.125) {};
\node[scale=1,draw=none, label=170:$t_i'$] () at ({-tan(30)*((3*0.125))}, 1-3*0.125) {};
\node[scale=1,draw=none, label=10:$t_j$] () at ({tan(30)*((4*0.125))}, 1-4*0.125) {};
\node[scale=1,draw=none, label=170:$t_j'$] () at ({-tan(30)*((4*0.125))}, 1-4*0.125) {};
\node[scale=1,draw=none, label=10:$t_k$] () at ({tan(30)*((6*0.125))}, 1-6*0.125) {};
\node[scale=1,draw=none, label=170:$t_k'$] () at ({-tan(30)*((6*0.125))}, 1-6*0.125) {};
\node[scale=1,draw=none, label=10:$t_s$] () at ({tan(30)*((7*0.125))}, 1-7*0.125) {};
\node[scale=1,draw=none, label=170:$t_s'$] () at ({-tan(30)*((7*0.125))}, 1-7*0.125) {};

\draw[thick] (v3) -- (v4');

\begin{scope}[shift={(2,0)}]
\node[] (v0) at (0,1) {};
\node[draw=none] (v0') at (0,1) {};

\node[draw=none, scale=3] at (-0.8,1.1) {$(b)$};

\foreach \j/\i in {0/1,1/2,2/3,3/4,4/5,5/6,6/7}{
\node (v\i) at ({tan(30)*((\i*0.125))}, 1-\i*0.125) {};
\node (v\i') at ({-tan(30)*((\i*0.125))}, 1-\i*0.125) {};
\draw[very thin] (v\i) -- (v\j);
\draw[very thin] (v\i') -- (v\j');
}

\draw[very thin] (v7) -- (v7');

\node[scale=0,draw=none, label=90:$t_0{=}t_0''$] () at (0,1) {};
\node[scale=1,draw=none, label=10:$t_1''$] () at ({tan(30)*((1*0.125))}, 1-1*0.125) {};
\node[scale=1,draw=none, label=170:$t_1$] () at ({-tan(30)*((1*0.125))}, 1-1*0.125) {};
\node[scale=1,draw=none, label=10:$t_i''$] () at ({tan(30)*((3*0.125))}, 1-3*0.125) {};
\node[scale=1,draw=none, label=170:$t_i$] () at ({-tan(30)*((3*0.125))}, 1-3*0.125) {};
\node[scale=1,draw=none, label=10:$t_j''$] () at ({tan(30)*((4*0.125))}, 1-4*0.125) {};
\node[scale=1,draw=none, label=170:$t_j$] () at ({-tan(30)*((4*0.125))}, 1-4*0.125) {};
\node[scale=1,draw=none, label=10:$t_k''$] () at ({tan(30)*((6*0.125))}, 1-6*0.125) {};
\node[scale=1,draw=none, label=170:$t_k$] () at ({-tan(30)*((6*0.125))}, 1-6*0.125) {};
\node[scale=1,draw=none, label=10:$t_s''$] () at ({tan(30)*((7*0.125))}, 1-7*0.125) {};
\node[scale=1,draw=none, label=170:$t_s$] () at ({-tan(30)*((7*0.125))}, 1-7*0.125) {};

\draw[thick] (v3) -- (v4');
\draw[thick] (v4) -- (v6');
\end{scope}

\end{tikzpicture}

%% file: twoedges2.tikz
\begin{tikzpicture}
\tikzset{vertex/.style={shape=circle, draw=black, scale=0.3, fill=black}}

\tikzset{noNode/.style={draw=none}}

%\tikzset{2edge/.style={thick}}

\node[vertex] (v1) at (0,0) {};

\node[vertex, label=$v_2$] (v2) at (0,2) {};
\node[vertex, label=-90:$v_0^3$] (v0) at ({2*tan(30)}, 0) {};
\node[vertex, label=-90:$v_0^4$] (v0') at ({-2*tan(30)}, 0) {};
\node[vertex, label=$v_0^2$] (v0r) at ({4*tan(30)}, 2) {};
\node[vertex, label=$v_0^1$] (v0l) at ({-4*tan(30)}, 2) {};
\node[vertex] (v1r) at ({3*tan(30)}, 1) {};
\node[vertex] (v1l) at ({-3*tan(30)}, 1) {};

\begin{scope}%[2edge]
%\draw  (v2) --(v0) -- (v2') --(v0') -- (v2);
\draw (v0) -- (v2) -- (v0');
\draw (v0l) -- (v2) -- (v0r);
\draw (v0r) -- (v1r) -- (v0) -- (v1) -- (v0') -- (v1l) -- (v0l);
\end{scope}

\node at (-3.2,2.5) {$(a)$};

\end{tikzpicture}

%% file: twoedges0.tikz
\begin{tikzpicture}
\tikzset{vertex/.style={shape=circle, draw=black, scale=0.3, fill=black}}

\tikzset{noNode/.style={draw=none}}

%\tikzset{2edge/.style={thick}}

\node[vertex, label=$v_0^2$] (v2) at (0,2) {};
\node[vertex, label=-90:$v_0^3$] (v0) at ({2*tan(30)}, 0) {};
\node[vertex, label=-90:$v_2^1$] (v0') at ({-2*tan(30)}, 0) {};
\node[vertex, label=$v_2^2$] (v0r) at ({4*tan(30)}, 2) {};
\node[vertex, label=$v_0^1$] (v0l) at ({-4*tan(30)}, 2) {};
\node[vertex] (v1r) at ({tan(30)}, 1) {};
\node[vertex] (v1l) at ({-2*tan(30)}, 2) {};

\begin{scope}%[2edge]
\draw (v0l) -- (v1l) -- (v2) -- (v1r) -- (v0);
\draw (v0)-- (v0') -- (v0l);
\draw (v0) --(v0r) --(v2) -- (v0') ;
\end{scope}

\node at (-3.2,2.5) {$(b)$};

\end{tikzpicture}

%% file: butterfly.tikz
\begin{tikzpicture}
\tikzset{vertex/.style={shape=circle, draw=black, scale=0.3, fill=black}}

\tikzset{noNode/.style={draw=none}}

%\tikzset{2edge/.style={thick}}

\node[draw=none] (_) at (0, {-4.82*tan(30)}) {};

\node[vertex, label=$e$] (v1) at (0,0) {};
\node[vertex] (v2') at (0,-2) {};
\node[vertex] (v2) at (0,2) {};
\node[vertex] (v0) at ({2*tan(30)}, 0) {};
\node[vertex] (v0') at ({-2*tan(30)}, 0) {};
\node[vertex] (v0r) at ({4*tan(30)}, 2) {};
\node[vertex] (v0l) at ({-4*tan(30)}, 2) {};
\node[vertex] (v0r') at ({4*tan(30)}, -2) {};
\node[vertex] (v0l') at ({-4*tan(30)}, -2) {};
\node[vertex] (v1r) at ({3*tan(30)}, 1) {};
\node[vertex] (v1l) at ({-3*tan(30)}, 1) {};
\node[vertex] (v1r') at ({3*tan(30)}, -1) {};
\node[vertex] (v1l') at ({-3*tan(30)}, -1) {};

\begin{scope}%[2edge]
\draw  (v2) --(v0) -- (v2') --(v0') -- (v2);
\draw (v0) -- (v1) -- (v0');
\draw (v0l) -- (v2) -- (v0r) -- (v1r) -- (v0) -- (v1r') -- (v0r') -- (v2') -- (v0l') -- (v1l') -- (v0') -- (v1l) -- (v0l);
\end{scope}

\end{tikzpicture}

%% file: butterfly_1closed.tikz
\begin{tikzpicture}
\tikzset{vertex/.style={shape=circle, draw=black, scale=0.3, fill=black}}

\tikzset{noNode/.style={draw=none}}

%\tikzset{2edge/.style={thick}}

\node[draw=none] (_) at (0, {-4.82*tan(30)}) {};

\node[vertex, label=$e$] (v1) at (0,0) {};
\node[vertex] (v2') at (0,-2) {};
\node[vertex] (v2) at (0,{(2)*tan(30)}) {};
\node[vertex] (v0) at ({2*tan(30)}, 0) {};
\node[vertex] (v0') at ({-2*tan(30)}, 0) {};
\node[vertex] (v0r) at (0, {4.82*tan(30)}) {};
%\node[vertex] (v0l) at ({-4*tan(30)}, 2) {};
\node[vertex] (v0r') at ({4*tan(30)}, -2) {};
\node[vertex] (v0l') at ({-4*tan(30)}, -2) {};
\node[vertex] (v1r) at ({-1.1*cos(157.5)},{(2)*tan(30) + 1.1*sin(157.5)}) {};
\node[vertex] (v1l) at ({1.1*cos(157.5)},{(2)*tan(30) + 1.1*sin(157.5)}) {};
\node[vertex] (v1r') at ({3*tan(30)}, -1) {};
\node[vertex] (v1l') at ({-3*tan(30)}, -1) {};

\begin{scope}%[2edge]
\draw  (v2) --(v0) -- (v2') --(v0') -- (v2);
\draw (v0) -- (v1) -- (v0');
\draw  (v0) -- (v1r') -- (v0r') -- (v2') -- (v0l') -- (v1l') -- (v0');
\draw (v1r) to[in=-35, out = 115] (v0r);
\draw (v1r) to[in=80, out = -70] (v0);
\draw (v1l) to[in=-145, out = 65] (v0r);
\draw (v1l) to[in=100, out = -110] (v0');
\draw (v0r) -- (v2);
\end{scope}

\end{tikzpicture}

%% file: butterfly_2closed.tikz
\begin{tikzpicture}
\tikzset{vertex/.style={shape=circle, draw=black, scale=0.3, fill=black}}

\tikzset{noNode/.style={draw=none}}

%\tikzset{2edge/.style={thick}}

\node[vertex, label=$e$] (v1) at (0,0) {};
\node[vertex] (v2') at (0,{-(2)*tan(30)}) {};
\node[vertex] (v2) at (0,{(2)*tan(30)}) {};
\node[vertex] (v0) at ({2*tan(30)}, 0) {};
\node[vertex] (v0') at ({-2*tan(30)}, 0) {};
\node[vertex] (v0r) at (0, {4.82*tan(30)}) {};
%\node[vertex] (v0l) at ({-4*tan(30)}, 2) {};
\node[vertex] (v0r') at (0, {-4.82*tan(30)}) {};
%\node[vertex] (v0l') at ({-4*tan(30)}, -2) {};
\node[vertex] (v1r) at ({-1.1*cos(157.5)},{(2)*tan(30) + 1.1*sin(157.5)}) {};
\node[vertex] (v1l) at ({1.1*cos(157.5)},{(2)*tan(30) + 1.1*sin(157.5)}) {};
\node[vertex] (v1r') at ({-1.1*cos(157.5)},{-(2)*tan(30) - 1.1*sin(157.5)}) {};
\node[vertex] (v1l') at ({1.1*cos(157.5)},{(-2)*tan(30) - 1.1*sin(157.5)}) {};

\begin{scope}%[2edge]
\draw  (v2) --(v0) -- (v2') --(v0') -- (v2);
\draw (v0) -- (v1) -- (v0');
%\draw  (v0) -- (v1r') -- (v0r') -- (v2') -- (v0l') -- (v1l') -- (v0');
\draw (v1r) to[in=-35, out = 115] (v0r);
\draw (v1r) to[in=80, out = -70] (v0);
\draw (v1l) to[in=-145, out = 65] (v0r);
\draw (v1l) to[in=100, out = -110] (v0');
\draw (v0r) -- (v2);

\draw (v1r') to[in=35, out = -115] (v0r');
\draw (v1r') to[in=-80, out = 70] (v0);
\draw (v1l') to[in=145, out = -65] (v0r');
\draw (v1l') to[in=-100, out = 110] (v0');
\draw (v0r') -- (v2');
\end{scope}

\end{tikzpicture}

%% file: lopsp.bbl
\begin{thebibliography}{10}

\bibitem{dualconnectivity}
D.~Bokal, G.~Brinkmann, and C.~{T.~Zamfirescu}.
\newblock The connectivity of the dual.
\newblock arXiv:1812.08510, to appear in Journal of Graph Theory.

\bibitem{gocox}
G.~Brinkmann, P.~Goetschalckx, and S.~Schein.
\newblock Comparing the constructions of {G}oldberg, {F}uller, {C}aspar, {K}lug
  and {C}oxeter, and a general approach to local symmetry-preserving
  operations.
\newblock {\em Proceedings of the Royal Society A-Mathematical Physical and
  Engineering Sciences}, 473(2206), 2017.

\bibitem{CK62}
D.L.D. Caspar and A.~Klug.
\newblock Physical principles in the construction of regular viruses.
\newblock In {\em Cold Spring Harb Symp Quant Biol.}, volume~27, pages 1--24,
  1962.

\bibitem{olaf_ddsymbols}
O.~{Delgado~Friedrichs}.
\newblock Data structures and algorithms for tilings {I}.
\newblock {\em Theoretical Computer Science}, 303(2--3):431--445, 2003.

\bibitem{heidi_thesis}
H.~Van den Camp.
\newblock The effect of local symmetry-preserving operations on the
  connectivity of embedded graphs.
\newblock Master's thesis, Universiteit Gent, 2020.

\bibitem{DressHuson87}
A.W.M. Dress and D.H. Huson.
\newblock On tilings of the plane.
\newblock {\em Geometriae Dedicata}, 24:295--310, 1987.

\bibitem{goetschalckx2020generation}
P.~Goetschalckx, K.~Coolsaet, and N.~Van Cleemput.
\newblock Generation of local symmetry-preserving operations on polyhedra.
\newblock {\em Ars Mathematica Contemporanea}, 18:223--239, 2020.

\bibitem{lopsp2020}
P.~Goetschalckx, K.~Coolsaet, and N.~Van Cleemput.
\newblock Local orientation-preserving symmetry preserving operations on
  polyhedra.
\newblock {\em Discrete Mathematics}, 344(1), 2020.
\newblock Art.No. 112156.

\bibitem{moh}
B.~Mohar.
\newblock Face-width of embedded graphs.
\newblock {\em Mathematica Slovaca}, 47(1):35--63, 1997.

\bibitem{randic2000bridges}
T.~Pisanski and M.~Randi{\'c}.
\newblock Bridges between geometry and graph theory.
\newblock In C.A. Gorini, editor, {\em Geometry at Work}. Mathematical
  Association of America, 2000.

\end{thebibliography}
